\documentclass[a4paper, 10pt, parskip=half]{scrartcl}

\usepackage[utf8]{inputenc}
\usepackage[T1]{fontenc}
\usepackage{lmodern}
\usepackage{csquotes}
\usepackage{amsmath}
\usepackage{amssymb}
\usepackage{amsthm}
\usepackage{amsfonts}
\usepackage{dsfont}
\usepackage{mathtools}
\usepackage{algorithm}
\usepackage{algorithmic}
\usepackage{marginnote}
\usepackage[dvipsnames]{xcolor}
\usepackage{hyperref}
\hypersetup{
  colorlinks=true,
  linkcolor=black,
  citecolor=black,
  urlcolor=blue,
  pdftitle={},
  pdfauthor={},
  pdfkeywords={},
  bookmarksopen=true,
}

\usepackage{authblk}
\usepackage{float}
\usepackage{subcaption}

\usepackage{xcolor}

\newcommand{\ML}{\mathbb{M}_{\mathrm L}}

\ifdefined\labelenumi%
  \renewcommand{\labelenumi}{(\roman{enumi})}
  
\fi

\usepackage{newunicodechar}
\newunicodechar{∇}{\nabla}
\newunicodechar{Δ}{\Delta}
\newunicodechar{α}{\alpha}
\newunicodechar{γ}{\gamma}
\newunicodechar{μ}{\mu}
\newunicodechar{ℕ}{\mathbb{N}}
\newunicodechar{ℝ}{\mathbb{R}}
\newunicodechar{∂}{\partial}
\newunicodechar{ν}{\nu}
\newunicodechar{Φ}{\Phi}
\newunicodechar{τ}{\tau}
\newunicodechar{ω}{\omega}
\newunicodechar{∞}{\infty}
\newunicodechar{ϕ}{\phi}

\makeatletter
\renewenvironment{proof}[1][\proofname]{\par
  \pushQED{\qed}%
  \normalfont \topsep0\p@\relax
  \trivlist
  \item[\hskip\labelsep\scshape
  #1\@addpunct{.}]\ignorespaces
}{%
  \popQED\endtrivlist\@endpefalse
}
\makeatother

\newtheorem{base}{Base}[section]
\numberwithin{equation}{section}

\newtheorem{theorem}[base]{Theorem} \newtheorem*{theorem*}{Theorem}
\newtheorem{lemma}[base]{Lemma} \newtheorem*{lemma*}{Lemma}
 \newtheorem*{prop*}{Proposition}
\newtheorem{cor}[base]{Corollary} \newtheorem*{cor*}{Corollary}
 \newtheorem*{algo*}{Algorithm}

\theoremstyle{definition}
\newtheorem{remark}[base]{Remark} \newtheorem*{remark*}{Remark}
 \newtheorem*{definition*}{Definition}
 \newtheorem*{example*}{Example}
 \newtheorem*{cond*}{Condition}

\numberwithin{algorithm}{section}

\allowdisplaybreaks

\begin{document}
\setkomafont{title}{\normalfont\Large}
\title{Flux-corrected transport stabilization of an evolutionary cross-diffusion
cancer invasion model}

\author[1]{Shahin Heydari \footnote{e-mail:\ heydari@karlin.mff.cuni.cz , corresponding author}}
\author[1]{Petr Knobloch \footnote{e-mail:\ knobloch@karlin.mff.cuni.cz}}
\author[2]{Thomas Wick \footnote{e-mail:\ thomas.wick@ifam.uni-hannover.de}}

\affil[1]{Charles University, 
          Faculty of Mathematics and Physics,
          Sokolovsk\'a 83, 18675 Praha 8, Czech Republic}
          
\affil[2]{Leibniz University Hannover,
	  Institute of Applied Mathematics,
	  Welfengarten 1, 30167 Hannover, Germany}

\date{}

\maketitle

\KOMAoptions{abstract=true}
\begin{abstract}
In the present work, we investigate a model of the invasion of healthy tissue
by cancer cells which is described by a system of nonlinear PDEs consisting of
a cross-diffusion-reaction equation and two additional nonlinear ordinary
differential equations. We show that when the convective part of the
system, the chemotactic term, is dominant, then straightforward numerical
methods for the studied system may be unstable. We present an implicit finite
element method using conforming $P_1$ or $Q_1$ finite elements to discretize
the model in space and the $\theta$-method for discretization in time. The
discrete problem is stabilized using a nonlinear flux-corrected transport
approach. It is proved that both the nonlinear scheme and the linearized
problems used in fixed-point iterations are solvable and positivity preserving.
Several numerical experiments are presented in 2D using the deal.II library to
demonstrate the performance of the proposed method.

\noindent
  \\[0.5pt]
 \textbf{AMS Classification (2020):} {
 65M22, 65M60, 
92C17, 
35Q92, 
 }
\end{abstract}

\section{Introduction}
\label{sec:introduction}

Keller and Segel \cite{keller1971model, keller1970initiation} proposed the
first mathematical model for description of chemotactical processes. Chemotaxis
refers to the motion in the direction to (or away from) the position of higher
concentration based on the gradient of chemical substances and its
chemotacitivity character which controls the speed of this motion. Their model
has been widely extended and followed to develop more sophisticated and complex
chemotaxis models and played a vitally important role in many areas of science,
in particular in medical and biological applications, for example, bacteria and
cell aggregation \cite{aida2006lower, mimura1996aggregating, tyson1999minimal},
tumor angiogenisis and invasion \cite{ anderson2000mathematical,
chaplain2005mathematical, chaplain2006mathematical, chaplain1993model},
biological pattern formation \cite{aida2004target, tyson1999minimal}, and
immune cell migration \cite{wu2005signaling}. From the analytical point of
view, mathematical analysis for chemotaxis systems of equations is a challenge
and causes many questions especially in the context of the existence and
uniqueness of solutions. In the last three decades, many researchers have been
actively involved and answered some of these questions
\cite{nanjundiah1973chemotaxis, corrias2004global, horstmann2005boundedness, 
tao2009combined, horstmann2011uniqueness, calvez2012blow}. From the numerical
point of view, so far a great deal of research on chemotaxis models has been
done in various areas, including the finite difference method
\cite{chaplain2005mathematical,chaplain2006mathematical,kolev2022unconditional},
discontinuous Galerkin method \cite{epshteyn2009new,li2017local}, finite
element method \cite{saito2007conservative,zhang2016characteristic,
zhao2020petrov}, finite volume method \cite{filbet2006finite},
operator-splitting methods \cite{ropp2009stability}, or fractional step 
algorithms \cite{tyson2000fractional}. However, many analytical and numerical 
aspects are still untouched and call for further investigation.

The chemotaxis problems are usually strongly coupled nonlinear systems of
equations whose solutions represent concentrations or densities and need to be
non-negative in order to satisfy the physics behind the system. Hence, it is
difficult to construct an efficient and accurate numerical method that does not
produce solutions with negative values. Another interesting aspect is singular,
spiky and oscillatory behavior of the solutions. In particular, when the
chemotaxis term dominates the diffusion and reaction terms, in other words
large chemosensitivity is present which corresponds to large Reynolds numbers,
it may give rise to nonphysical oscillations in the solution. To overcome this
problem, stabilization methods can be applied. Up to now, many scientists used
flux-corrected transport (FCT) algorithms, i.e., nonlinear high-resolution
schemes introduced by Boris and Book \cite{BB73, book1975flux, boris1976flux},
later developed based on linear finite element discretizations by Kuzmin,
L\"ohner et al.~\cite{lohner1987finite, Kuz09, Kuz12a, KT02}, and further 
extended to linear and nonlinear space-time FEM-FCT in \cite{FeNeuNaWi19}. In
\cite{strehl2010flux}, an implicit flux-corrected transport scheme was
developed and applied to three benchmark examples of the general Keller--Segal
model in two spatial dimensions. It was shown that the proposed method is
positivity preserving and sufficiently accurate, even in the cases where
solutions blow up in the center or at the boundary of the domain. The
investigations of the blow-up behavior of the solutions were further extended
to three spatial dimensions in \cite{strehl2013positivity}. In
\cite{sokolov2013numerical, sokolov2015afc}, an FEM-FCT scheme was coupled with
a level-set method to obtain positivity preserving solutions on a stationary
surface and evolving-in-time surfaces. It was shown that the proposed method is
able to produce accurate numerical solutions, which makes it possible to couple
the partial differential equations defined on a specific domain with the PDEs
that are defined on the surface of this domain.
This scheme was further used with operator-splitting techniques to solve
chemotaxis models in 3D. The operator-splitting method splitted a 3D problem
into a sequence of 1D subproblems and the FEM-FCT algorithm was used to solve
each 1D subproblem separately \cite{huang2020efficient}. In
\cite{sulman2019positivity}, the authors used an efficient adaptive moving mesh
finite element approach based on the parabolic Monge--Amp\`ere method for
determining the coordinate transformation for the adaptive mesh combined with
an FCT scheme which guarantees the non-negativity of the solutions. As a
result, the computational cost was significantly reduced. All aforementioned
techniques were also applied to the same benchmark examples. A different case
was studied in \cite{huang2021fully}, where the authors used the
pressure-correction scheme and flux-corrected transport algorithm to propose an
efficient linear positivity-preserving method for the solution of
chemotaxis--Stokes equations. 

In this work, we focus on a cancer-invasion model developed in
\cite{perumpanani1999two}, modeling the motion of cancer cells, degradation of
extracellular matrix, and certain enzymes (e.g., protease). The extracellular
matrix is degraded upon contact with protease which is produced where cancer
cells and extracellular matrix meet and decay over the time. In
\cite{fuest2022global}, we extended the proposed model by a diffusion term,
gave a rigorous proof for the existence of the global classical solution and 
presented numerical results for a Galerkin finite element discretization.
In the present paper, a diffusion term is not considered, which makes the
problem more challenging. In \cite{khalsaraei2016positivity}, one of the
authors of the present paper applied a positivity preserving non-standard
finite difference method to solve the nonlinear system in 1D, see also
\cite{chapwanya2014positivity} for related approaches. Here, we consider the
finite element method and apply the FCT technique to guarantee the positivity
preservation. First, however, we consider the more diffusive nonlinear 
low-order method. An additional nonlinearity is then introduced by the flux
correction. We prove that both nonlinear problems are solvable and positivity
preserving. To the best of our knowledge, the current work is a first attempt 
to gain an insight into the applicability of the FCT technique to the numerical
solution of a chemotaxis system without self-diffusion and to provide a rigorous
analysis of the solvability and positivity preservation. Note that the 
existence and uniqueness for the FEM-FCT method applied to linear evolutionary 
convection-diffusion equations has been addressed only recently in 
\cite{john2021solvability, john2021existence}. We also present a fixed-point
algorithm for the iterative solution of the FCT discretization and prove that
it is well posed and provides a non-negative solution at each step.
Consequently, the non-negativity of the approximate solution is guaranteed
independently of the choice of a stopping criterion. The properties of the
proposed FCT scheme are illustrated by various numerical simulations carried
out using our newly designed algorithm in the deal.II library
\cite{deal2020,dealII94}.

The outline of this paper is as follows. In Section~\ref{sec:model}, we 
formulate the mathematical model which is discretized by the Galerkin method in
Section~\ref{sec:galerkin}. Then, the FCT stabilization is introduced in
Section~\ref{sec:fct}, where also the solvability and positivity preservation
is proved. The fixed-point algorithm is proposed and investigated in
Section~\ref{sec:iter_sol}. In Section~\ref{sec:numerics}, we report several 
numerical simulations in two spatial dimensions carried out for various regimes.
Finally, our results are summarized in Section~\ref{sec:conclusions}.

\section{Mathematical model}
\label{sec:model}

In this section, we discuss the following nondimensionalized continuous model
of a malignant cancer invasion proposed by Perumpanani et al. in
\cite{perumpanani1999two, marchant_norbury_perumpanani}. The model contains
three unknown variables, namely the cancer cell density $u=u(x,t)$, connective
tissue $c=c(x,t)$, and protease $p=p(x,t)$, and it consists of the equations
\begin{alignat}{2}
   & \dfrac{\partial u}{\partial t} = \mu\,u(1 - u) 
   - \chi\,\nabla \cdot (u \nabla c)  \qquad
   && \text{in $\Omega \times (0, T]$}\,, \label{eq1} \\
   & \dfrac{\partial c}{\partial t} = -pc  
   && \text{in $\Omega \times (0, T]$}\,, \label{eq2}\\
   & \dfrac{\partial p}{\partial t} = \epsilon^{-1} (uc - p) 
   && \text{in $\Omega \times (0, T]$}\,, \label{eq3}
\end{alignat}
where $\Omega$ is a bounded polyhedral domain in $ℝ^d, d\in\{1,2,3\}$, $[0,T]$ 
is a time interval, and $\mu$, $\chi$, $\epsilon$ are positive constants. Here,
$\mu$ and $\chi$ denote the proliferation and
haptotaxis rate of cancer cells, respectively, and the parameter $\epsilon$ is
supposed to be small since the units of connective tissues and invasive cells
are much larger than the protease. In the process of invasion, the connective
tissue is affected by the invasive flux of $u\nabla c$ into its compartment.
Since the connective tissue does not contain any empty space large enough for
passing of passive cancer cells, it degrades by protease which is produced by
invasive cancer cells upon contact with connective tissue. It can be shown that
if the initial conditions of the above model are non-negative, then the 
computed solutions stay non-negative at all times, for more details see
\cite{marchant_norbury_perumpanani, marchant_norbury_sherratt,
chapwanya2014positivity} and the references therein.

The system (\ref{eq1})--(\ref{eq3}) is subjected to the homogeneous Neumann
boundary condition
\begin{equation}\label{eq4}
   u\,\frac{\partial c}{\partial n} = 0\qquad
   \text{on}\,\,\,\partial\Omega\times[0,T]\,,
\end{equation}
where $n$ is the unit outward normal vector on $\partial \Omega$. The above
equations are endowed with the initial conditions
\begin{equation}\label{eq5} 
   u(x,0) = u^0(x)\,, \quad c(x,0) = c^0(x)\,, \quad p(x,0) = p^0(x)\,,\qquad
   x\in\Omega\,,
\end{equation}
where $u^0,\,c^0,\,p^0:\Omega\to[0,1]$ are given functions.

In \cite{fuest2022global}, we considered a modified version of
(\ref{eq1})--(\ref{eq3}) containing an extra diffusion term in (\ref{eq1}).
Precisely, instead of the equation (\ref{eq1}), we considered
\begin{equation}\label{eq6}
  \dfrac{\partial u}{\partial t} = \mu\,u(1 - u) 
  - \chi\,\nabla \cdot (u \nabla c)  + \alpha^{-1}\Delta u
  \qquad\quad \text{in $\Omega \times (0, T]$}
\end{equation}
with a positive constant $\alpha$. This required to replace the boundary
condition (\ref{eq4}) by
\begin{equation*}
   \alpha^{-1}\,\frac{\partial u}{\partial n} 
   = \chi\,u\,\frac{\partial c}{\partial n}\qquad
   \text{on}\,\,\,\partial\Omega\times[0,T]\,.
\end{equation*}
Thus, the problem considered in this paper corresponds to the limit case
$\alpha\to\infty$ of the problem from \cite{fuest2022global}. In that paper, we
proved the existence of global classical solutions for two- and
three-dimensional bounded domains $\Omega$ with smooth boundaries and we proved
that these solutions are non-negative. Moreover, we showed that by fixing the
proliferation rate $\mu$ and varying the haptotaxis $\chi$ one can make either
the diffusion or the transport of the cancer cells dominant. The domination of
the convection term can produce spurious oscillations and a blow-up in the
solution of the system as it is the case to be considered in here.

\section{A Galerkin discretization}
\label{sec:galerkin}

The solution of the problem \eqref{eq1}--\eqref{eq5} satisfies
\begin{alignat}{2}
   &\left(\dfrac{\partial u}{\partial t},v\right) = \mu\,\big(u(1 - u),v\big) + 
     \chi\,\big(u \nabla c,\nabla v\big) \qquad
   && \text{in $(0, T]$ and for $v\in H^1(\Omega)$}\,, \label{eq1v} \\
   &c(x,t)=c^0(x)\,{\mathrm e}^{-\int_0^t\,p(x,s)\,{\mathrm d}s} 
   &&\forall\,\,(x,t)\in\Omega\times[0, T]\,,\label{eq2v}\\
   &p(x,t)={\mathrm e}^{-t/\epsilon}\left[p^0(x)
   +\frac1\epsilon\,\int_0^t\,u(x,s)\,c(x,s)\,
   {\mathrm e}^{s/\epsilon}\,{\mathrm d}s\right]\quad
   &&\forall\,\,(x,t)\in\Omega\times[0, T]\,,\label{eq3v}
\end{alignat}
where $(\cdot,\cdot)$ denotes the inner product in $L^2(\Omega)$ or
$L^2(\Omega)^d$. To define an approximate solution of \eqref{eq1}--\eqref{eq5},
we first introduce a triangulation ${\mathcal T}_h$ of $\Omega$ consisting of
simplicial (for $d=1,2,3$), quadrilateral (for $d=2$) or hexahedral (for $d=3$)
shape-regular cells possessing the usual compatibility properties (see, e.g.,
\cite{Ciarlet}). For any cell $K\in{\mathcal T}_h$, we denote by $h_K$ the
diameter of $K$ and assume that $h_K\le h$. We denote by $V_h\subset
H^1(\Omega)$ the usual conforming $P_1$ or $Q_1$ finite element space
constructed using the triangulation ${\mathcal T}_h$. Let $\phi_1,\dots,\phi_M$
be the standard basis functions of $V_h$ associated with the vertices
$x_1,\dots,x_M$ of ${\mathcal T}_h$. Thus, the basis functions are non-negative
and satisfy $\phi_i(x_j)=\delta_{ij}$ for $i,j=1,\dots,M$, where $\delta_{ij}$
is the Kronecker symbol. Any function $v_h\in V_h$ can be identified with a
coefficient vector $\mathbf{v}=(v_j)_{j=1}^M$ with respect to these basis
functions. Precisely, introducing the bijective operator $\pi_h:{\mathbb
R}^M\to V_h$ by
\begin{equation*}
   \pi_h\mathbf{v}=\sum_{j=1}^M\,v_j\,\phi_j\,,
\end{equation*}
one has $v_h=\pi_h\mathbf{v}$. The assumed shape regularity of ${\mathcal
T}_h$ implies that
\begin{equation}\label{eq:grad_est}
   \|\nabla\phi_i\|_{L^2(K)}^{}\le\kappa\,h_K^{d/2-1}\qquad
   \forall\,\,K\in {\mathcal T}_h\,,\,\,i=1,\dots,M\,,
\end{equation}
where $\kappa$ is a fixed constant independent of $i$, $K$, and $h$.  Next, the
time interval $[0,T]$ is decomposed by $0=t_0<t_1<\dots<t_N=T$ and we set
$\tau_n = t_n - t_{n-1}$, $n=1,\dots,N$. At each time level $t_n$, the solution
of \eqref{eq1}--\eqref{eq5} will be approximated by functions
$u_h^n,c_h^n,p_h^n\in V_h$. These functions can be identified with coefficient
vectors $\mathbf{u}^n=(u_j^n)_{j=1}^M$, $\mathbf{c}^n=(c_j^n)_{j=1}^M$,
$\mathbf{p}^n=(p_j^n)_{j=1}^M$, respectively, satisfying
$u_h^n=\pi_h\mathbf{u}^n$, $c_h^n=\pi_h\mathbf{c}^n$,
$p_h^n=\pi_h\mathbf{p}^n$.  Note that $u_h^n(x_i)=u_i^n$, $c_h^n(x_i)=c_i^n$,
and $p_h^n(x_i)=p_i^n$ for $i=1,\dots,M$. We set
\begin{equation}\label{eq:init_cond}
   u^0_i=u^0(x_i)\,,\qquad c^0_i=c^0(x_i)\,,\qquad p^0_i=p^0(x_i)\,,\qquad\quad 
   i=1,\dots,M\,.
\end{equation}
Using linear interpolation with respect to time between the time levels
gives functions $u_{h,\tau}$, $c_{h,\tau}$, $p_{h,\tau}$ defined on
$\overline\Omega\times[0,T]$. For example, $u_{h,\tau}$ satisfies
\begin{equation*}
   u_{h,\tau}(x,t)=\frac1{\tau_{n+1}}
   \left[u_h^{n+1}(x)(t-t_n)+u_h^n(x)(t_{n+1}-t)\right]\qquad\forall\,\,
   x\in\overline\Omega\,,\,\,t\in[t_n,t_{n+1}]\,,\,\,n=0,\dots,N-1\,,
\end{equation*}
or, equivalently,
\begin{equation*}
   u_{h,\tau}(x_i,t)=\frac1{\tau_{n+1}}
   \left[u_i^{n+1}(t-t_n)+u_i^n(t_{n+1}-t)\right]\qquad\forall\,\,
   i=1,\dots,M\,,\,\,t\in[t_n,t_{n+1}]\,,\,\,n=0,\dots,N-1\,.
\end{equation*}

Replacing the space $H^1(\Omega)$ in \eqref{eq1v} by $V_h$ and applying 
the $\theta$-method for discretization in time (with $\theta\in[0,1]$), one 
obtains
\begin{align}
  &\left(\frac{u_h^{n+1}-u_h^n}{\tau_{n+1}},v_h\right) =
   \theta\,\mu\,\big(u_h^{n+1}(1 - u_h^{n+1}),v_h\big) +
   \theta\,\chi\,\big(u_h^{n+1}\nabla c_h^{n+1},\nabla v_h\big)\nonumber\\
  &\hspace*{10mm}
   +(1-\theta)\,\mu\,\big(u_h^n(1 - u_h^n),v_h\big) + 
   (1-\theta)\,\chi\,\big(u_h^n\nabla c_h^n,\nabla v_h\big) 
   \qquad\forall\,\,v_h\in V_h\,,\,\,n=0,\dots,N-1\,.\label{eq1vd}
\end{align}
Defining the matrices $\mathbb M=(m_{ij})_{i,j=1}^M$ and 
${\mathbb A}^n=(a^n_{ij})_{i,j=1}^M$ with
\begin{equation*}
   m_{ij}=(\phi_j,\phi_i)\,,\qquad 
   a^n_{ij}=-\mu\,\big(\phi_j(1 - u_h^n),\phi_i\big) 
            -\chi\,\big(\phi_j\nabla c_h^n,\nabla\phi_i\big)\,,
\end{equation*}
the discrete variational problem \eqref{eq1vd} can be written in the matrix
form
\begin{equation}\label{eq1m}
   ({\mathbb M}+\theta\,\tau_{n+1}\,{\mathbb A}^{n+1})\,\mathbf{u}^{n+1}=
   ({\mathbb M}-(1-\theta)\,\tau_{n+1}\,{\mathbb A}^n)\,\mathbf{u}^n\,,\qquad
   n=0,\dots,N-1\,.
\end{equation}
The relations \eqref{eq2v} and \eqref{eq3v} suggest to define the coefficients
of $c_h^n$ and $p_h^n$ by
\begin{alignat}{2}
   &c_i^n=c^0(x_i)\,{\mathrm e}^{-\int_0^{t_n}\,p_{h,\tau}(x_i,s)\,{\mathrm d}s}\,,
   &&i=1,\dots,M\,,\,\,n=0,\dots,N\,,\label{eq2vd}\\
   & p_i^n={\mathrm e}^{-t_n/\epsilon}\left[p^0(x_i)
   +\frac1\epsilon\,\int_0^{t_n}\,u_{h,\tau}(x_i,s)\,c_{h,\tau}(x_i,s)\,
   {\mathrm e}^{s/\epsilon}\,{\mathrm d}s\right],\qquad
   &&i=1,\dots,M\,,\,\,n=0,\dots,N\,.\label{eq3vd}
\end{alignat}
Then, for $i=1,\dots,M$ and $n=0,\dots,N-1$, one has
\begin{align}
   &c_i^{n+1}=c_i^n\,{\mathrm e}^{-\int_{t_n}^{t_{n+1}}\,p_{h,\tau}(x_i,s)\,{\mathrm d}s}\,,\\
   & p_i^{n+1}={\mathrm e}^{-\tau_{n+1}/\epsilon}\,p^n_i
   +\frac1\epsilon\,{\mathrm e}^{-t_{n+1}/\epsilon}
   \int_{t_n}^{t_{n+1}}\,u_{h,\tau}(x_i,s)\,c_{h,\tau}(x_i,s)\,
   {\mathrm e}^{s/\epsilon}\,{\mathrm d}s\,.\label{eq3vsf}
\end{align}
A direct computation gives
\begin{alignat}{3}
   c_i^{n+1}&=&\,c_i^n\,{\mathrm e}^{-\tau_{n+1}\,(p^{n+1}_i+p^n_i)/2}\,&,
   \label{eq2vf}\\
   p_i^{n+1}&=&\,\,{\mathrm e}^{-\tau_{n+1}/\epsilon}\,p^n_i
   +\frac1{\tau_{n+1}^2}\Bigg\{&
   \Big(u^{n+1}_i\,(\epsilon-\tau_{n+1})-u^n_i\,\epsilon\Big)
   \Big(c^{n+1}_i\,(\epsilon-\tau_{n+1})-c^n_i\,\epsilon\Big)\nonumber\\
   &&-\,&\Big(u^{n+1}_i\,\epsilon-u^n_i\,(\epsilon+\tau_{n+1})\Big)
     \Big(c^{n+1}_i\,\epsilon-c^n_i\,(\epsilon+\tau_{n+1})\Big)\,
     {\mathrm e}^{-\tau_{n+1}/\epsilon}\nonumber\\
   &&+\,&(u^{n+1}_i-u^n_i)(c^{n+1}_i-c^n_i)\,\epsilon^2
   \left(1-{\mathrm e}^{-\tau_{n+1}/\epsilon}\right)\Bigg\},\label{eq3vf}
\end{alignat}
for $i=1,\dots,M$ and $n=0,\dots,N-1$. Note that the effects described by the 
model \eqref{eq1}--\eqref{eq5}, such as chemotaxis, strongly rely on the 
nonlinear coupling terms. Therefore, all nonlinearities are treated implicitly
in the discrete problem \eqref{eq1m}--\eqref{eq3vd}.

To compute a solution of the nonlinear problem \eqref{eq1m}--\eqref{eq3vd} at
time $t_{n+1}$ (assuming that the solution vectors $\mathbf{u}^{n}$,
$\mathbf{c}^{n}$, and $\mathbf{p}^{n}$ at the previous time instant $t_n$ are
known), we apply simple fixed-point iterations leading to sequences 
$\mathbf{u}^{n+1}_k=(u_{j,k}^{n+1})_{j=1}^M$, 
$\mathbf{c}^{n+1}_k=(c_{j,k}^{n+1})_{j=1}^M$, and
$\mathbf{p}^{n+1}_k=(p_{j,k}^{n+1})_{j=1}^M$. We set
$\mathbf{u}^{n+1}_0=\mathbf{u}^n$, $\mathbf{c}^{n+1}_0=\mathbf{c}^n$,
$\mathbf{p}^{n+1}_0=\mathbf{p}^n$ and then, for $k>0$ and $i=1,\dots,M$, we
define
\begin{alignat}{3}
   c_{i,k}^{n+1}&=&\,\,c_i^n\,
   {\mathrm e}^{-\tau_{n+1}\,(p^{n+1}_{i,k-1}+p^n_i)/2}&\,,\label{eq:cik}\\
   p_{i,k}^{n+1}&=&{\mathrm e}^{-\tau_{n+1}/\epsilon}\,p^n_i
   +\frac1{\tau_{n+1}^2}\Bigg\{&
   \Big(u^{n+1}_{i,k-1}\,(\epsilon-\tau_{n+1})-u^n_i\,\epsilon\Big)
   \Big(c^{n+1}_{i,k}\,(\epsilon-\tau_{n+1})-c^n_i\,\epsilon\Big)\nonumber\\
   &&-\,&\Big(u^{n+1}_{i,k-1}\,\epsilon-u^n_i\,(\epsilon+\tau_{n+1})\Big)
     \Big(c^{n+1}_{i,k}\,\epsilon-c^n_i\,(\epsilon+\tau_{n+1})\Big)\,
     {\mathrm e}^{-\tau_{n+1}/\epsilon}\nonumber\\
   &&+\,&(u^{n+1}_{i,k-1}-u^n_i)(c^{n+1}_{i,k}-c^n_i)\,\epsilon^2
   \left(1-{\mathrm e}^{-\tau_{n+1}/\epsilon}\right)\Bigg\}.\label{eq:pik}
\end{alignat}
The iterate $\mathbf{u}^{n+1}_k$ is computed by solving the linear system
\begin{equation}\label{eq:uk}
   ({\mathbb M}+\theta\,\tau_{n+1}\,{\mathbb A}^{n+1}_{k-1})\,\mathbf{u}^{n+1}_k
   =({\mathbb M}-(1-\theta)\,\tau_{n+1}\,{\mathbb A}^n)\,\mathbf{u}^n\,,
\end{equation}
where the matrix ${\mathbb A}^{n+1}_{k-1}$ is defined by
\begin{equation}\label{eq:akm1}
   {\mathbb A}^{n+1}_{k-1}
   =\Big(-\mu\,\big(\phi_j(1 - u_{h,k-1}^{n+1}),\phi_i\big) 
        -\chi\,\big(\phi_j\nabla c_{h,k}^{n+1},\nabla\phi_i\big)\Big)_{i,j=1}^M
\end{equation}
and $u_{h,k-1}^{n+1}=\pi_h\mathbf{u}^{n+1}_{k-1}$ and 
$c_{h,k}^{n+1}=\pi_h\mathbf{c}^{n+1}_k$ are the finite element functions
corresponding to the coefficient vectors $\mathbf{u}^{n+1}_{k-1}$ and
$\mathbf{c}^{n+1}_k$, respectively.

The linear system \eqref{eq:uk} has the form
\begin{equation}\label{eq:general_system}
   {\mathbb B}\,\mathbf{u}^{n+1}
   ={\mathbb K}\,\mathbf{u}^n
\end{equation}
and it is desirable that this system is positivity preserving, i.e., that
$\mathbf{u}^{n+1}\ge0$ if $\mathbf{u}^n\ge0$. A necessary and sufficient
condition for this property is ${\mathbb B}^{-1}\,{\mathbb K}\ge0$ but this
condition is difficult to verify. Sufficient conditions are formulated in the
following lemma. Note that, throughout the paper, an inequality of the type
$\mathbf{u}^n\ge0$ means that the inequality holds for each component of the
vector $\mathbf{u}^n$. Similarly, the fact that all entries of a matrix 
$\mathbb K$ are non-negative is expressed by ${\mathbb K}\ge0$.

\begin{lemma}\label{lem:lin_syst}
Let the matrices ${\mathbb B}=(b_{ij})_{i,j=1}^M$ and 
${\mathbb K}=(k_{ij})_{i,j=1}^M$ satisfy
\begin{equation*}
  b_{ii}\ge0\,, \quad k_{ii}\ge0\,, \quad 
  b_{ij}\leqslant 0\,, \quad k_{ij} \geqslant 0,
  \qquad\forall\,\,i,j=1,\dots,M\,,\,\,i\neq j\,,
\end{equation*}
and let $\mathbb B$ be a strictly diagonally dominant or an irreducibly
diagonally dominant matrix. Then $\mathbb B$ is an M-matrix and the scheme 
\eqref{eq:general_system} is positivity preserving.
\end{lemma}

\begin{proof}According to \cite[Theorem~3.27]{Var00}, $\mathbb B$ is an 
M-matrix. Thus, ${\mathbb B}^{-1}\ge0$ and hence also ${\mathbb
B}^{-1}\,{\mathbb K}\ge0$, which implies the result.
\end{proof}

In general, the linear system \eqref{eq:uk} originating from a standard 
Galerkin discretization does not satisfy the above constraints because the mass 
matrix is non-negative and the stiffness matrix may contain positive
off-diagonal entries. Our numerical results in Section~\ref{sec:numerics} show
that indeed the concentration $\mathbf{u}$ may become negative in some parts of
the computational domain $\Omega$.

\section{FCT stabilization}
\label{sec:fct}

As we will see in Section~\ref{sec:numerics}, the magnitude of the solutions 
gradients can be extremely large in some regions. The solution of the Galerkin
discretization from the previous section may become negative especially in
these regions. As a remedy, in the following we will modify the Galerkin
discretization to guarantee a positivity preservation property. As shown by 
Kuzmin \cite{Kuz09, Kuz12a, KT02}, this property can be readily enforced at the 
discrete level using a conservative manipulation of the mass and stiffness 
matrices. The former will be approximated by its diagonal counterpart $\ML$ 
constructed using row-sum mass lumping, whereas the latter will be modified by
adding an artificial diffusion matrix. To limit the amount of the artificial
diffusion, the FEM-FCT approach will be applied following \cite{Kuz12a}.

Since the methods considered in this section guarantee that the approximate
solutions are non-negative, it is possible to replace the matrix
${\mathbb A}^n$ from the previous section by $\tilde{\mathbb
A}^n=(\tilde{a}^n_{ij})_{i,j=1}^M$ with
\begin{equation*}
   \tilde{a}^n_{ij}=-\mu\,\big(\phi_j(1 - |u_h^n|),\phi_i\big) 
            -\chi\,\big(\phi_j\nabla c_h^n,\nabla\phi_i\big)\,.
\end{equation*}
The matrix $\tilde{\mathbb A}^n$ is more suitable for theoretical
considerations than the matrix ${\mathbb A}^n$. However, a non-negative
approximate solution $u^n_h$, $c^n_h$, $p^n_h$ satisfying a discrete problem
based on the matrix $\tilde{\mathbb A}^n$ will satisfy also the corresponding
discrete problem with the original matrix ${\mathbb A}^n$.

Using the matrix $\tilde{\mathbb A}^n$, we introduce a symmetric artificial 
diffusion matrix ${\mathbb D}^n=(d^n_{ij})_{i,j=1}^M$ defined by
\begin{equation*}
  d^n_{ij}=-\max\{\tilde{a}^n_{ij},0,\tilde{a}^n_{ji}\}\quad
  \mbox{for} \ i\ne j\,,\qquad\quad
  d^n_{ii}=-\sum_{j=1,j \neq i}^M d^n_{ij}\,,
\end{equation*}
and we set ${\mathbb L}^n=\tilde{\mathbb A}^n+{\mathbb D}^n$. Note that 
${\mathbb L}^n=(l^n_{ij})_{i,j=1}^M$ is a Z-matrix (i.e., it has
non-positive off-diagonal entries). Furthermore, we introduce the lumped mass 
matrix $\ML=\mbox{diag}(m_1,\dots,m_M)$ with
\begin{equation*}
   m_i=\sum_{j=1}^M\,m_{ij}\,,\qquad i=1,\dots,M\,.
\end{equation*}
Now, the simplest way to enforce the positivity preservation is to consider the
so-called low-order method corresponding to the so-called high-order method 
\eqref{eq1m} which is defined by
\begin{equation}\label{eq1low}
   (\ML+\theta\,\tau_{n+1}\,{\mathbb L}^{n+1})\,\mathbf{u}^{n+1}=
   (\ML-(1-\theta)\,\tau_{n+1}\,{\mathbb L}^n)\,\mathbf{u}^n\,,\qquad
   n=0,\dots,N-1\,.
\end{equation}
Note that the matrix ${\mathbb L}^{n+1}$ depends on $\mathbf{u}^{n+1}$ and 
$\mathbf{c}^{n+1}$ so that the low-order problem is again nonlinear. In
contrast to the Galerkin discretization \eqref{eq1m}, it is now possible to
assure the positivity preservation for sufficiently small time steps.

\begin{lemma}\label{lem:time_step}
Let the time step $\tau_{n+1}$ satisfy the conditions
\begin{equation}\label{eq:time_step}
   (1-\theta)\,\tau_{n+1}\,l^n_{ii}\le m_i\,,\qquad
   \theta\,\tau_{n+1}\,\Big(\mu\,m_i
   +\chi\,\big(\nabla c_h^{n+1},\nabla\phi_i\big)\Big)<m_i\,,\qquad\quad
   i=1,\dots,M\,.
\end{equation}
Then the matrix $\ML-(1-\theta)\,\tau_{n+1}\,{\mathbb L}^n$ has non-negative
entries and $\ML+\theta\,\tau_{n+1}\,{\mathbb L}^{n+1}$ is an M-matrix.
\end{lemma}

\begin{proof}
The first condition in \eqref{eq:time_step} implies that
$\ML-(1-\theta)\,\tau_{n+1}\,{\mathbb L}^n$ has non-negative diagonal entries.
The off-diagonal entries of this matrix are non-negative as well, since $\ML$
is diagonal and ${\mathbb L}^n$ is a Z-matrix. 

Denoting $\mathbb{B}=\ML+\theta\,\tau_{n+1}\,{\mathbb L}^{n+1}$, one has for
any $i\in\{1,\dots,M\}$
\begin{equation*}
   \sum_{j=1}^M\,b_{ij}=m_i+\theta\,\tau_{n+1}\,\sum_{j=1}^M\,\tilde{a}^{n+1}_{ij}
   =m_i-\theta\,\tau_{n+1}\,\Big(\mu\,\big(1 - |u_h^{n+1}|,\phi_i\big)
   +\chi\,\big(\nabla{c}_h^{n+1},\nabla\phi_i\big)\Big)\,,
\end{equation*}
where we used the fact that $\sum_{j=1}^M\,\phi_j=1$. Since $(1,\phi_i)=m_i$,
it follows from the second condition in \eqref{eq:time_step} that
$\sum_{j=1}^M\,b_{ij}>0$. Thus, $b_{ii}>\sum_{j\neq i}\,|b_{ij}|$, i.e.,
$\mathbb{B}$ is strictly diagonally dominant and hence non-singular. Moreover,
$\mathbb{B}$ is a matrix of non-negative type and hence it is an M-matrix (see,
e.g., \cite[Corollary 3.13]{BJK23}).
\end{proof}

\begin{cor}\label{cor:pos_preserv}
Let the time step $\tau_{n+1}$ satisfy the conditions \eqref{eq:time_step}.
Then the low-order scheme \eqref{eq1low} is positivity preserving, i.e.,
\begin{equation}\label{eq:pos_preserv}
   \mathbf{u}^n\geqslant0\quad \Rightarrow\quad \mathbf{u}^{n+1}\geqslant 0\,.
\end{equation}
\end{cor}

\begin{proof}
According to Lemma~\ref{lem:time_step}, the matrix
$\ML+\theta\,\tau_{n+1}\,{\mathbb L}^{n+1}$ is non-singular,
$(\ML+\theta\,\tau_{n+1}\,{\mathbb L}^{n+1})^{-1}\ge0$, and 
$\ML-(1-\theta)\,\tau_{n+1}\,{\mathbb L}^n\ge0$, which immediately implies 
\eqref{eq:pos_preserv}.
\end{proof}

\begin{remark}\label{rem:pos_preserv}
The second condition in \eqref{eq:time_step} involves ${c}^{n+1}_h$ which
implicitly depends on $\tau_{n+1}$ through $\mathbf{p}^{n+1}$ and hence also
through $\mathbf{u}^{n+1}$. Therefore, it is desirable to replace this
condition by a condition independent of $c^{n+1}_h$. This is possible since we
will show that the values of ${c}^{n+1}_h$ are in the interval $[0,1]$. Then,
employing \eqref{eq:grad_est}, one gets
\begin{displaymath}
   \big(\nabla{c}_h^{n+1},\nabla\phi_i\big)
   =\sum_{j=1}^M\,{c}^{n+1}_j\,(\nabla\phi_j,\nabla\phi_i)
   \le\sum_{K\ni x_i}\,\sum_{j=1}^M\,\|\nabla\phi_j\|_{L^2(K)}^{}\,
   \|\nabla\phi_i\|_{L^2(K)}^{}
   \le n_{\mathrm{v}}\,\kappa^2\sum_{K\ni x_i}h_K^{d-2}\,,
\end{displaymath}
where $n_{\mathrm{v}}$ is the number of vertices of a cell in ${\mathcal T}_h$
($n_{\mathrm{v}}=d+1$ for simplices, $n_{\mathrm{v}}=4$ for quadrilaterals, and
$n_{\mathrm{v}}=8$ for hexahedra). Thus, if the time step $\tau_{n+1}$
satisfies
\begin{equation}\label{eq:time_step2}
   \theta\,\tau_{n+1}\left(\mu\,m_i
   +\chi\,n_{\mathrm{v}}\,\kappa^2\sum_{K\ni x_i}h_K^{d-2}\right)<m_i\,,\qquad
   i=1,\dots,M\,,
\end{equation}
and ${c}_h^{n+1}\in[0,1]$, then the second condition in \eqref{eq:time_step}
holds. Note that \eqref{eq:time_step2} may be significantly more restrictive
than \eqref{eq:time_step}.
\end{remark}

To prove that the low-order discretization consisting of the equations
\eqref{eq1low}, \eqref{eq2vf}, and \eqref{eq3vf} has a solution, we shall use
the following consequence of Brouwer's fixed-point theorem.

\begin{lemma}\label{lem:brouwer}
Let $X$ be a finite-dimensional Hilbert space with inner product
$(\cdot,\cdot)_X^{}$ and norm $\|\cdot\|_X^{}$. Let $P:X\to X$ be a continuous
mapping and $K>0$ a real number such that $(Px,x)_X^{}>0$ for any $x\in X$ with
$\|x\|_X^{}=K$. Then there exists $x\in X$ such that $\|x\|_X^{}<K$ and $Px=0$.
\end{lemma}

\begin{proof}See \cite[p.~164, Lemma 1.4]{Temam77}.\end{proof}

\begin{theorem}\label{thm:low_order_existence}
Consider any $n\in\{0,\dots,N-1\}$ and let
$\mathbf{u}^n,\mathbf{c}^n,\mathbf{p}^n\in{\mathbb R}^M$ satisfy
$\mathbf{u}^n\ge0$, $1\ge\mathbf{c}^n\ge0$, $\mathbf{p}^n\ge0$. Let the time 
step $\tau_{n+1}$ satisfy the conditions
\begin{equation}\label{eq:time_step3}
   (1-\theta)\,\tau_{n+1}\,l^n_{ii}\le m_i\,,\qquad
   \theta\,\tau_{n+1}\left(\mu\,m_i
   +\chi\,n_{\mathrm{v}}\,\kappa^2\sum_{K\ni x_i}h_K^{d-2}\right)<m_i\,,\qquad
   i=1,\dots,M\,.
\end{equation}
Then there exist vectors $\mathbf{u}^{n+1},\mathbf{c}^{n+1}, 
\mathbf{p}^{n+1}\in{\mathbb R}^M$ satisfying \eqref{eq1low}, \eqref{eq2vf}, 
\eqref{eq3vf} and $\mathbf{u}^{n+1}\ge0$, $1\ge\mathbf{c}^{n+1}\ge0$, 
$\mathbf{p}^{n+1}\ge0$.
\end{theorem}

\begin{proof}
To get rid of the exponential dependence on $\mathbf{p}^{n+1}$ when estimating
the nonlinear terms in \eqref{eq1low}, we replace \eqref{eq2vf} by
\begin{equation}\label{eq2vfmod}
   c_i^{n+1}=c_i^n\,{\mathrm e}^{-\tau_{n+1}\,(|p^{n+1}_i|+p^n_i)/2}\,,\qquad
   i=1,\dots,M\,.
\end{equation}
At the end of the proof, we will show that $\mathbf{p}^{n+1}\ge0$ so that the
original relation \eqref{eq2vf} will be recovered.

For $u,p\in{\mathbb R}$ and $i=1,\dots,M$, we introduce the notation 
\begin{equation}\label{eq:Ci}
   C_i(p)=c_i^n\,{\mathrm e}^{-\tau_{n+1}\,(|p|+p^n_i)/2}
\end{equation}
and
\begin{align}
   P_i(u,p)=\,\,\frac1{\tau_{n+1}^2}&\,u\,C_i(p)\,
   \Bigg\{(\epsilon-\tau_{n+1})^2
   +\epsilon^2\left(1-2\,{\mathrm e}^{-\tau_{n+1}/\epsilon}\right)\Bigg\}
   \nonumber\\
   +\frac\epsilon{\tau_{n+1}^2}&\,u\,c^n_i\,
   \Bigg\{\tau_{n+1}\left(1+{\mathrm e}^{-\tau_{n+1}/\epsilon}\right)
   -2\,\epsilon\left(1-{\mathrm e}^{-\tau_{n+1}/\epsilon}\right)\Bigg\}
   \nonumber\\
   +\frac\epsilon{\tau_{n+1}^2}&\,u^n_i\,C_i(p)\,
   \Bigg\{\tau_{n+1}\left(1+{\mathrm e}^{-\tau_{n+1}/\epsilon}\right)
   -2\,\epsilon\left(1-{\mathrm e}^{-\tau_{n+1}/\epsilon}\right)\Bigg\}
   \nonumber\\
   +\frac1{\tau_{n+1}^2}&\,u^n_i\,c^n_i\,
   \Bigg\{\epsilon^2\left(2-{\mathrm e}^{-\tau_{n+1}/\epsilon}\right)
   -(\epsilon+\tau_{n+1})^2\,{\mathrm e}^{-\tau_{n+1}/\epsilon}\Bigg\}
   +{\mathrm e}^{-\tau_{n+1}/\epsilon}\,p^n_i\,.\label{eq:Pi}
\end{align}
Then, the validity of \eqref{eq2vfmod} and \eqref{eq3vf} is equivalent to
\begin{equation*}
   c^{n+1}_i=C_i(p^{n+1}_i)\,,\quad p^{n+1}_i=P_i(u^{n+1}_i,p^{n+1}_i)\,,
   \qquad i=1,\dots,M\,.
\end{equation*}
Note that
\begin{equation}\label{eq:P_i_est}
   |P_i(u,p)|\le(|u|+u^n_i)\,c^n_i
   \left(\frac{2\,\epsilon+\tau_{n+1}}{\tau_{n+1}}\right)^2+p^n_i\qquad
   \forall\,\,u,p\in{\mathbb R}\,,\,\,i=1,\dots,M\,.
\end{equation}
Furthermore, for $\mathbf{u},\mathbf{p}\in{\mathbb R}^M$ and $i,j=1,\dots,M$, 
we denote
\begin{align}
   &A_{ij}(\mathbf{u},\mathbf{p})=
   -\mu\,\big(\phi_j(1 -|\pi_h\mathbf{u}|),\phi_i\big)
   -\chi\,\big(\phi_j\nabla(\pi_h\mathbf{C}(\mathbf{p})),\nabla\phi_i\big)\,,
   \label{eq:Aij}\\
   &D_{ij}(\mathbf{u},\mathbf{p})=
   -\max\{A_{ij}(\mathbf{u},\mathbf{p}),0,A_{ji}(\mathbf{u},\mathbf{p})\}
   \quad\mbox{for} \ i\ne j\,,\qquad\quad
   D_{ii}(\mathbf{u},\mathbf{p})=
   -\sum_{j=1,j \neq i}^M D_{ij}(\mathbf{u},\mathbf{p})\,,\label{eq:Dij}\\
   &S_i(\mathbf{u},\mathbf{p})=m_i\,u_i
   +\theta\,\tau_{n+1}\,\sum_{j=1}^M\,(A_{ij}(\mathbf{u},\mathbf{p})
   +D_{ij}(\mathbf{u},\mathbf{p}))\,u_j
   -[(\ML-(1-\theta)\,\tau_{n+1}\,{\mathbb L}^n)\,\mathbf{u}^n]_i\,,
\end{align}
where $\mathbf{C}(\mathbf{p})=(C_i(p_i))_{i=1}^M$. Then \eqref{eq1low} with
$c^{n+1}_h$ defined by \eqref{eq2vfmod} is equivalent to
\begin{equation*}
   S_i(\mathbf{u}^{n+1},\mathbf{p}^{n+1})=0\,,\qquad i=1,\dots,M\,.
\end{equation*}
Therefore, defining the operator $P:{\mathbb R}^{2M}\to{\mathbb R}^{2M}$ by
\begin{equation}\label{eq:P}
   P\,\mathbf{U}=(S_1(\mathbf{u},\mathbf{p}),\dots,S_M(\mathbf{u},\mathbf{p}),
                  p_1-P_1(u_1,p_1),\dots,p_M-P_M(u_M,p_M))\qquad
   \forall\,\,\mathbf{U}=(\mathbf{u},\mathbf{p})\in{\mathbb R}^{2M}\,,
\end{equation}
the vectors $\mathbf{u}^{n+1}$, $\mathbf{c}^{n+1}$, $\mathbf{p}^{n+1}$ are a
solution of \eqref{eq1low}, \eqref{eq2vfmod}, \eqref{eq3vf} if and only if
$\mathbf{U}=(\mathbf{u}^{n+1},\mathbf{p}^{n+1})$ satisfies $P\,\mathbf{U}=0$ 
and $\mathbf{c}^{n+1}=\mathbf{C}(\mathbf{p}^{n+1})$.

To show that the equation $P\,\mathbf{U}=0$ has a solution, we will verify the
assumptions of Lemma~\ref{lem:brouwer}. Since it is obvious that the operator
$P$ is continuous, it suffices to investigate the product
$(P\,\mathbf{U},\mathbf{U})$, where $(\cdot,\cdot)$ is the Euclidean inner
product in ${\mathbb R}^{2M}$. We will denote the corresponding norm by 
$\|\cdot\|$. The Euclidean norm in ${\mathbb R}^M$ will be denoted by
$\|\cdot\|_M^{}$. Since the matrix $(D_{ij}(\mathbf{u},\mathbf{p}))_{i,j=1}^M$
is symmetric and has zero row sums and non-positive off-diagonal entries, one
obtains
\begin{equation*}
   \sum_{i,j=1}^M\,u_i\,D_{ij}(\mathbf{u},\mathbf{p})\,u_j
   =-\frac12\,\sum_{i,j=1}^M\,D_{ij}(\mathbf{u},\mathbf{p})(u_i-u_j)^2\ge0\qquad
   \forall\,\,\mathbf{u},\mathbf{p}\in{\mathbb R}^M\,.
\end{equation*}
Furthermore, since $0\le\mathbf{C}(\mathbf{p})\le1$, the expressions
$\big(\phi_j\nabla(\pi_h\mathbf{C}(\mathbf{p})),\nabla\phi_i\big)$ can be
bounded independently of $\mathbf{p}$. Therefore, using the equivalence of
norms on finite-dimensional spaces, one obtains
\begin{equation*}
   \theta\,\tau_{n+1}\,\sum_{i,j=1}^M\,u_i\,(A_{ij}(\mathbf{u},\mathbf{p})
   +D_{ij}(\mathbf{u},\mathbf{p}))\,u_j\ge\theta\,\tau_{n+1}\,\mu\,
   \|\pi_h\mathbf{u}\|_{L^3(\Omega)}^3-C_1\,\|\mathbf{u}\|_M^2\ge
   C_2\,\|\mathbf{u}\|_M^3-C_1\,\|\mathbf{u}\|_M^2\,,
\end{equation*}
where $C_1$ and $C_2$ are positive constants independent of $\mathbf{u}$ and
$\mathbf{p}$. Thus,
\begin{equation}\label{eq:Si_est}
   \sum_{i=1}^M\,u_i\,S_i(\mathbf{u},\mathbf{p})\ge
   C_2\,\|\mathbf{u}\|_M^3-C_1\,\|\mathbf{u}\|_M^2-C_3\,\|\mathbf{u}\|_M^{}\,,
\end{equation}
where $C_3=\|(\ML-(1-\theta)\,\tau_{n+1}\,{\mathbb L}^n)\,\mathbf{u}^n\|_M^{}$.
Finally, using \eqref{eq:P_i_est}, it follows that
\begin{equation*}
   \sum_{i=1}^M\,p_i\,(p_i-P_i(u_i,p_i))\ge
   \|\mathbf{p}\|_M^2-C_4\,\|\mathbf{p}\|_M^{}\,\|\mathbf{u}\|_M^{}
   -C_5\,\|\mathbf{p}\|_M^{}
\end{equation*}
with positive constants $C_4$ and $C_5$ independent of $\mathbf{u}$ and
$\mathbf{p}$. Applying the Young inequality, the previous two inequalities
imply that there exist positive constants $C_6$ and $C_7$ such that
\begin{equation*}
   (P\,\mathbf{U},\mathbf{U})\ge\frac12\,\|\mathbf{U}\|^2+
   C_2\,\|\mathbf{u}\|_M^3-C_6\,\|\mathbf{u}\|_M^2-C_7\ge
   \frac12\,\|\mathbf{U}\|^2-\frac{C_6^3}{C_2^2}-C_7
   \qquad\forall\,\,\mathbf{U}=(\mathbf{u},\mathbf{p})\in{\mathbb R}^{2M}\,.
\end{equation*}
Thus, for any $K>\sqrt{2\,C_6^3/C_2^2+2\,C_7}$, one has
$(P\,\mathbf{U},\mathbf{U})>0$ for any $\mathbf{U}\in{\mathbb R}^{2M}$ with 
$\|\mathbf{U}\|=K$. Therefore, according to Lemma~\ref{lem:brouwer}, there
exists a solution $\mathbf{U}$ of the equation $P\,\mathbf{U}=0$ and hence also
a solution $\mathbf{u}^{n+1}$, $\mathbf{c}^{n+1}$, $\mathbf{p}^{n+1}$ of 
\eqref{eq1low}, \eqref{eq2vfmod}, and \eqref{eq3vf}.

It immediately follows from \eqref{eq2vfmod} that $0\le\mathbf{c}^{n+1}\le1$.
Thus, according to Corollary~\ref{cor:pos_preserv} and
Remark~\ref{rem:pos_preserv}, the solution satisfies $\mathbf{u}^{n+1}\ge0$.
Since \eqref{eq3vf} is equivalent to \eqref{eq3vsf}, one also has
$\mathbf{p}^{n+1}\ge0$ and hence \eqref{eq2vf} is satisfied as well.
\end{proof}

Although the solution of \eqref{eq1low}, \eqref{eq2vf}, \eqref{eq3vf} does not
possess negative values under the time step restrictions \eqref{eq:time_step},
it is usually very inaccurate since too much artificial diffusion is introduced
by the modifications leading to the low-order method \eqref{eq1low},
cf.~Section~\ref{sec:hapto}. Therefore,
in the FEM-FCT methodology, a correction term $\overline{\mathbf{f}}^{n+1}$ is 
added in such a way that the method becomes less diffusive while negative
values are still excluded. This leads to an extension of \eqref{eq1low} in 
the form
\begin{equation*}
   (\ML+\theta\,\tau_{n+1}\,{\mathbb L}^{n+1})\,\mathbf{u}^{n+1}=
   (\ML-(1-\theta)\,\tau_{n+1}\,{\mathbb L}^n)\,\mathbf{u}^n
   +\overline{\mathbf{f}}^{n+1}\,.
\end{equation*}
The high-order method \eqref{eq1m} (with ${\mathbb A}^n$ replaced by 
$\tilde{\mathbb A}^n$) is recovered if
\begin{equation}\label{eq:recover}
   \overline{\mathbf{f}}^{n+1}=
   (\ML-{\mathbb M})(\mathbf{u}^{n+1}-\mathbf{u}^n)
   +\theta\,\tau_{n+1}\,{\mathbb D}^{n+1}\,\mathbf{u}^{n+1}
   +(1-\theta)\,\tau_{n+1}\,{\mathbb D}^n\,\mathbf{u}^n\,.
\end{equation}
Since ${\mathbb D}^n$ has zero row sums, one can write
\begin{equation*}
   ({\mathbb D}^n\,\mathbf{u}^n)_i=\sum_{j=1}^M\,d^n_{ij}\,(u^n_j-u^n_i)\,,
   \qquad i=1,\dots,M\,.
\end{equation*}
For the terms with the matrices ${\mathbb D}^{n+1}$ and $\ML-{\mathbb M}$
(which also have zero row sums), one can proceed analogously and hence
\eqref{eq:recover} holds if an only if
\begin{equation*}
   \overline{\mathbf{f}}^{n+1}=\Bigg(\sum_{j=1}^M\,f^{n+1}_{ij}\Bigg)_{i=1}^M\,,
\end{equation*}
where the algebraic fluxes $f^{n+1}_{ij}$ are given by
\begin{equation}\label{eq:fluxes}
   f^{n+1}_{ij}=-m_{ij}\,(u_j^{n+1}-u_i^{n+1})+m_{ij}\,(u^n_j-u^n_i)
   +\theta\,\tau_{n+1}\,d^{n+1}_{ij}\,(u_j^{n+1}-u_i^{n+1})
   +(1-\theta)\,\tau_{n+1}\,d^n_{ij}\,(u^n_j-u^n_i)\,.
\end{equation}
Because $\mathbb M$, ${\mathbb D}^{n+1}$, and ${\mathbb D}^n$ are symmetric
matrices, one has $f^{n+1}_{ij} = - f^{n+1}_{ji}$. Note also that the fluxes 
depend on (unknown) values of the approximate solution at time level $t_{n+1}$.

Now, the idea of the FCT approach is to limit the fluxes $f^{n+1}_{ij}$ by
solution dependent correction factors $\alpha^{n+1}_{ij}\in[0,1]$ called
limiters so that the non-negativity of the approximate solution can be
guaranteed but less artificial diffusion is introduced than in case of the
low-order method. This leads to the discrete problem
\begin{equation}\label{eq1fct}
   (\ML+\theta\,\tau_{n+1}\,{\mathbb L}^{n+1})\,\mathbf{u}^{n+1}=
   (\ML-(1-\theta)\,\tau_{n+1}\,{\mathbb L}^n)\,\mathbf{u}^n
   +\Bigg(\sum_{j=1}^M\,\alpha^{n+1}_{ij}\,f^{n+1}_{ij}\Bigg)_{i=1}^M\,.
\end{equation}
The original Galerkin discretization is recovered for $\alpha_{ij} = 1$ while
the largest amount of artificial diffusion is introduced for $\alpha_{ij} = 0$.
The latter setting is appropriate in the neighborhood of steep fronts and large
gradients. The artificial diffusion can be removed in regions where the
solution is smooth and where non-positive off-diagonal entries of the stiffness
matrix do not pose any threat to non-negativity. The corrected fluxes depend on
the approximate solution in a nonlinear way but since the problem in here is
already nonlinear, we can treat both nonlinearities simultaneously.

It is convenient to write the nonlinear problem \eqref{eq1fct} in the form
\begin{eqnarray}
   \ML\,\overline{\mathbf{u}}
   &=&(\ML-(1-\theta)\,\tau_{n+1}\,{\mathbb L}^n)\,\mathbf{u}^n\,,
   \label{eq:step1}\\
   \ML\,\tilde{\mathbf{u}}&=&\ML\,\overline{\mathbf{u}}
   +\Bigg(\sum_{j=1}^M\,\alpha^{n+1}_{ij}\,f^{n+1}_{ij}\Bigg)_{i=1}^M\,,
   \label{eq:step2}\\
   (\ML+\theta\,\tau_{n+1}\,{\mathbb L}^{n+1})\,\mathbf{u}^{n+1}
   &=&\ML\,\tilde{\mathbf{u}}\,.\label{eq:step3}
\end{eqnarray}
According to Lemma~\ref{lem:time_step}, the steps \eqref{eq:step1} and
\eqref{eq:step3} are positivity preserving under the conditions
\eqref{eq:time_step}. To guarantee the positivity preservation of the second
step, the limiters $\alpha^{n+1}_{ij}$ have to be defined appropriately. We
will apply the Zalesak algorithm \cite{Zal79} which will be described next.

The solution of the nonlinear problem \eqref{eq:step1}--\eqref{eq:step3} is
computed by fixed-point iterations where the algebraic fluxes are calculated
using the previous iterate. Since the properties of the Zalesak algorithm do
not depend on the form of these fluxes, we will denote them simply by $f_{ij}$.
Then, the aim is to find limiters $\alpha_{ij}\in[0,1]$ such that the solution
$\tilde{\mathbf{u}}$ of
\begin{equation*}
   \ML\,\tilde{\mathbf{u}}=\ML\,\overline{\mathbf{u}}
   +\Bigg(\sum_{j=1}^M\,\alpha_{ij}\,f_{ij}\Bigg)_{i=1}^M
\end{equation*}
satisfies
\begin{equation}\label{eq:tcd_femfct_01}
   \overline u_i^{\mathrm{min}}\le\tilde u_i\le\overline u_i^{\mathrm{max}}\,,
   \qquad i=1,\dots,M\,,
\end{equation}
where
\begin{equation*}
   \overline u_i^{\mathrm{min}} = \min_{j\in{\mathcal N}_i\cup\{i\}} \overline u_j\,,
   \quad 
   \overline u_i^{\mathrm{max}} = \max_{j\in{\mathcal N}_i\cup\{i\}} \overline u_j\,,
   \qquad i=1,\dots,M\,,
\end{equation*}
and ${\mathcal N}_i$ is the index set of neighbour vertices to the vertex $x_i$
(note that two vertices of the triangulation ${\mathcal T}_h$ are called
neighbouring if they are contained in the same mesh cell). To preserve
conservativity, it is important that the limiters $\alpha_{ij}$ form a
symmetric matrix. The limiting process begins with cancelling all fluxes that
are diffusive in nature and tend to flatten the solution profiles,
cf.~\cite{Kuz12a}. The required modification is
\begin{equation}\label{eq:prelimiting}
   f_{ij}:=0\qquad\mbox{if}\,\,\,\,
   f_{ij}\,(\overline{u}_j-\overline{u}_i) > 0\,.
\end{equation}
The remaining fluxes are truly antidiffusive and the computation of 
$\alpha_{ij}$ involves the following steps:
\begin{itemize}
\item[1.]
Compute the sum of positive/negative antidiffusive fluxes into node $i$
\begin{equation}\label{eq:zalesak1}
   P_{i}^{+} = \sum_{j\in{\mathcal N}_i}\,\max \{0, {f}_{ij}\}\,,\qquad 
   P_{i}^{-} = \sum_{j\in{\mathcal N}_i}\,\min \{0, {f}_{ij}\}\,.
\end{equation}
\item[2.]
Compute the distance to a local extremum of the auxiliary solution 
$\overline{\mathbf{u}}$
\begin{equation}\label{eq:zalesak2}
   Q_i^+=m_i\,(\overline u_i^{\mathrm{max}}-\overline u_i) \,,\qquad
   Q_i^-=m_i\,(\overline u_i^{\mathrm{min}}-\overline u_i) \,.
\end{equation}
\item[3.] 
Compute the nodal correction factors for the net increment to node $i$
\begin{equation}\label{eq:zalesak3}
   R_i^+=\min\left\lbrace 1,\dfrac{Q_i^+}{ P_i^+}\right\rbrace,\qquad
   R_i^-=\min\left\lbrace 1,\dfrac{Q_i^-}{ P_i^-}\right\rbrace.
\end{equation}
If a denominator is zero, set the respective value of $R_i^+$ or $R_i^-$ equal 
to $1$.
\item[4.]
Check the sign of the antidiffusive flux and define the correction factor by
\begin{equation}\label{eq:zalesak4}
   \alpha_{ij} = \left\{
   \begin{array}{ll}
      \min\{R_{i} ^{+}, R_{j} ^{-} \}\quad&\mbox{if}\quad\,f_{ij}>0\,,\\[1mm]
      1&\mbox{if}\quad\,f_{ij} = 0\,, \\[1mm]
      \min\{R_{i} ^{-}, R_{j} ^{+} \}\quad&\mbox{if}\quad\,f_{ij}<0\,.
   \end{array}\right.
\end{equation}
\end{itemize}
It can be easily verified (see, e.g., \cite{BJK23}) that this algorithm leads 
to the property \eqref{eq:tcd_femfct_01}.

Now we are in a position to prove the solvability and positivity preservation
for the above FCT discretization.

\begin{theorem}\label{thm:fct_existence}
Consider any $n\in\{0,\dots,N-1\}$ and let
$\mathbf{u}^n,\mathbf{c}^n,\mathbf{p}^n\in{\mathbb R}^M$ satisfy
$\mathbf{u}^n\ge0$, $1\ge\mathbf{c}^n\ge0$, $\mathbf{p}^n\ge0$. Let the time 
step $\tau_{n+1}$ satisfy the conditions \eqref{eq:time_step3}.
Then there exist vectors $\mathbf{u}^{n+1},\mathbf{c}^{n+1}, 
\mathbf{p}^{n+1}\in{\mathbb R}^M$ satisfying \eqref{eq1fct}, \eqref{eq2vf}, 
\eqref{eq3vf} where the fluxes $f^{n+1}_{ij}$ are given by \eqref{eq:fluxes}
and \eqref{eq:prelimiting} and the limiters $\alpha^{n+1}_{ij}$ are computed
using the Zalesak algorithm \eqref{eq:zalesak1}--\eqref{eq:zalesak4} from the
fluxes $f^{n+1}_{ij}$. Moreover, these vectors satisfy $\mathbf{u}^{n+1}\ge0$,
$1\ge\mathbf{c}^{n+1}\ge0$, and $\mathbf{p}^{n+1}\ge0$.
\end{theorem}

\begin{proof}
The proof follows the lines of that of Theorem~\ref{thm:low_order_existence}.
Thus, we again start with replacing \eqref{eq2vf} by \eqref{eq2vfmod}. We again
define $C_i$, $P_i$, $A_{ij}$, and $D_{ij}$ by \eqref{eq:Ci}, \eqref{eq:Pi},
\eqref{eq:Aij} and \eqref{eq:Dij}, respectively, whereas $S_i$ are now defined
by
\begin{align*}
   &S_i(\mathbf{u},\mathbf{p})=m_i\,u_i
   +\theta\,\tau_{n+1}\,\sum_{j=1}^M\,(A_{ij}(\mathbf{u},\mathbf{p})
   +D_{ij}(\mathbf{u},\mathbf{p}))\,u_j-\sum_{j=1}^M\,
   \alpha_{ij}(\mathbf{u},\mathbf{p})\,\tilde f_{ij}(\mathbf{u},\mathbf{p})\\
   &\hspace*{85mm}
   -[(\ML-(1-\theta)\,\tau_{n+1}\,{\mathbb L}^n)\,\mathbf{u}^n]_i\,,
\end{align*}
where $\alpha_{ij}(\mathbf{u},\mathbf{p})$ are defined by the Zalesak algorithm
\eqref{eq:zalesak1}--\eqref{eq:zalesak4} for the algebraic fluxes 
$\tilde f_{ij}(\mathbf{u},\mathbf{p})$ defined by
\begin{equation*}
   \tilde f_{ij}(\mathbf{u},\mathbf{p}) = \left\{
   \begin{array}{ll}
      f_{ij}(\mathbf{u},\mathbf{p})\quad&\mbox{if}\quad\,
      f_{ij}(\mathbf{u},\mathbf{p})(\overline{u}_j-\overline{u}_i)\le0\,,\\[1mm]
      0&\mbox{if}\quad\,
      f_{ij}(\mathbf{u},\mathbf{p})(\overline{u}_j-\overline{u}_i)>0\,,
   \end{array}\right.
\end{equation*}
with $\overline{\mathbf{u}}$ from \eqref{eq:step1} and
\begin{equation*}
   f_{ij}(\mathbf{u},\mathbf{p})
   =\big(-m_{ij}+\theta\,\tau_{n+1}\,D_{ij}(\mathbf{u},\mathbf{p})\big)(u_j-u_i)
   +\big(m_{ij}+(1-\theta)\,\tau_{n+1}\,d^n_{ij}\big)(u^n_j-u^n_i)\,.
\end{equation*}
Then, defining the operator $P:{\mathbb R}^{2M}\to{\mathbb R}^{2M}$ by
\eqref{eq:P}, the vectors $\mathbf{u}^{n+1}$, $\mathbf{c}^{n+1}$,
$\mathbf{p}^{n+1}$ are a solution of \eqref{eq1fct}, \eqref{eq2vfmod},
\eqref{eq3vf} if and only if $\mathbf{U}=(\mathbf{u}^{n+1},\mathbf{p}^{n+1})$
satisfies $P\,\mathbf{U}=0$ and
$\mathbf{c}^{n+1}=\mathbf{C}(\mathbf{p}^{n+1})$.

The solvability of the equation $P\,\mathbf{U}=0$ will be again proved using
Lemma~\ref{lem:brouwer}. To show the continuity of the operator $P$ at any point
$\widetilde{\mathbf{U}}\equiv(\widetilde{\mathbf{u}},\widetilde{\mathbf{p}})
\in{\mathbb R}^{2M}$, it suffices to consider the terms
$\alpha_{ij}(\mathbf{u},\mathbf{p})\,\tilde f_{ij}(\mathbf{u},\mathbf{p})$ 
since the remaining terms in the definition of $P$ are clearly continuous.
Moreover, $f_{ij}$ and hence also $\tilde f_{ij}$ are continuous. Thus, if
$\tilde f_{ij}(\widetilde{\mathbf{U}})\neq0$, then the denominators in the 
formulas defining $\alpha_{ij}(\mathbf{U})$ with 
$\mathbf{U}=(\mathbf{u},\mathbf{p})$ do not vanish in a neighborhood of
$\widetilde{\mathbf{U}}$ and hence $\alpha_{ij}$ is continuous at
$\widetilde{\mathbf{U}}$. Consequently, also $\alpha_{ij}\,\tilde f_{ij}$ is
continuous at $\widetilde{\mathbf{U}}$. If 
$\tilde f_{ij}(\widetilde{\mathbf{U}})=0$, then
\begin{equation*}
   |(\alpha_{ij}\,\tilde f_{ij})(\mathbf{U})
   -(\alpha_{ij}\,\tilde f_{ij})(\widetilde{\mathbf{U}})|
   =|(\alpha_{ij}\,\tilde f_{ij})(\mathbf{U})|
   \le|\tilde f_{ij}(\mathbf{U})|
   =|\tilde f_{ij}(\mathbf{U})-\tilde f_{ij}(\widetilde{\mathbf{U}})|\,,
\end{equation*}
which shows that $\alpha_{ij}\,\tilde f_{ij}$ is again continuous at 
$\widetilde{\mathbf{U}}$.

To estimate $(P\,\mathbf{U},\mathbf{U})$ from below, let us denote
\begin{equation*}
   \widetilde\alpha_{ij}(\mathbf{u},\mathbf{p}) = \left\{
   \begin{array}{ll}
      \alpha_{ij}(\mathbf{u},\mathbf{p})\quad&\mbox{if}\quad\,
      f_{ij}(\mathbf{u},\mathbf{p})(\overline{u}_j-\overline{u}_i)\le0\,,\\[1mm]
      0&\mbox{if}\quad\,
      f_{ij}(\mathbf{u},\mathbf{p})(\overline{u}_j-\overline{u}_i)>0\,.
   \end{array}\right.
\end{equation*}
Then $\widetilde\alpha_{ij}$ again form a symmetric matrix and 
$\alpha_{ij}\,\tilde f_{ij}=\widetilde\alpha_{ij}\,f_{ij}$. Therefore, 
$S_i(\mathbf{u},\mathbf{p})$ can be written in the form
\begin{align*}
   S_i(\mathbf{u},\mathbf{p})=&\sum_{j=1}^M\,m_{ij}\,u_j
   +\theta\,\tau_{n+1}\,\sum_{j=1}^M\,A_{ij}(\mathbf{u},\mathbf{p})\,u_j\\
   +&\sum_{j=1}^M\,\big(1-\widetilde\alpha_{ij}(\mathbf{u},\mathbf{p})\big)
    \big(-m_{ij}+\theta\,\tau_{n+1}\,D_{ij}(\mathbf{u},\mathbf{p})\big)
    (u_j-u_i)\\
   +&\sum_{j=1}^M\,\big(1-\widetilde\alpha_{ij}(\mathbf{u},\mathbf{p})\big)
   \big(m_{ij}+(1-\theta)\,\tau_{n+1}\,d^n_{ij}\big)(u^n_j-u^n_i)\\
   -&[({\mathbb M}-(1-\theta)\,\tau_{n+1}\,{\mathbb A}^n)\,\mathbf{u}^n]_i\,.
\end{align*}
Denoting $B_{ij}=\big(1-\widetilde\alpha_{ij}(\mathbf{u},\mathbf{p})\big)
\big(-m_{ij}+\theta\,\tau_{n+1}\,D_{ij}(\mathbf{u},\mathbf{p})\big)$, one has
\begin{equation*}
   \sum_{i,j=1}^M\,u_i\,\big(1-\widetilde\alpha_{ij}(\mathbf{u},\mathbf{p})\big)
   \big(-m_{ij}+\theta\,\tau_{n+1}\,D_{ij}(\mathbf{u},\mathbf{p})\big)
   (u_j-u_i)=-\frac12\,\sum_{i,j=1}^M\,B_{ij}\,(u_i-u_j)^2\ge0\,,
\end{equation*}
since the matrix $(B_{ij})_{i,j=1}^M$ is symmetric and has non-positive
off-diagonal entries. Therefore, one again obtains \eqref{eq:Si_est} where the
constants $C_1$, $C_2$ are the same as in the proof of
Theorem~\ref{thm:low_order_existence} and
\begin{equation*}
   C_3=\|\mathbf{g}\|_M^{}+
   \|({\mathbb M}-(1-\theta)\,\tau_{n+1}\,{\mathbb A}^n)\,\mathbf{u}^n\|_M^{}\,,
\end{equation*}
where
\begin{equation*}
   g_i=\sum_{j=1}^M\,|m_{ij}+(1-\theta)\,\tau_{n+1}\,d^n_{ij}|\,|u^n_j-u^n_i|\,,
   \qquad i=1,\dots,M\,.
\end{equation*}
Thus, in the same way as in the proof of Theorem~\ref{thm:low_order_existence},
one concludes that there exists a solution $\mathbf{U}$ of the equation
$P\,\mathbf{U}=0$ and hence also a solution $\mathbf{u}^{n+1}$,
$\mathbf{c}^{n+1}$, $\mathbf{p}^{n+1}$ of \eqref{eq1fct}, \eqref{eq2vfmod}, and
\eqref{eq3vf}.

To prove the positivity preservation, we write \eqref{eq1fct} in the form
\eqref{eq:step1}--\eqref{eq:step3}. Since
$\ML-(1-\theta)\,\tau_{n+1}\,{\mathbb L}^n\ge0$ according to 
Lemma~\ref{lem:time_step}, one has $\overline{\mathbf{u}}\ge0$. Applying 
\eqref{eq:tcd_femfct_01}, one gets $\tilde{\mathbf{u}}\ge0$. Since
$0\le\mathbf{c}^{n+1}\le1$ due to \eqref{eq2vfmod}, it follows from
Lemma~\ref{lem:time_step} and Remark~\ref{rem:pos_preserv} that the matrix
$\ML+\theta\,\tau_{n+1}\,{\mathbb L}^{n+1}$ is an M-matrix. Consequently,
$\mathbf{u}^{n+1}\ge0$ in view of \eqref{eq:step3}. Since \eqref{eq3vf} is
equivalent to \eqref{eq3vsf}, one also has $\mathbf{p}^{n+1}\ge0$ and hence
\eqref{eq2vf} is satisfied as well.
\end{proof}

\section{Iterative solution of the FCT discretization}
\label{sec:iter_sol}

To compute a solution of the nonlinear problem \eqref{eq1fct}, \eqref{eq2vf},
\eqref{eq3vf} at time $t_{n+1}$, we will proceed similarly as for the Galerkin
discretization in Section~\ref{sec:galerkin}. Thus, given approximations
$\mathbf{u}^{n+1}_{k-1}$, $\mathbf{c}^{n+1}_{k-1}$, $\mathbf{p}^{n+1}_{k-1}$
(with some $k>0$) of $\mathbf{u}^{n+1}$, $\mathbf{c}^{n+1}$,
$\mathbf{p}^{n+1}$, respectively, we compute $\mathbf{c}^{n+1}_k$, 
$\mathbf{p}^{n+1}_k$ using \eqref{eq:cik}, \eqref{eq:pik}. The iterate 
$\mathbf{u}^{n+1}_k$ is computed by solving the linear system
\begin{equation}\label{eq:fctk}
   (\ML+\theta\,\tau_{n+1}\,{\mathbb L}^{n+1}_{k-1})\,\mathbf{u}^{n+1}_k
   =(\ML-(1-\theta)\,\tau_{n+1}\,{\mathbb L}^n)\,\mathbf{u}^n
   +\Bigg(\sum_{j=1}^M\,\alpha^{n+1}_{ij,k-1}\,f^{n+1}_{ij,k-1}\Bigg)_{i=1}^M\,,
\end{equation}
where ${\mathbb L}^{n+1}_{k-1}={\mathbb A}^{n+1}_{k-1}+{\mathbb D}^{n+1}_{k-1}$
with the matrix ${\mathbb A}^{n+1}_{k-1}$ defined in \eqref{eq:akm1} and the 
artificial diffusion matrix ${\mathbb D}^{n+1}_{k-1}$ defined by
\begin{equation}\label{eq:dijk}
  d^{n+1}_{ij,k-1}
  =-\max\{a^{n+1}_{ij,k-1},0,a^{n+1}_{ji,k-1}\}\quad
  \mbox{for} \ i\ne j\,,\qquad\quad
  d^{n+1}_{ii,k-1}=-\sum_{j=1,j \neq i}^M d^{n+1}_{ij,k-1}\,.
\end{equation}
The algebraic fluxes $f^{n+1}_{ij,k-1}$ are given by
\begin{equation}\label{eq:fluxesk}
   f^{n+1}_{ij,k-1}=\big(-m_{ij}+\theta\,\tau_{n+1}\,d^{n+1}_{ij,k-1}\big)
   (u_{j,k-1}^{n+1}-u_{i,k-1}^{n+1})
   +\big(m_{ij}+(1-\theta)\,\tau_{n+1}\,d^n_{ij}\big)(u^n_j-u^n_i)
\end{equation}
and we again consider the prelimiting step
\begin{equation}\label{eq:prelimitingk}
   f^{n+1}_{ij,k-1}:=0\qquad\mbox{if}\,\,\,\,
   f^{n+1}_{ij,k-1}\,(\overline{u}_j-\overline{u}_i) > 0\,,
\end{equation}
with $\overline{\mathbf{u}}$ from \eqref{eq:step1}. The limiters
$\alpha^{n+1}_{ij,k-1}$ are computed from the fluxes $f^{n+1}_{ij,k-1}$ using 
the Zalesak algorithm \eqref{eq:zalesak1}--\eqref{eq:zalesak4}. The following
result shows that, under suitable time step restrictions, the above-defined
iterates are uniquely determined and preserve non-negativity. This is important
since, in practice, the fixed-point iterations are usually terminated when a
stopping criterion is met, i.e., typically before reaching the solution of the 
nonlinear problem \eqref{eq1fct}, \eqref{eq2vf}, \eqref{eq3vf}.

\begin{theorem}\label{thm:iterates}
Consider any $n\in\{0,\dots,N-1\}$ and $k\in{\mathbb N}$ and let
$\mathbf{u}^n,\mathbf{c}^n,\mathbf{p}^n\in{\mathbb R}^M$ and 
$\mathbf{u}^{n+1}_{k-1},\mathbf{p}^{n+1}_{k-1}\in{\mathbb R}^M$ be arbitrary 
vectors satisfying $\mathbf{u}^n\ge0$, $1\ge\mathbf{c}^n\ge0$,
$\mathbf{p}^n\ge0$, and $\mathbf{u}^{n+1}_{k-1}\ge0$,
$\mathbf{p}^{n+1}_{k-1}\ge0$. Let $\mathbf{c}^{n+1}_k$, $\mathbf{p}^{n+1}_k$ be
given by \eqref{eq:cik}, \eqref{eq:pik}. Let the time step $\tau_{n+1}$ satisfy
the conditions
\begin{equation}\label{eq:time_step4}
   (1-\theta)\,\tau_{n+1}\,l^n_{ii}\le m_i\,,\qquad
   \theta\,\tau_{n+1}
   \Big(\mu\,\big(1 - \pi_h\mathbf{u}^{n+1}_{k-1},\phi_i\big)
   +\chi\,\big(\nabla(\pi_h\mathbf{c}^{n+1}_k),\nabla\phi_i\big)\Big)
   < m_i\,,\qquad i=1,\dots,M\,.
\end{equation}
Then the linear system \eqref{eq:fctk} has a unique solution
$\mathbf{u}^{n+1}_k$ and one has $\mathbf{u}^{n+1}_k\ge0$, 
$1\ge\mathbf{c}^{n+1}_k\ge0$, and $\mathbf{p}^{n+1}_k\ge0$.
\end{theorem}

\begin{proof}
The formula \eqref{eq:cik} immediately implies that
$1\ge\mathbf{c}^{n+1}_k\ge0$. Since \eqref{eq:pik} can be written in the form
\eqref{eq3vsf} with $u_{h,\tau}$ and $c_{h,\tau}$ defined using
$\mathbf{u}^{n+1}_{k-1}$ and $\mathbf{c}^{n+1}_k$, respectively, at time
$t_{n+1}$, one has $\mathbf{p}^{n+1}_k\ge0$. Since
$\ML-(1-\theta)\,\tau_{n+1}\,{\mathbb L}^n\ge0$ according to
Lemma~\ref{lem:time_step}, the solution of \eqref{eq:step1} satisfies
$\overline{\mathbf{u}}\ge0$. Then \eqref{eq:tcd_femfct_01} implies
$\tilde{\mathbf{u}}\ge0$ for the solution of
\begin{equation*}
   \ML\,\tilde{\mathbf{u}}=\ML\,\overline{\mathbf{u}}
   +\Bigg(\sum_{j=1}^M\,\alpha^{n+1}_{ij,k-1}\,f^{n+1}_{ij,k-1}\Bigg)_{i=1}^M\,.
\end{equation*}
Finally, we use the fact that $\mathbf{u}^{n+1}_k$ satisfies
\begin{equation}\label{eq:step3k}
   (\ML+\theta\,\tau_{n+1}\,{\mathbb L}^{n+1}_{k-1})\,\mathbf{u}^{n+1}_k
   =\ML\,\tilde{\mathbf{u}}\,.
\end{equation}
It follows from the proof of Lemma~\ref{lem:time_step} that, under the second
condition in \eqref{eq:time_step4}, the matrix
$(\ML+\theta\,\tau_{n+1}\,{\mathbb L}^{n+1}_{k-1})$ is an M-matrix and hence
$\mathbf{u}^{n+1}_k$ is uniquely determined and satisfies
$\mathbf{u}^{n+1}_k\ge0$.
\end{proof}

\begin{remark}
From the physical point of view, the quantities $u$, $c$, and $p$ should be not
only non-negative but also bounded by $1$ from above. We have proved that this
is the case for the approximations of $c$. Moreover, if this would be true 
also for the approximations of $u$, the integral form \eqref{eq3vd} would
provide this property also for the approximations of $p$. Unfortunately, a
proof of the upper bound for the approximations of $u$ is not available and
numerical results suggest that this bound can be violated. Note that a standard
proof of upper bounds for FCT discretizations relies on the decomposition
\eqref{eq:step1}--\eqref{eq:step3}. Then, in particular, one would need that
the solution of \eqref{eq:step1} satisfies $\overline{\mathbf{u}}\le1$ if
${\mathbf{u}}^n\le1$. Choosing ${\mathbf{u}}^n={\mathbf 1}$ (a vector with all
components equal to~$1$), this requirement implies that 
$(1-\theta)\,{\mathbb A}^n\,{\mathbf 1}
=(1-\theta)\,{\mathbb L}^n\,{\mathbf 1}\ge0$, i.e., the row sums of the matrix
$(1-\theta)\,{\mathbb A}^n$ have to be non-negative. Similarly, to derive an
upper bound from \eqref{eq:step3k}, one would need that $\theta\,{\mathbb
A}^{n+1}_{k-1}\,{\mathbf 1}\ge0$. It is clear that the validity of these row
sum conditions cannot be expected.
\end{remark}

\begin{remark}
The second condition on $\tau_{n+1}$ in \eqref{eq:time_step4} depends on
${\mathbf{c}}^{n+1}_k$ which itself depends on $\tau_{n+1}$. Consequently, in
general, one has to proceed iteratively to find $\tau_{n+1}$ which satisfies
\eqref{eq:time_step4}. To avoid this and also the dependence of $\tau_{n+1}$ on
the fixed-point iteration index $k$, it is possible to replace
\eqref{eq:time_step4} by \eqref{eq:time_step3},
cf.~Remark~\ref{rem:pos_preserv}.
\end{remark}

We summarize the procedure for obtaining a high-resolution positivity
preserving scheme for solving \eqref{eq1}--\eqref{eq5} in
Algorithm~\ref{alg:nonlin_femfct}.

\begin{algorithm}[h!]\caption{Iterative scheme for computing an approximation
of the solution to the nonlinear FCT discretization.}\label{alg:nonlin_femfct}
\begin{algorithmic}[1]
   \STATE{Choose a tolerance $\mathrm{Tol}>0$ and a damping factor
          $\beta\in(0,1]$.}
   \STATE{Compute the initial values $\mathbf{c}^0$, $\mathbf{p}^0$, and
          $\mathbf{u}^0$ by \eqref{eq:init_cond}.}
   \STATE{Compute the mass matrix $\mathbb M$ and the lumped mass matrix $\ML$.}
   \FOR{$n=0,1,\ldots,N-1$}
      \STATE{Compute the stiffness matrix ${\mathbb A}^n$ and the 
             artificial diffusion matrix ${\mathbb D}^n$ and set 
             ${\mathbb L}^n={\mathbb A}^n+{\mathbb D}^n$.}
      \STATE{Choose $\tau_{n+1}$ satisfying \eqref{eq:time_step3}.}
      \STATE{Compute the intermediate solution $\overline{\mathbf{u}}$ from
                 \eqref{eq:step1}.} 
      \STATE{Set $\mathbf{c}^{n+1}_0=\mathbf{c}^n$, 
               $\mathbf{p}^{n+1}_0=\mathbf{p}^n$, and 
                                  $\mathbf{u}^{n+1}_0=\mathbf{u}^n$.}
      \FOR{$k=1,2,\ldots$}
         \STATE{Compute $\mathbf{c}^{n+1}_k$ from \eqref{eq:cik} using
                $\mathbf{c}^{n}, \mathbf{p}^{n}$ and $\mathbf{p}_{k-1}^{n+1}$.}
         \STATE{Compute $\mathbf{p}^{n+1}_k$ from \eqref{eq:pik} using 
                $\mathbf{p}^n, \mathbf{c}^n, \mathbf{u}^n $, 
                $\mathbf{c}^{n+1}_k$, and $\mathbf{u}_{k-1}^{n+1}$.}
         \STATE{Compute the stiffness ${\mathbb A}^{n+1}_{k-1}$ from 
                  \eqref{eq:akm1} using $\mathbf{c}^{n+1}_k$ and 
                  $\mathbf{u}_{k-1}^{n+1}$.}
         \STATE{Compute the artificial diffusion matrix 
                ${\mathbb D}^{n+1}_{k-1}$ from \eqref{eq:dijk} and set
                ${\mathbb L}^{n+1}_{k-1}
                 ={\mathbb A}^{n+1}_{k-1}+{\mathbb D}^{n+1}_{k-1}$.}
         \STATE{Compute the algebraic fluxes $f^{n+1}_{ij,k-1}$ from
                 \eqref{eq:fluxesk} and \eqref{eq:prelimitingk}.}
         \STATE{Compute the limiters $\alpha^{n+1}_{ij,k-1}$ by the Zalesak
                algorithm \eqref{eq:zalesak1}--\eqref{eq:zalesak4} using the 
                fluxes $f^{n+1}_{ij,k-1}$ and the intermediate solution 
                $\overline{\mathbf{u}}$.}
         \STATE{Compute $\mathbf{u}^{n+1}_k$ by solving the linear system
                \eqref{eq:fctk}.}
         \IF{$\max\left\lbrace \Vert \mathbf{c}_{k}^{n+1} -
            \mathbf{c}_{k-1}^{n+1}\Vert_M^{}, \Vert \mathbf{p}_{k}^{n+1} -
            \mathbf{p}_{k-1}^{n+1}\Vert_M^{}, \Vert \mathbf{u}_{k}^{n+1} -
            \mathbf{u}_{k-1}^{n+1}\Vert_M^{}\right\rbrace  
            < \mathrm{Tol}$} 
            \STATE{Go to line 23.}
         \ELSE
         \STATE{Set $\mathbf{c}_{k}^{n+1}:=\beta\,\mathbf{c}_{k}^{n+1} 
                     + (1-\beta)\,\mathbf{c}_{k-1}^{n+1}\,,\quad
                     \mathbf{p}_{k}^{n+1}:=\beta\,\mathbf{p}_{k}^{n+1} 
                     + (1-\beta)\,\mathbf{p}_{k-1}^{n+1}\,,\quad
                     \mathbf{u}_{k}^{n+1}:=\beta\,\mathbf{u}_{k}^{n+1} 
                     + (1-\beta)\,\mathbf{u}_{k-1}^{n+1}$.}
         \ENDIF
      \ENDFOR
      \STATE{Set $\mathbf{c}^{n+1} = \mathbf{c}_{k}^{n+1}\,, \quad 
                  \mathbf{p}^{n+1} = \mathbf{p}_{k}^{n+1}\,, \quad 
                  \mathbf{u}^{n+1} = \mathbf{u}_{k}^{n+1}$.}
   \ENDFOR
\end{algorithmic}
\end{algorithm}

\section{Numerical results}
\label{sec:numerics}
In the following, we present several numerical experiments to verify the
positivity preserving properties of the proposed scheme for the model
(\ref{eq1})--(\ref{eq5}).

The computations are performed on a square domain $\Omega = (0 , 20)^2$ which
is decomposed into quadrilateral mesh cells obtained by uniform refinements.
Precisely, after $r$ refinements, the triangulation ${\mathcal T}_h$ consists
of $2^{2\,r}$ equal squares. If not otherwise stated, we consider five
refinements, i.e., ${\mathcal T}_h$ consists of $32\times32$ mesh cells. As
explained above, conforming bilinear finite elements are used for approximating
all unknown variables. The final time is $T=50$ and the parameter 
$\epsilon = 0.2$ is used. The values of the remaining parameters of the model
will be specified for the particular computations. The initial conditions are
defined by
\begin{equation*}
   u^0(x) = {\mathrm e}^{-|x|^2}\,,\qquad 
   c^0(x) = 1 - \frac12\,{\mathrm e}^{-|x|^2}\,, \qquad 
   p^0{(x)} = \frac12\,{\mathrm e}^{-|x|^2}\,.
\end{equation*}
If not otherwise stated, we apply the A-stable Crank-Nicolson method
corresponding to $\theta = 0.5$ for the time discretization. In one case, we
will also discuss the application of the unconditionally stable backward Euler 
method corresponding to $\theta = 1$. Algorithm~\ref{alg:nonlin_femfct} is used
with the tolerance $\mathrm{Tol}=10^{-8}$
and the damping factor $\beta=0.5$.
The linear system \eqref{eq:fctk} is solved using the sparse direct solver
UMFPACK \cite{Dav04}. Our newly developed algorithms are implemented in the
open-source finite element library deal.II \cite{deal2020,dealII94}.

\subsection{Comparison between the standard Galerkin FEM and the FEM-FCT
scheme in presence of diffusion}

To begin with, in the first example we consider the modified model subjected to
an extra diffusion term in the equation \eqref{eq1} with diffusion coefficient
$\alpha^{-1}$, as considered in \cite{fuest2022global}, i.e., the equation
\eqref{eq1} is replaced by \eqref{eq6}. We consider $\alpha=10$, $\chi=1$, and
$\mu=1$. As can be seen from Figs.~\ref{fig1} and \ref{fig3}, the FEM-FCT
scheme introduces slightly more artificial diffusion than the standard Galerkin
FEM. One can observe that the cancer cells invade the extracellular matrix
and occupy the whole domain completely at the final time. Next, we decrease the
amount of the diffusion by setting $\alpha=1000$ and keep the proliferation and
haptotaxis rate as before. As can be seen from Figs.~\ref{fig5} and \ref{fig6},
the standard Galerkin FEM shows some oscillations in the front layer and the
numerical simulation breaks down when the solution reaches the boundary of the
computational domain, whereas applying the FEM-FCT removes the oscillations and
keeps the solution positive at all times. The corresponding snapshots of the
cancer cell density, extracellular matrix, and protease are plotted along the
line $y=x$ in Figs. \ref{fig2}, \ref{fig4}, \ref{fig6}, and \ref{fig8}.

\begin{figure}[H]
	\centering
	\begin{subfigure}{0.24\textwidth}
		\includegraphics[width=\textwidth]{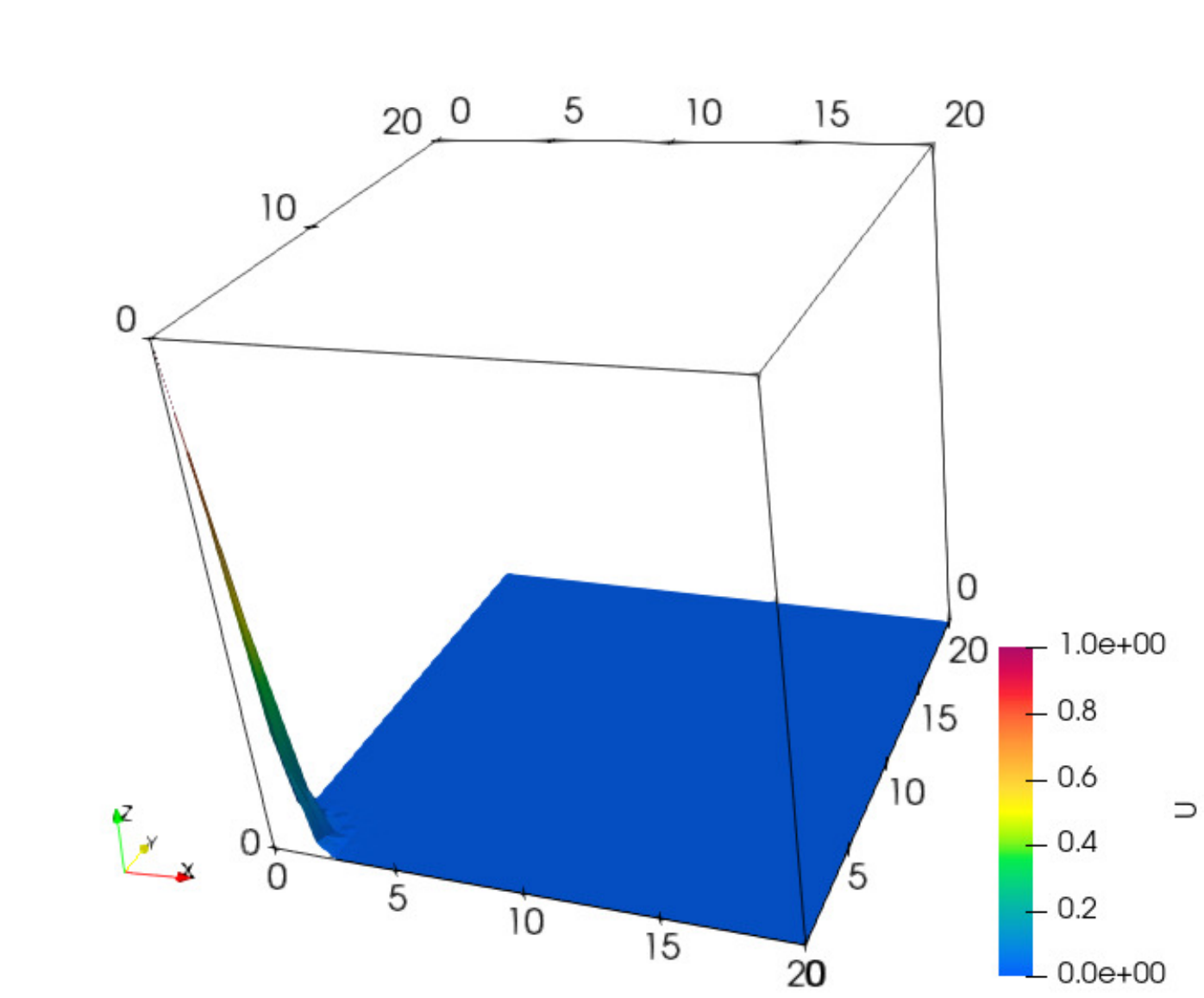}
		\caption{t = 0}
	\end{subfigure}
	\begin{subfigure}{0.24\textwidth}
		\includegraphics[width=\textwidth]{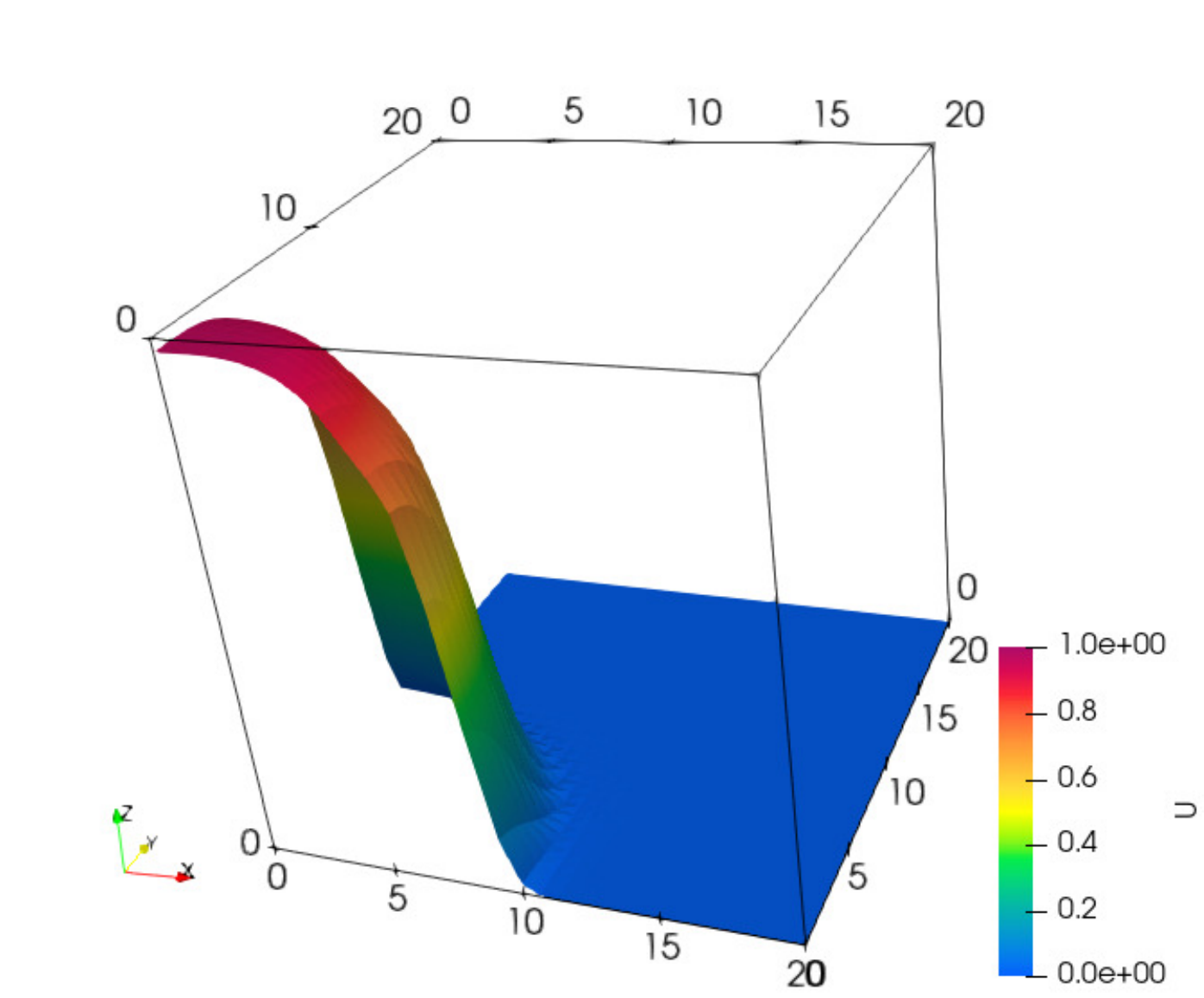}
		\caption{t = 10}
	\end{subfigure}
	\begin{subfigure}{0.24\textwidth}
		\includegraphics[width=\textwidth]{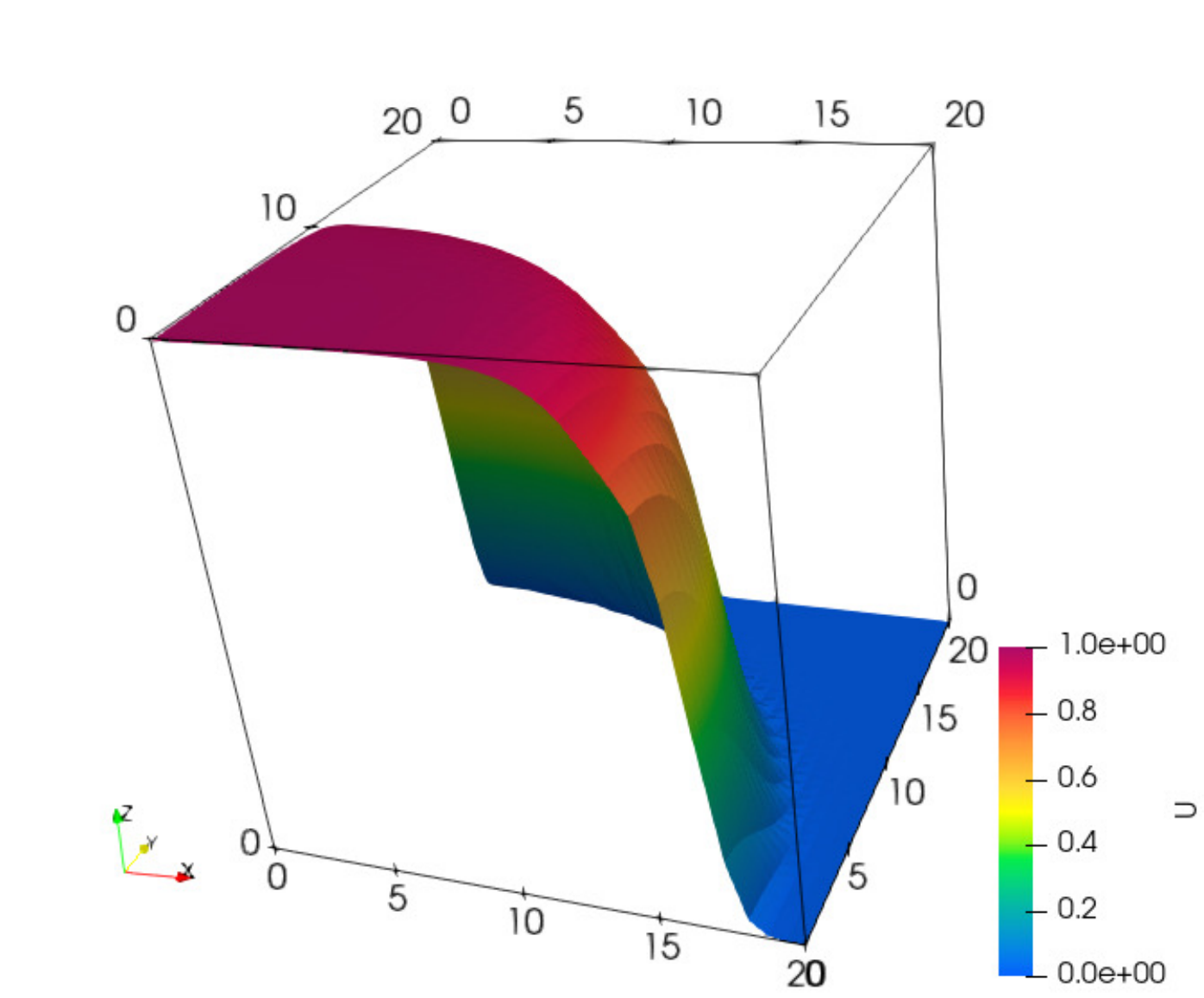}
		\caption{t = 20}
	\end{subfigure}
	\begin{subfigure}{0.24\textwidth}
		\includegraphics[width=\textwidth]{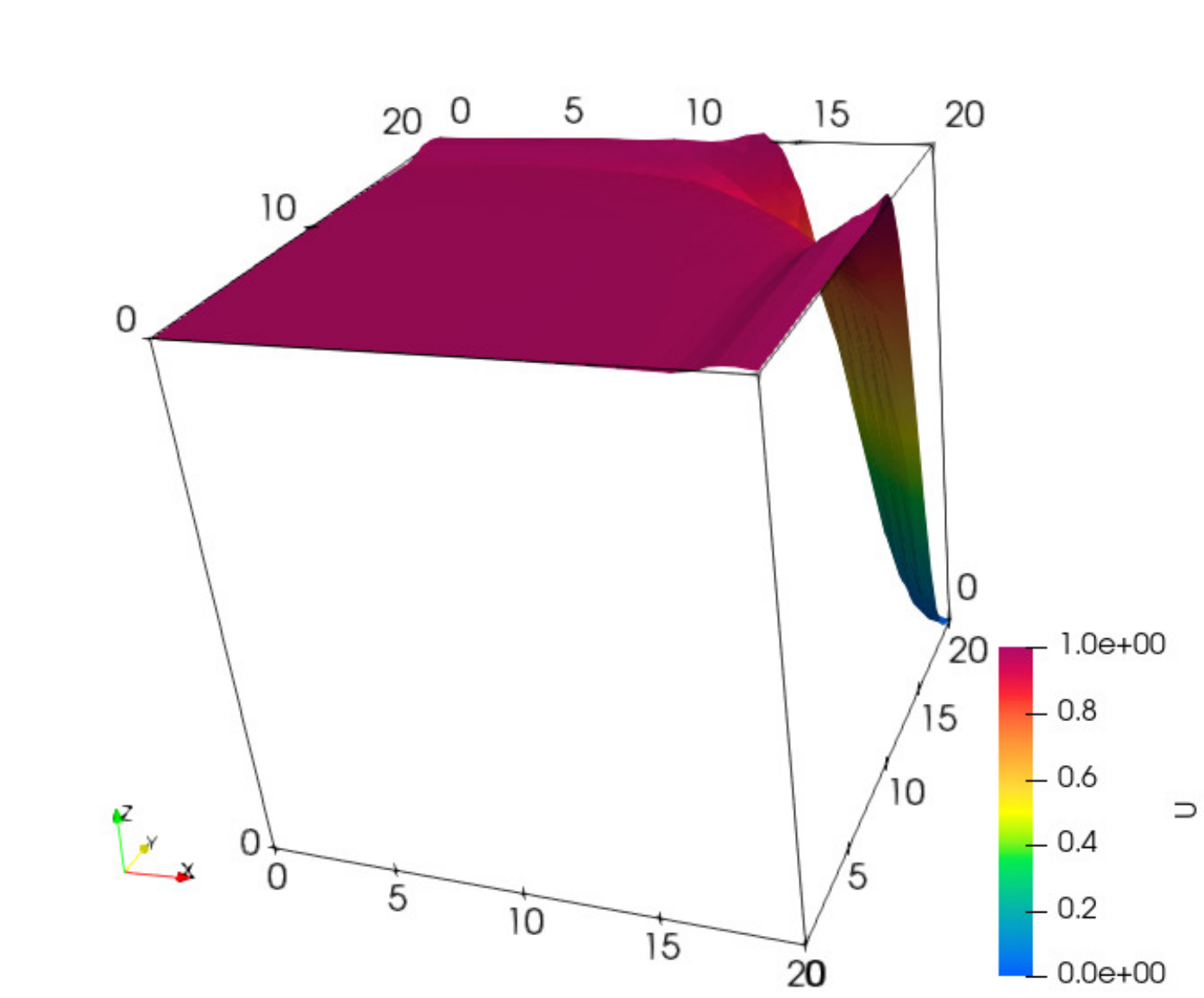}
		\caption{t = 30}
	\end{subfigure}
	\caption{
	Cancer cell invasion $u$ at different time instants $t=0, 10, 20, 30$,
obtained with the standard Galerkin FEM for $\alpha =10$, $\mu = 1$ and
$\chi=1$.
	 }
	\label{fig1}
\end{figure}

\begin{figure}[H]
	\centering
	\begin{subfigure}{0.24\textwidth}
		\includegraphics[width=\textwidth]{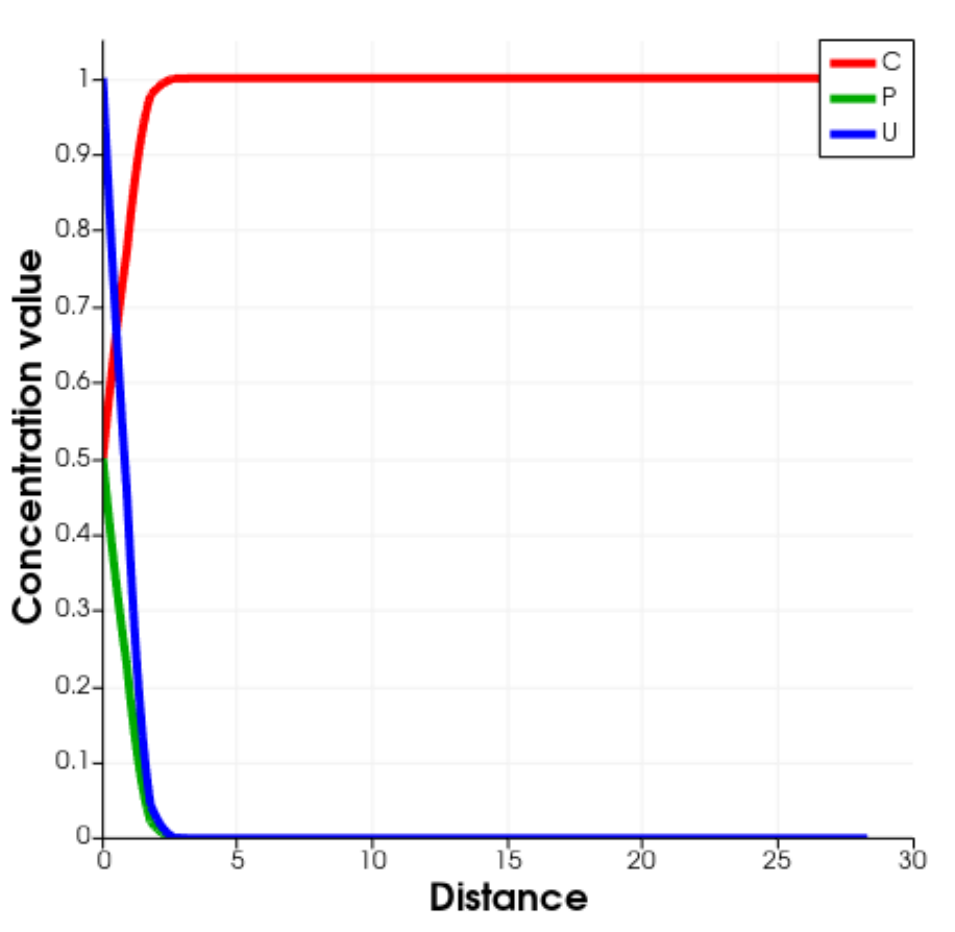}
		\caption{t = 0}
	\end{subfigure}
	\begin{subfigure}{0.24\textwidth}
		\includegraphics[width=\textwidth]{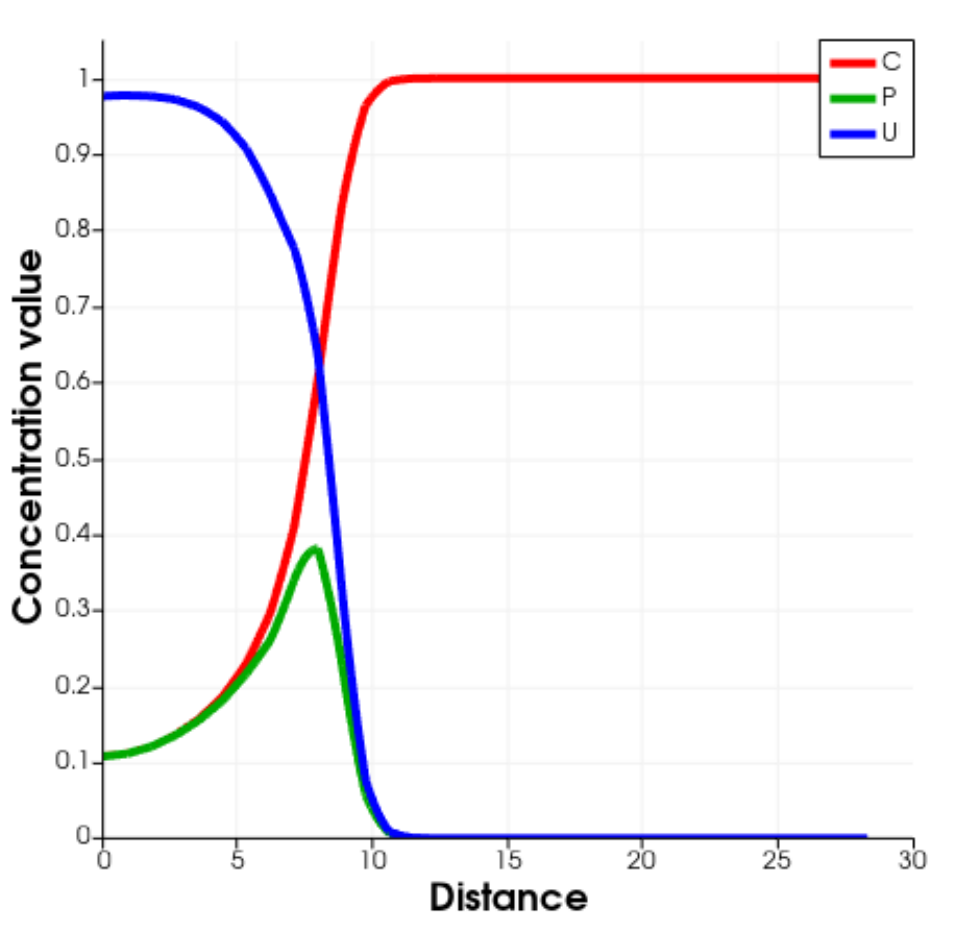}
		\caption{t = 10}
	\end{subfigure}
	\begin{subfigure}{0.24\textwidth}
		\includegraphics[width=\textwidth]{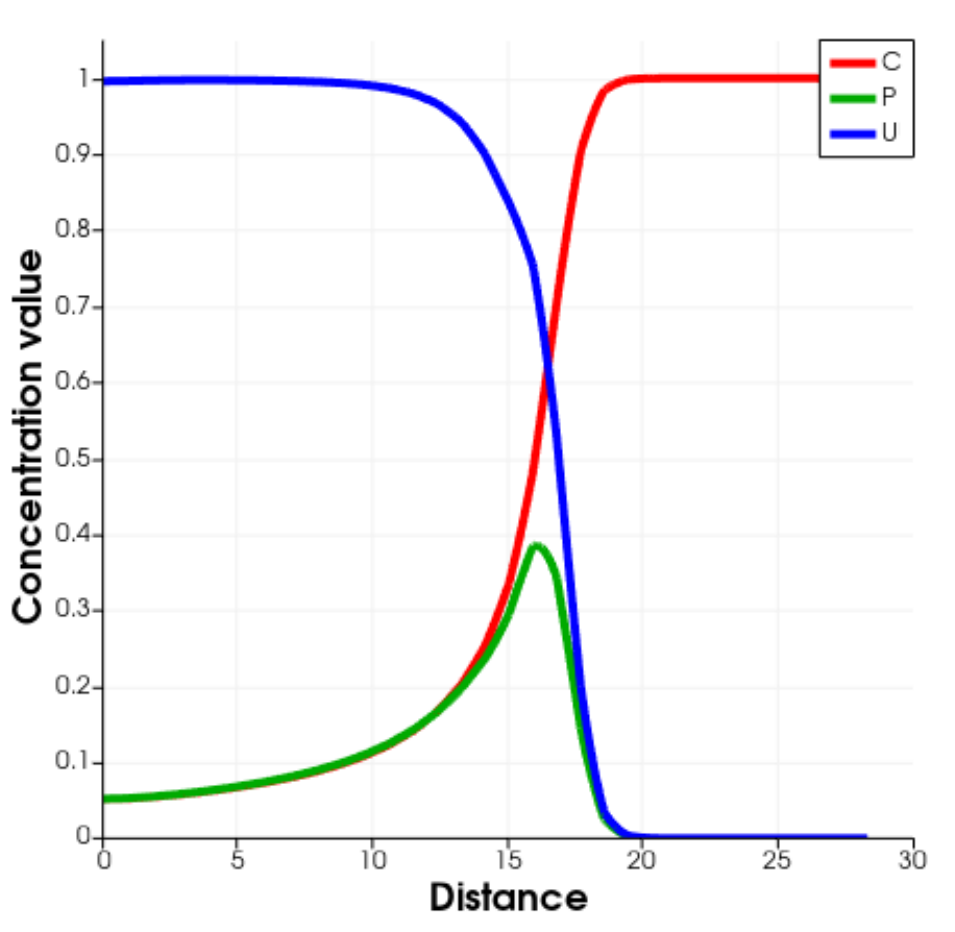}
		\caption{t = 20}
	\end{subfigure}
	\begin{subfigure}{0.24\textwidth}
		\includegraphics[width=\textwidth]{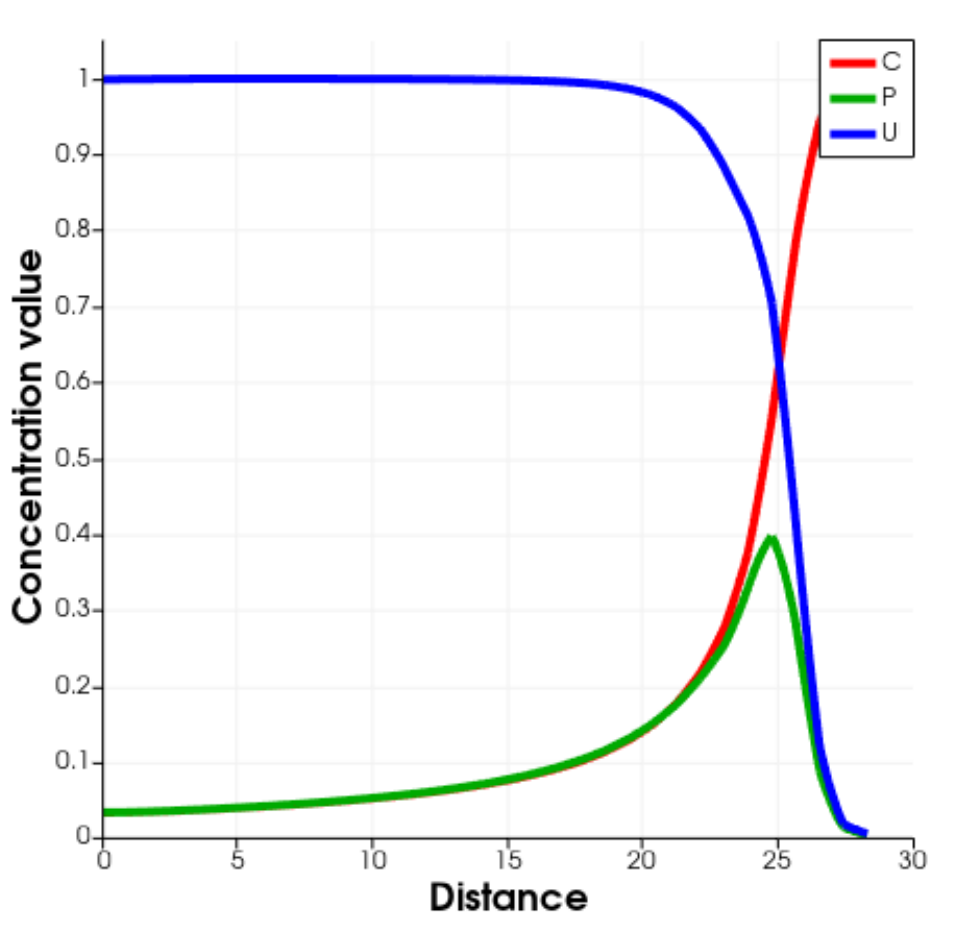}
		\caption{t = 30}
	\end{subfigure}
	\caption{
	Cancer cell invasion $u$, connective tissue $c$, and protease $p$ at
different time instants $t=0, 10, 20, 30$, obtained with the standard Galerkin
FEM for   $\alpha =10$, $\mu = 1$ and $\chi=1$.
	 }
	\label{fig2}
\end{figure}

\begin{figure}[H]
	\centering
	\begin{subfigure}{0.24\textwidth}
		\includegraphics[width=\textwidth]{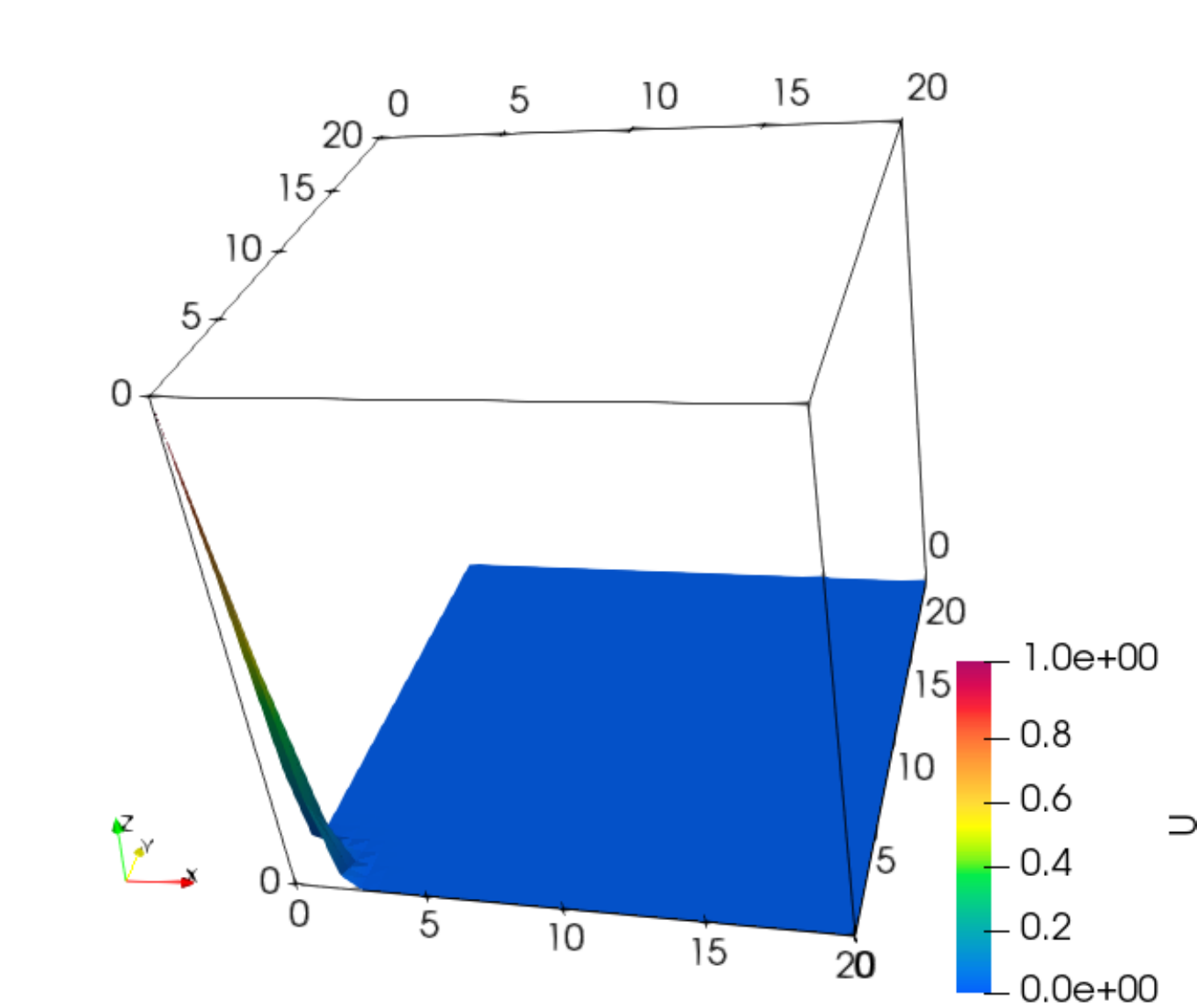}
		\caption{t = 0}
	\end{subfigure}
	\begin{subfigure}{0.24\textwidth}
		\includegraphics[width=\textwidth]{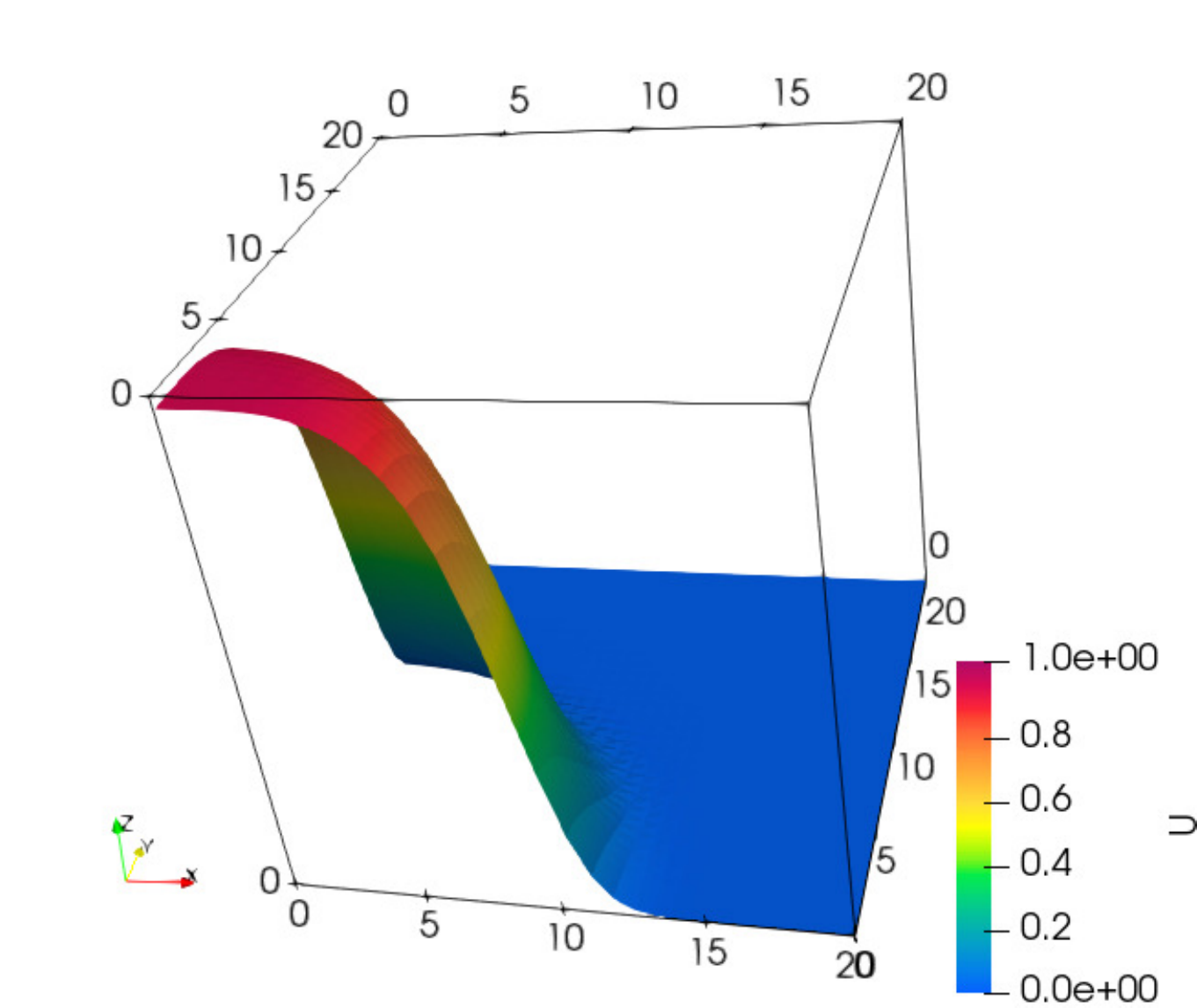}
		\caption{t = 10}
	\end{subfigure}
	\begin{subfigure}{0.24\textwidth}
		\includegraphics[width=\textwidth]{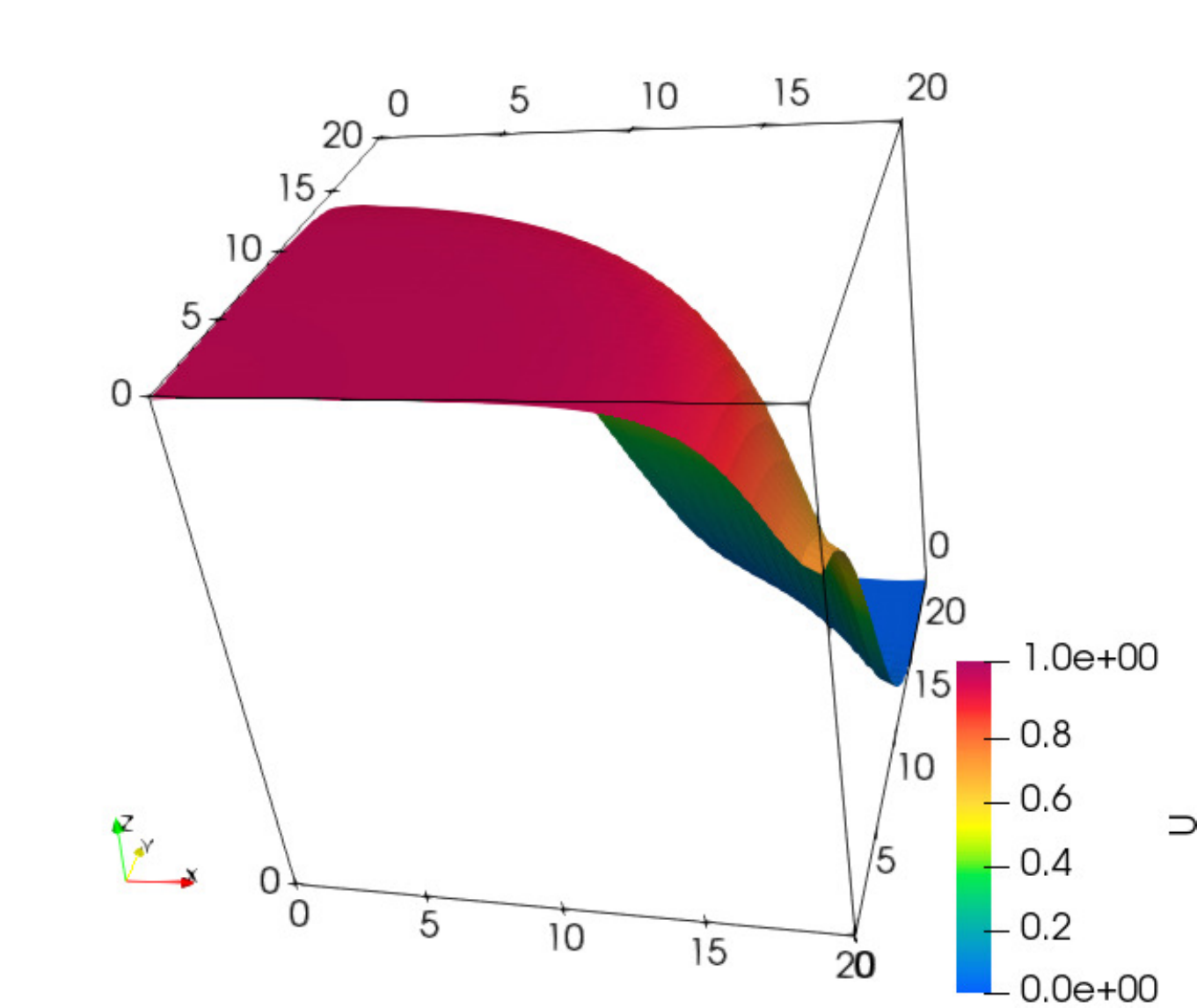}
		\caption{t = 20}
	\end{subfigure}
	\begin{subfigure}{0.24\textwidth}
		\includegraphics[width=\textwidth]{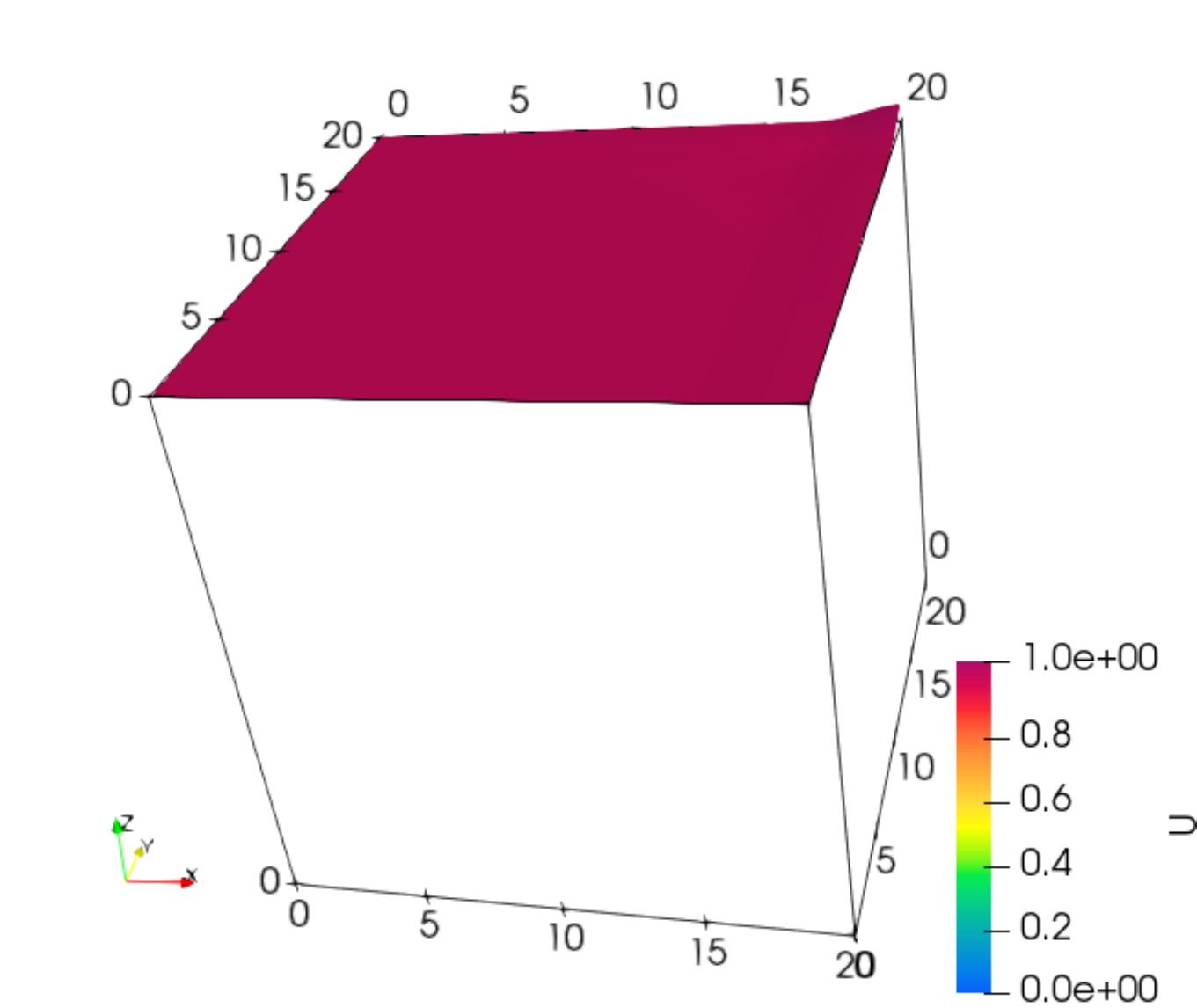}
		\caption{t = 30}
	\end{subfigure}
	\caption{
	Cancer cell invasion $u$ at different time instants $t=0, 10, 20, 30$,
obtained with the FEM-FCT scheme for $\alpha =10$, $\mu = 1$ and $\chi=1$.
	 }
	\label{fig3}
\end{figure}

\begin{figure}[H]
	\centering
	\begin{subfigure}{0.24\textwidth}
		\includegraphics[width=\textwidth]{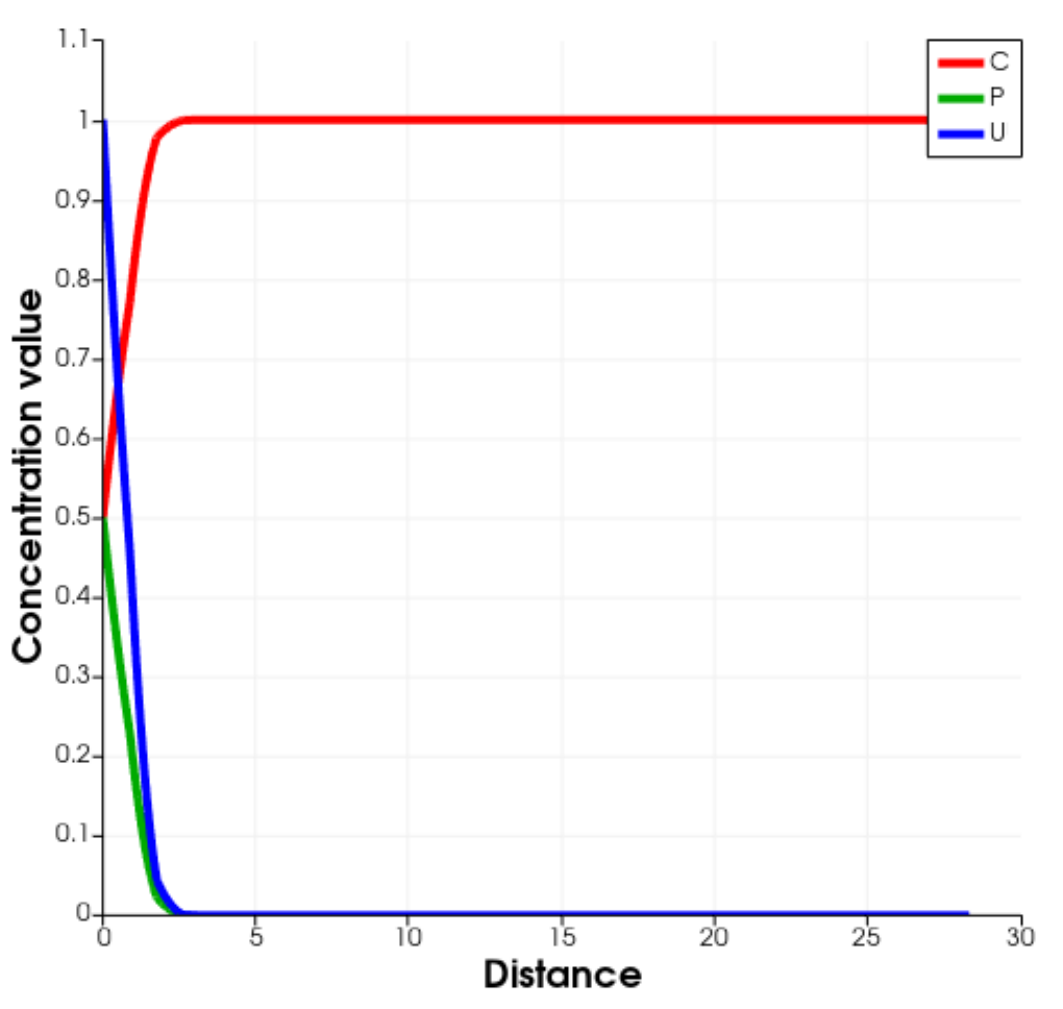}
		\caption{t = 0}
	\end{subfigure}
	\begin{subfigure}{0.24\textwidth}
		\includegraphics[width=\textwidth]{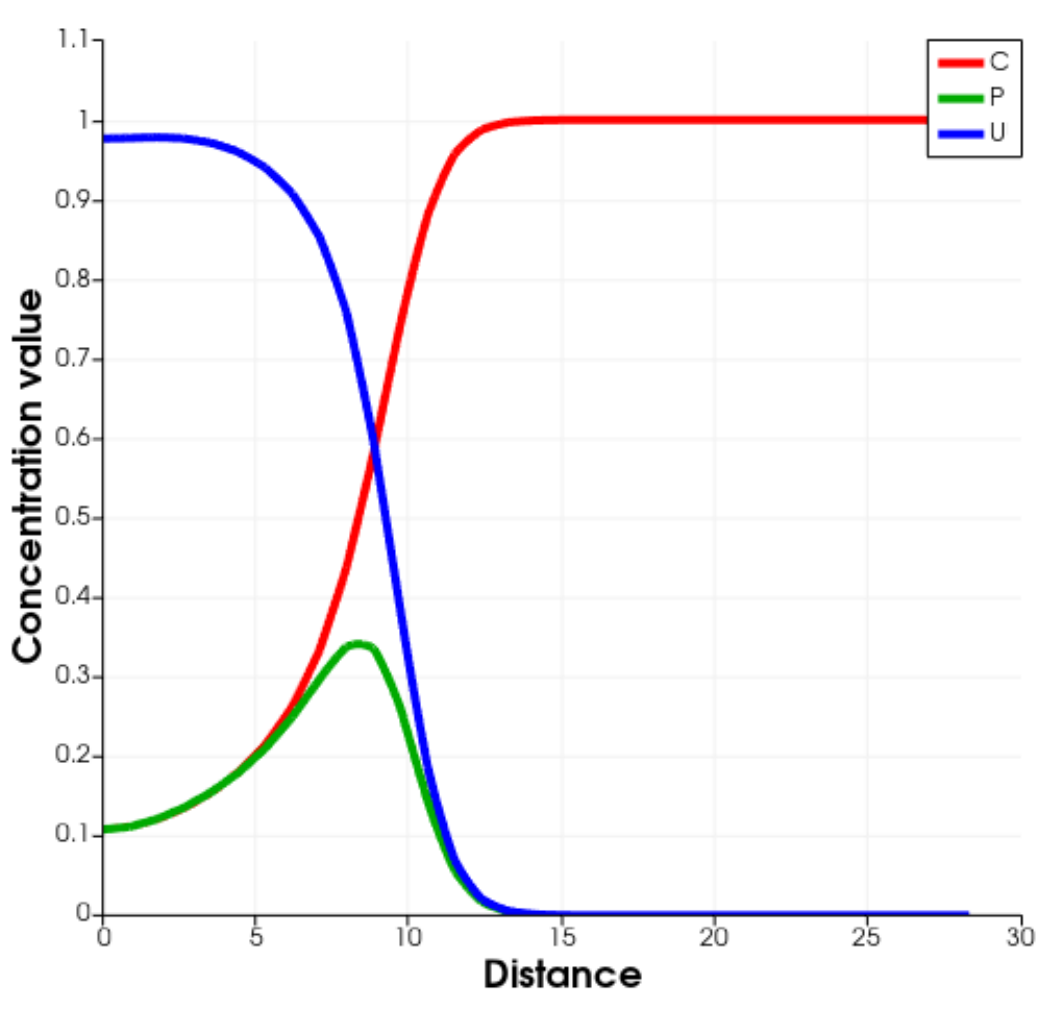}
		\caption{t = 10}
	\end{subfigure}
	\begin{subfigure}{0.24\textwidth}
		\includegraphics[width=\textwidth]{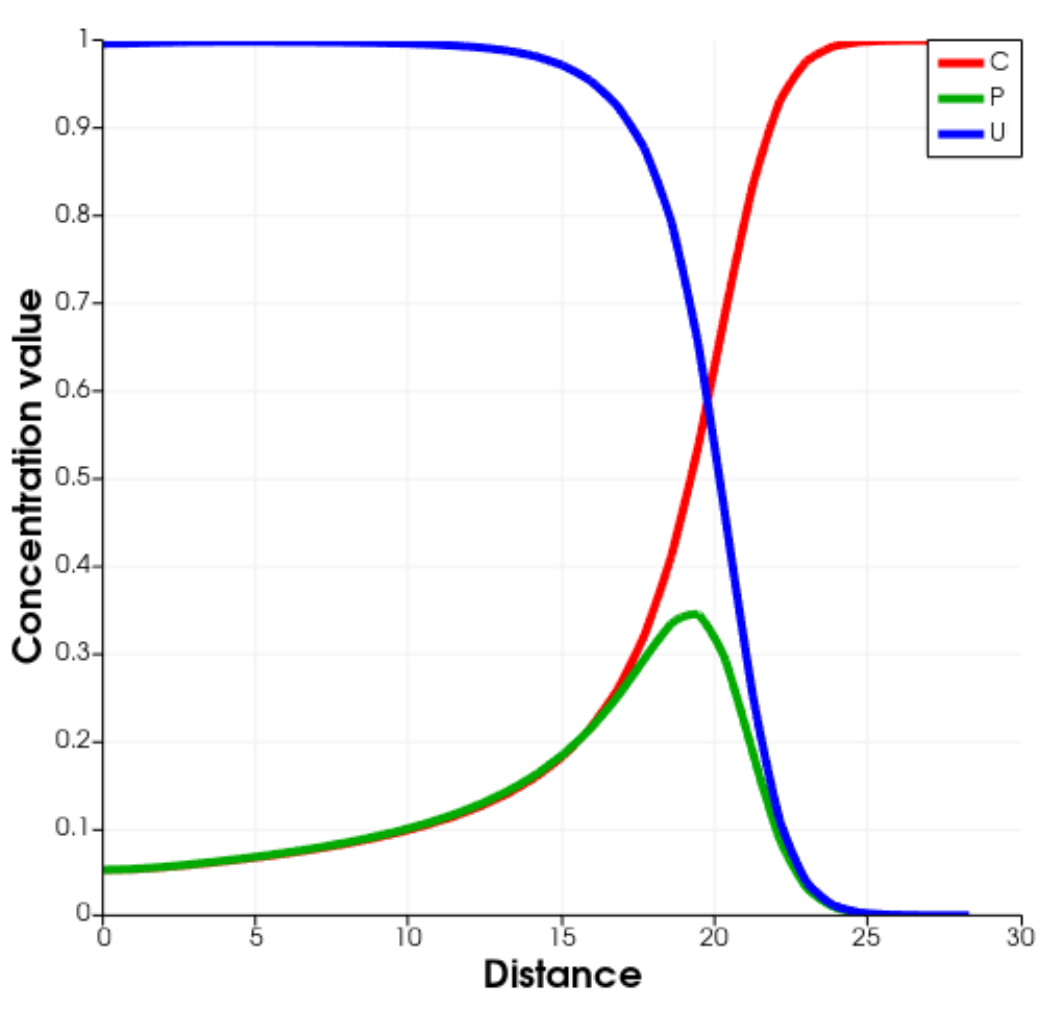}
		\caption{t = 20}
	\end{subfigure}
	\begin{subfigure}{0.24\textwidth}
		\includegraphics[width=\textwidth]{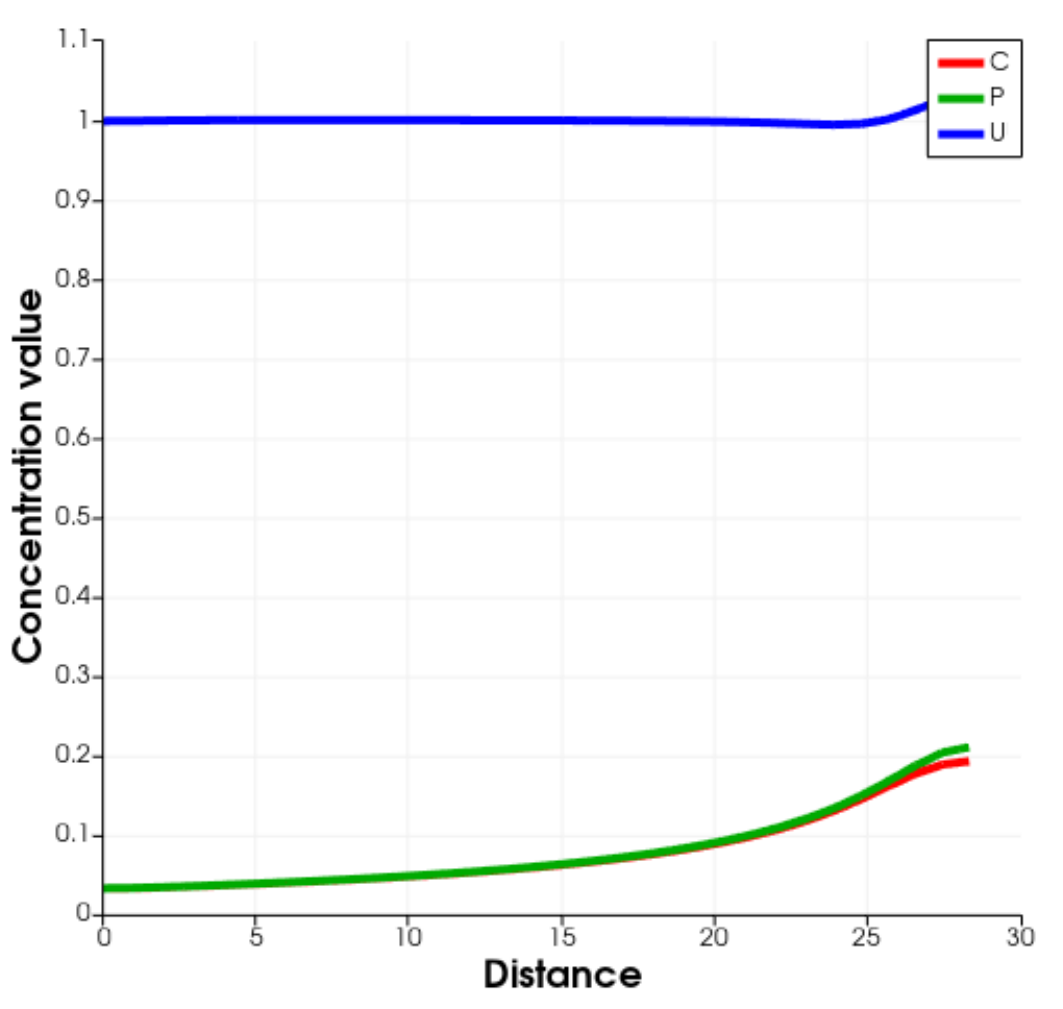}
		\caption{t = 30}
	\end{subfigure}
	\caption{
	Cancer cell invasion $u$, connective tissue $c$, and protease $p$ at
different time instants $t=0, 10, 20, 30$, 
obtained with the FEM-FCT scheme for $\alpha =10$, $\mu = 1$ and $\chi=1$.
	 }
	\label{fig4}
\end{figure}

\begin{figure}[H]
	\centering
	\begin{subfigure}{0.24\textwidth}
		\includegraphics[width=\textwidth]{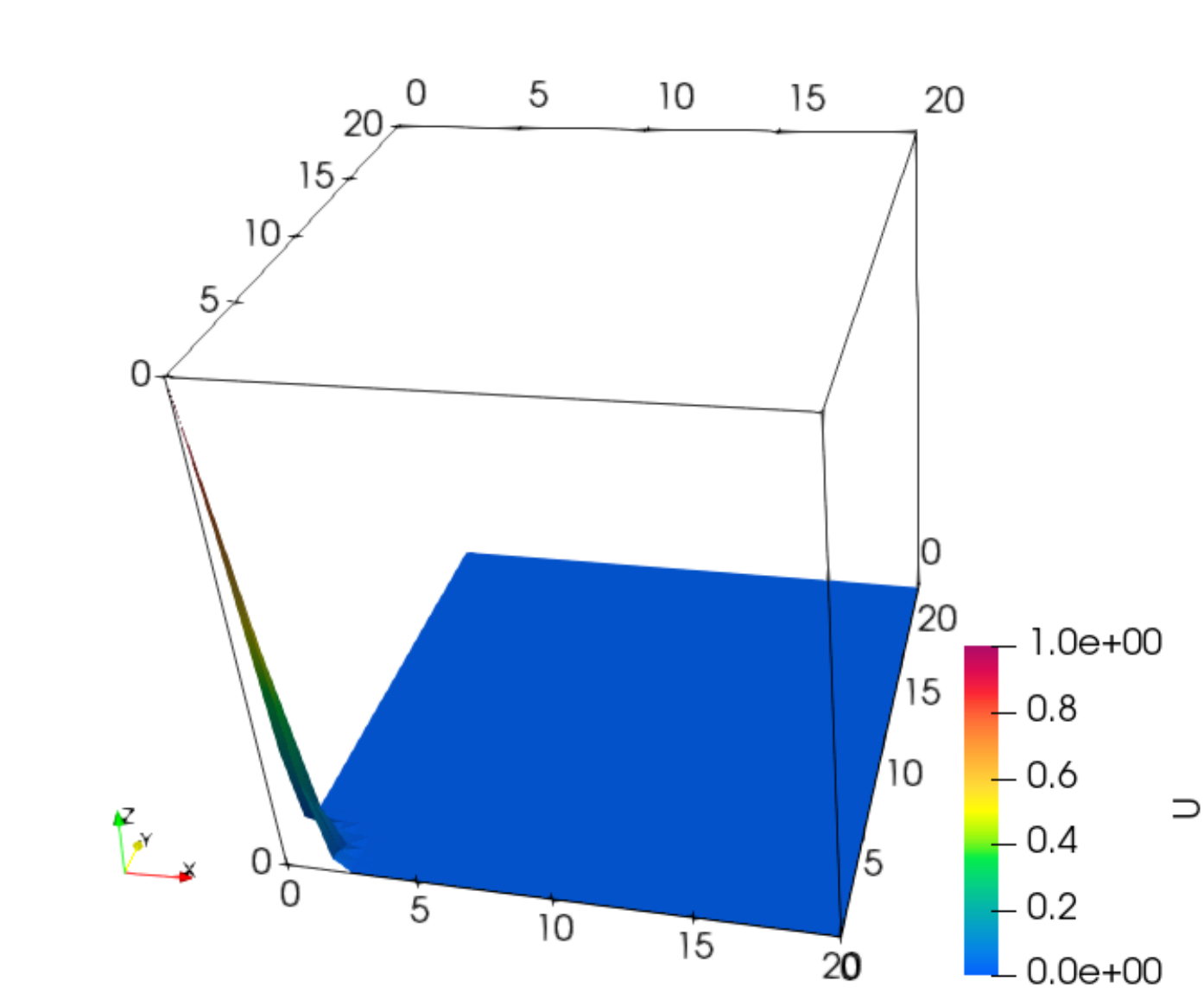}
		\caption{t = 0}
	\end{subfigure}
	\begin{subfigure}{0.24\textwidth}
		\includegraphics[width=\textwidth]{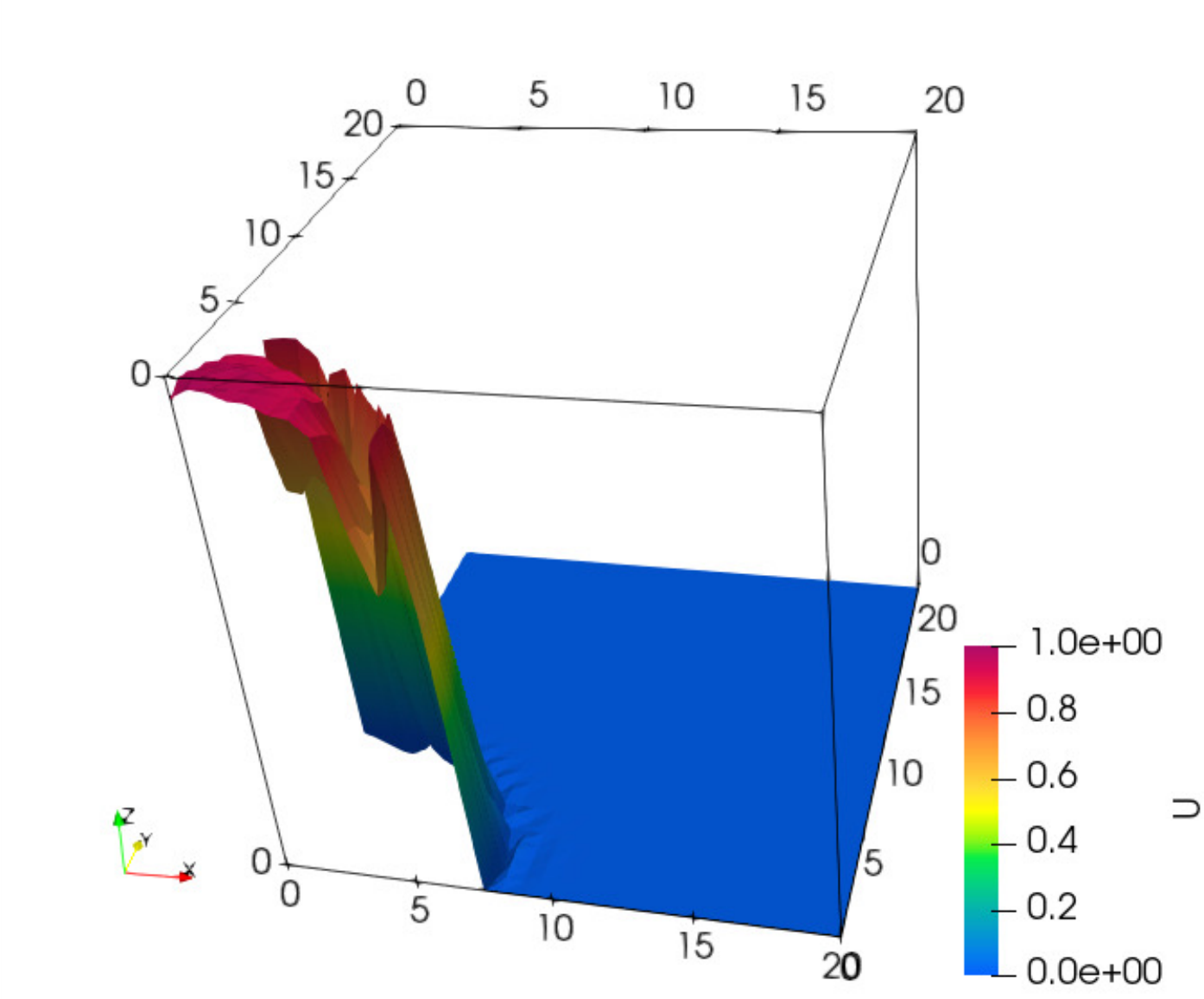}
		\caption{t = 10}
	\end{subfigure}
	\begin{subfigure}{0.24\textwidth}
		\includegraphics[width=\textwidth]{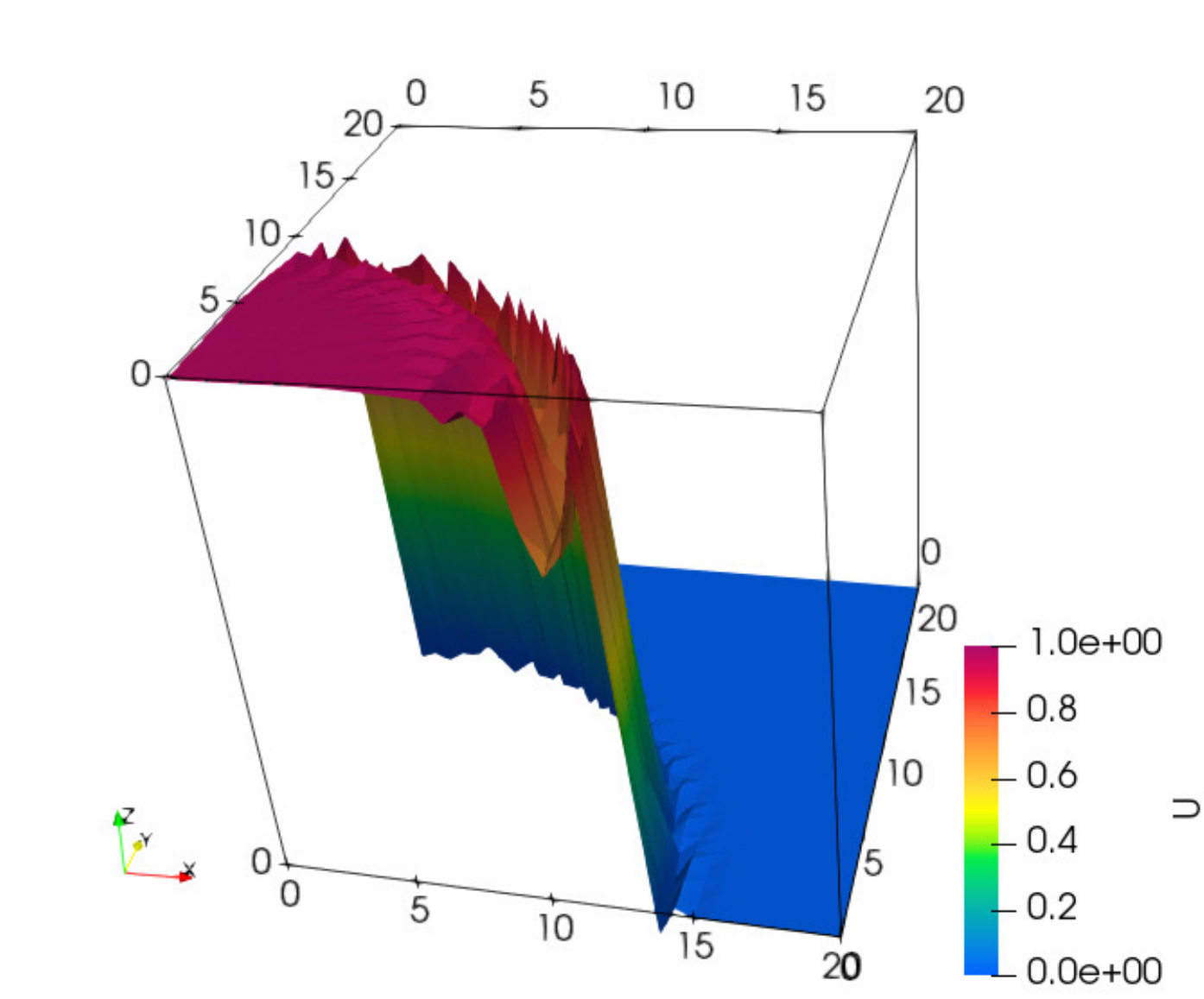}
		\caption{t = 20}
	\end{subfigure}
	\begin{subfigure}{0.24\textwidth}
		\includegraphics[width=\textwidth]{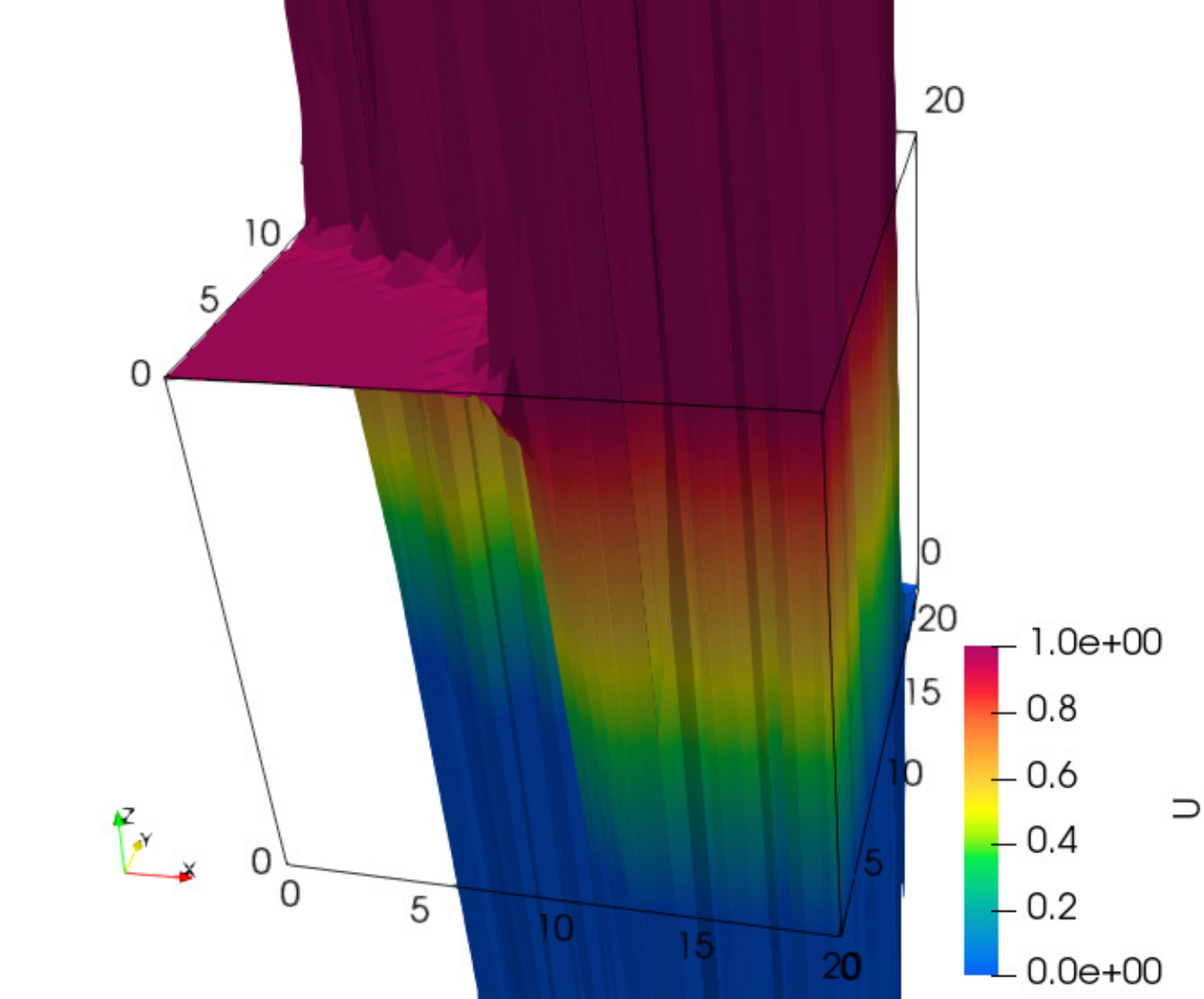}
		\caption{t = 30}
	\end{subfigure}
	\caption{
	Cancer cell invasion $u$ at different time instants $t=0, 10, 20, 30$,
obtained with the standard Galerkin FEM for $\alpha =1000$, $\mu = 1$ and
$\chi=1$.
	 }
	\label{fig5}
\end{figure} 

\begin{figure}[H]
	\centering
	\begin{subfigure}{0.24\textwidth}
		\includegraphics[width=\textwidth]{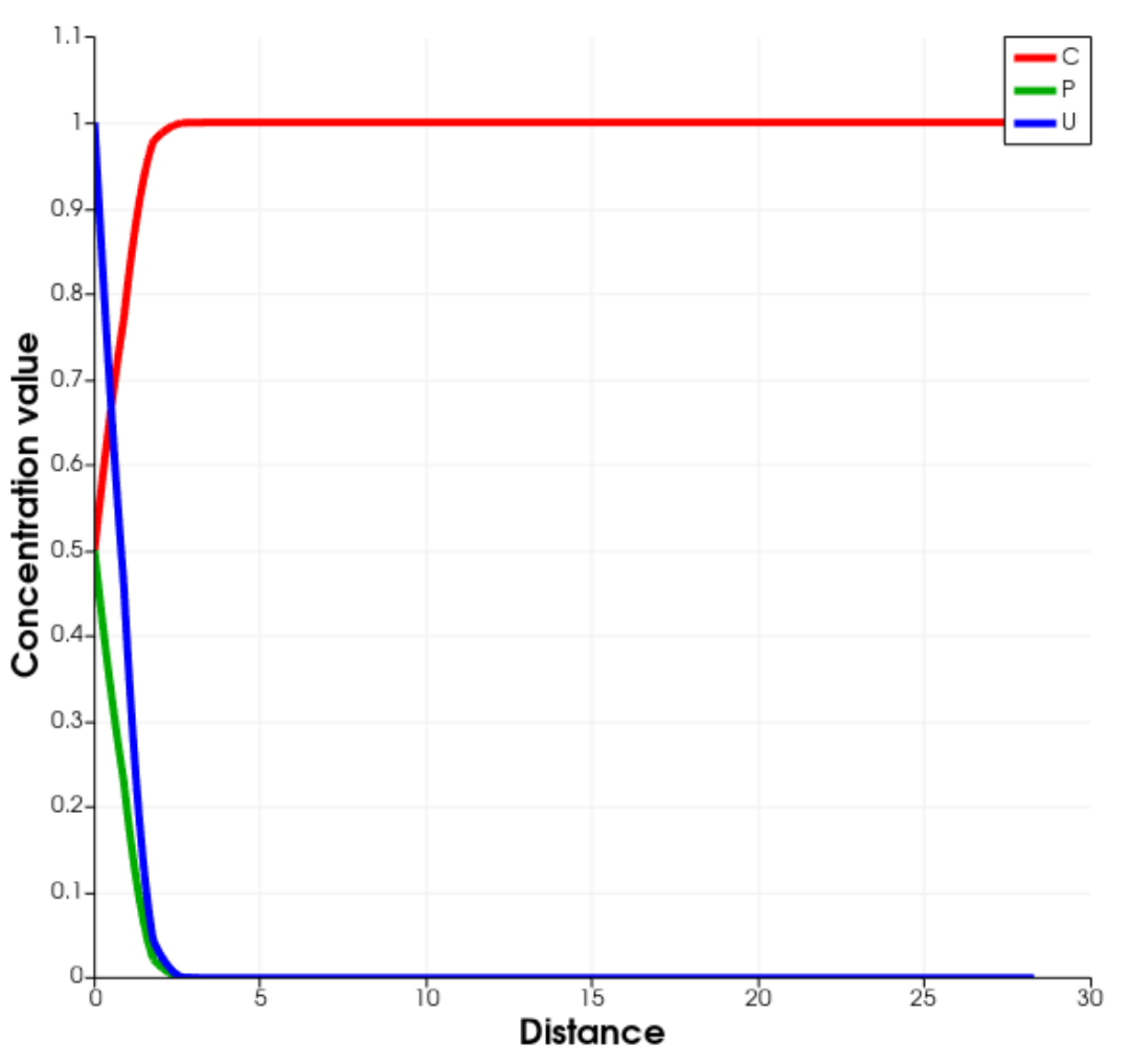}
		\caption{t = 0}
	\end{subfigure}
	\begin{subfigure}{0.24\textwidth}
		\includegraphics[width=\textwidth]{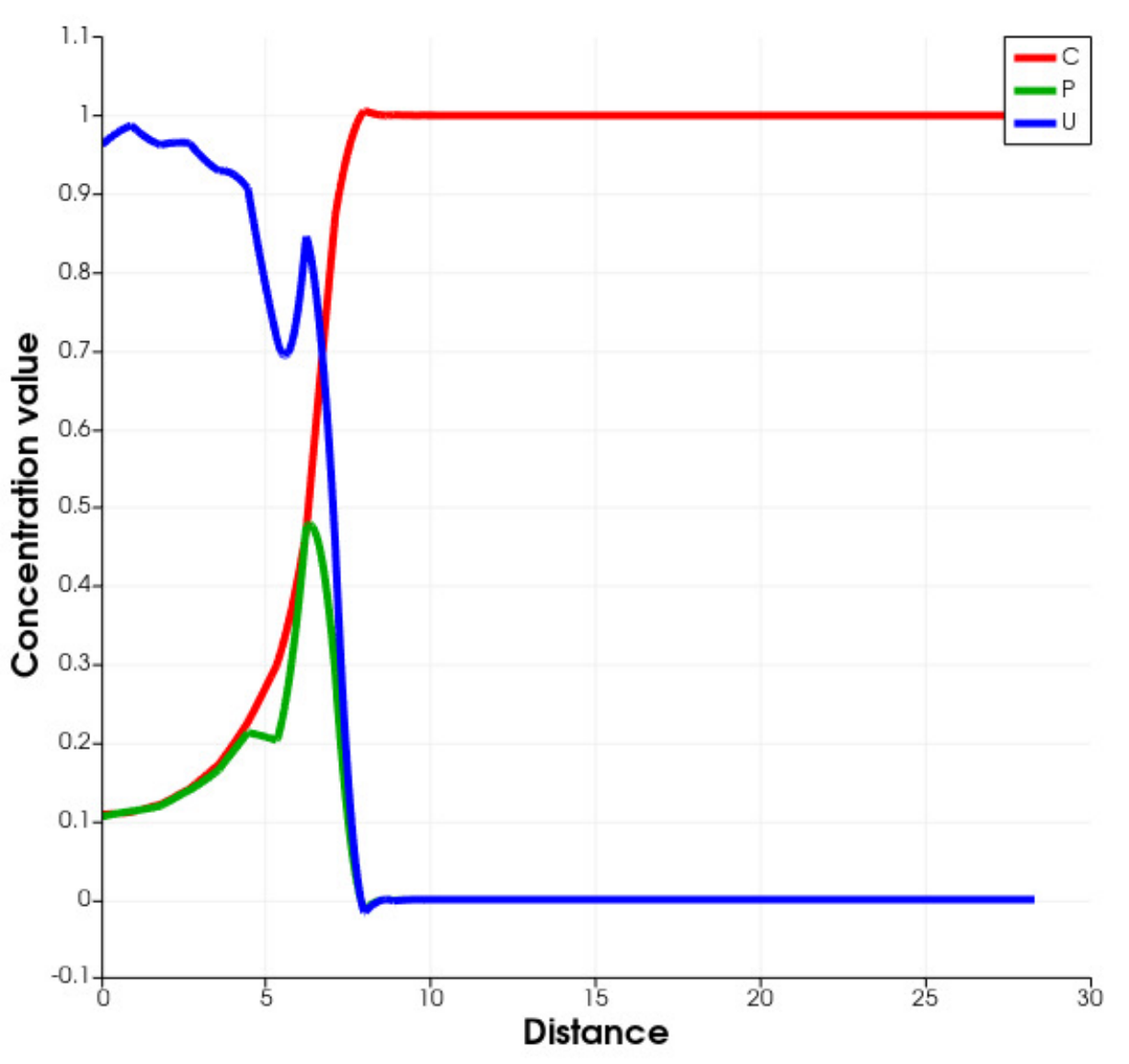}
		\caption{t = 10}
	\end{subfigure}
	\begin{subfigure}{0.24\textwidth}
		\includegraphics[width=\textwidth]{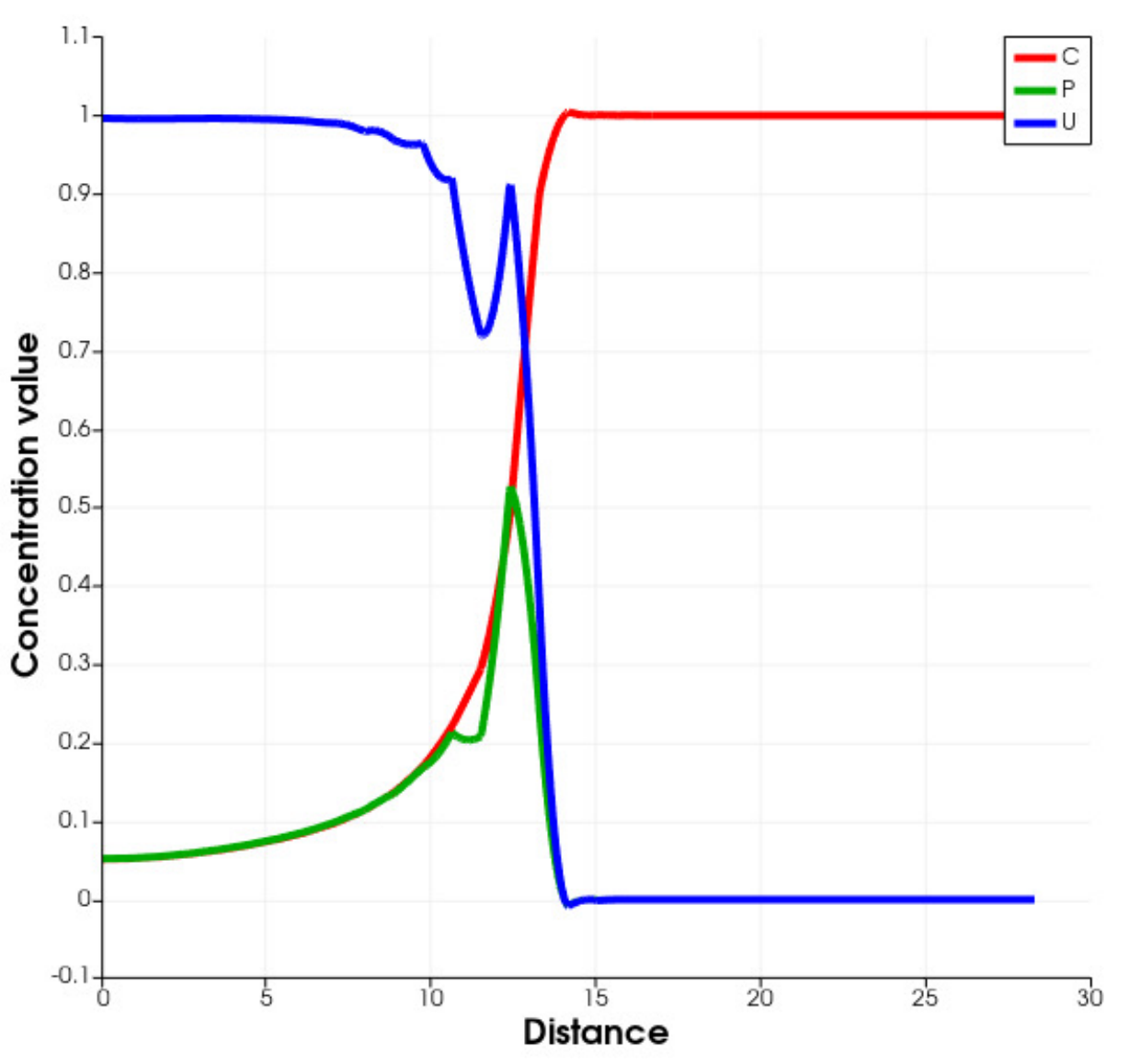}
		\caption{t = 20}
	\end{subfigure}
	\begin{subfigure}{0.24\textwidth}
		\includegraphics[width=\textwidth]{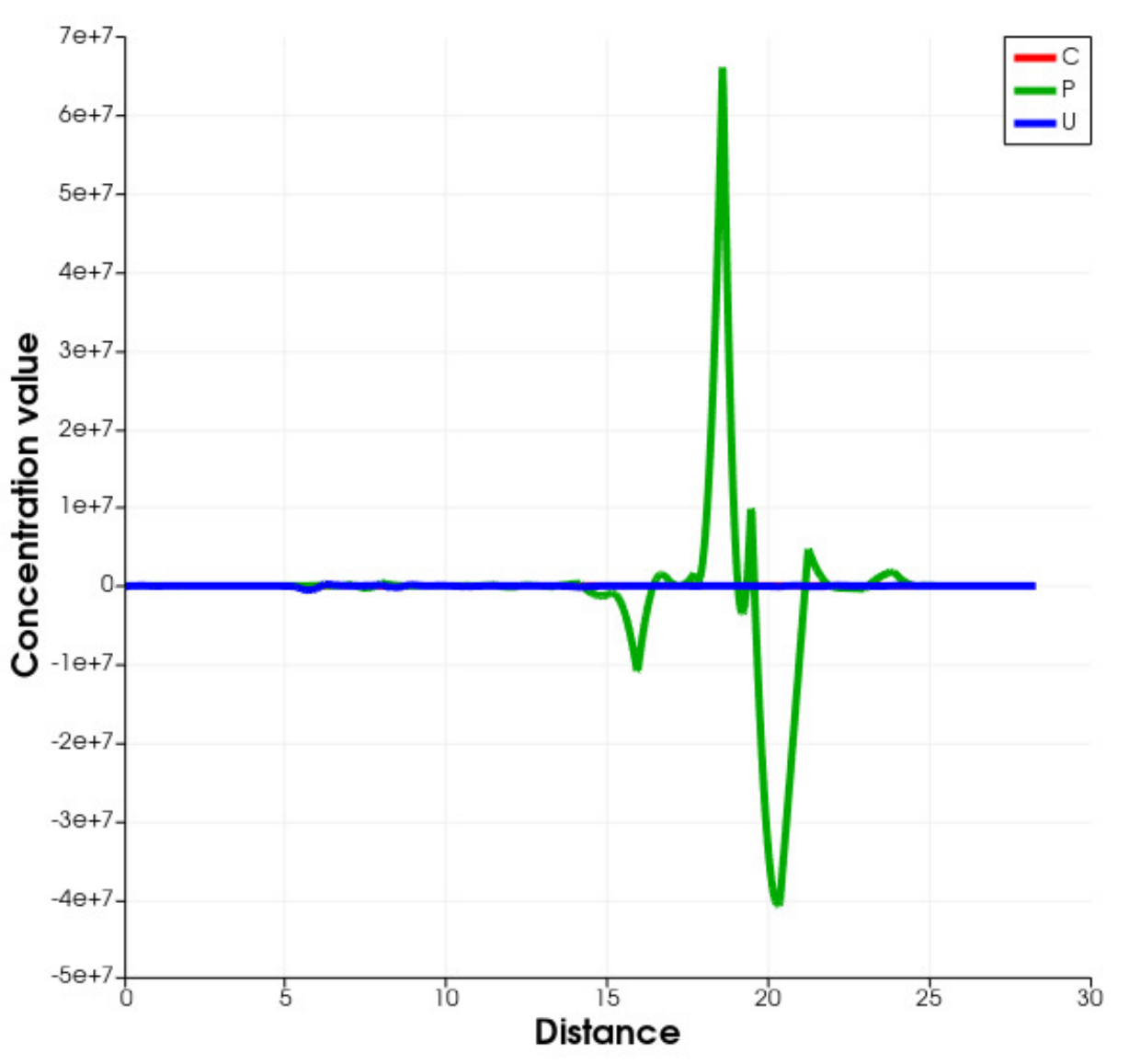}
		\caption{t = 30}
	\end{subfigure}
	\caption{
	Cancer cell invasion $u$, connective tissue $c$, and protease $p$ at
different time instants $t=0, 10, 20, 30$, obtained with the standard Galerkin 
FEM for $\alpha = 1000$, $\mu = 1$ and $\chi=1$.
	 }
	\label{fig6}
\end{figure} 

\begin{figure}[H]
	\centering
	\begin{subfigure}{0.24\textwidth}
		\includegraphics[width=\textwidth]{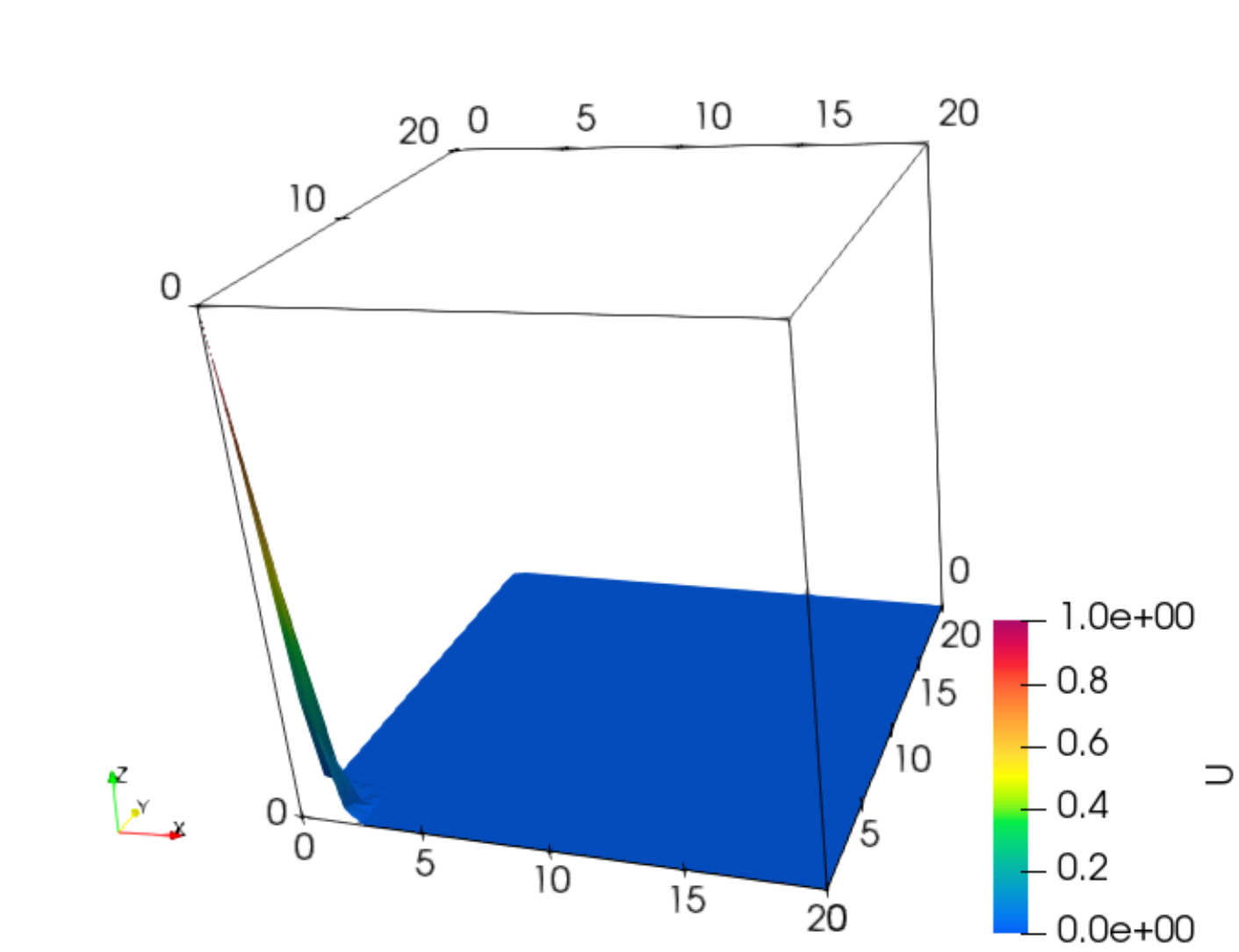}
		\caption{t = 0}
	\end{subfigure}
	\begin{subfigure}{0.24\textwidth}
		\includegraphics[width=\textwidth]{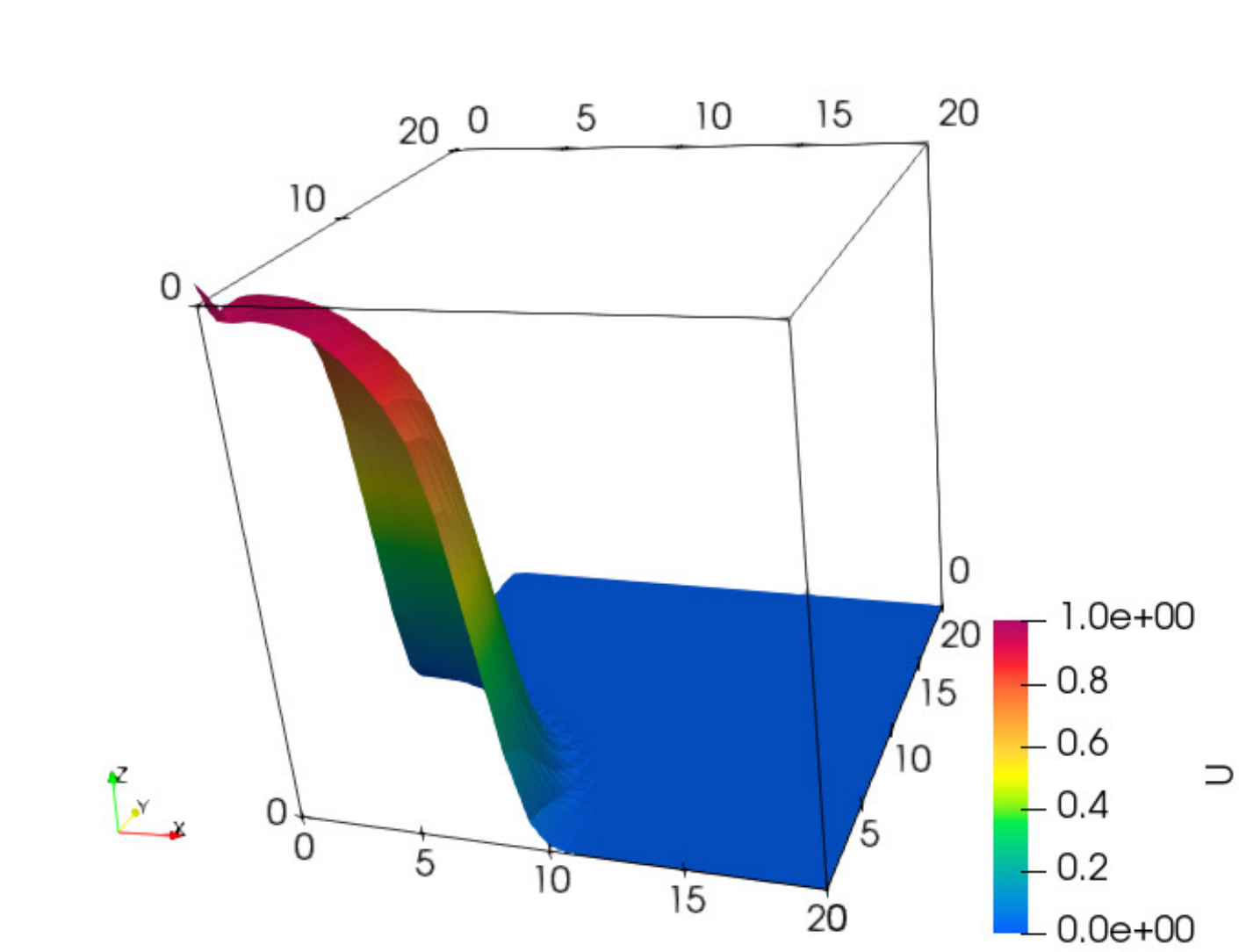}
		\caption{t = 10}
	\end{subfigure}
	\begin{subfigure}{0.24\textwidth}
		\includegraphics[width=\textwidth]{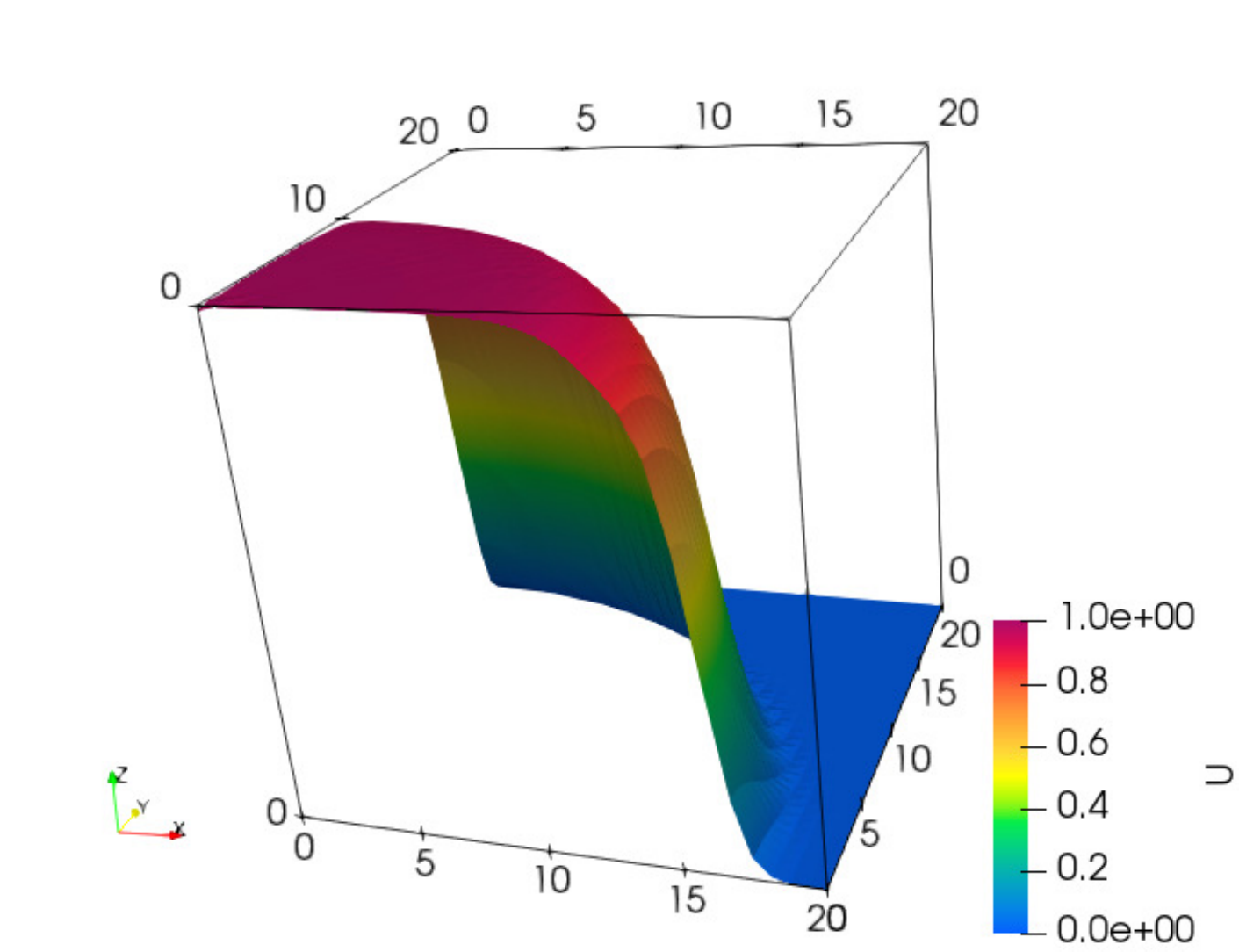}
		\caption{t = 20}
	\end{subfigure}
	\begin{subfigure}{0.24\textwidth}
		\includegraphics[width=\textwidth]{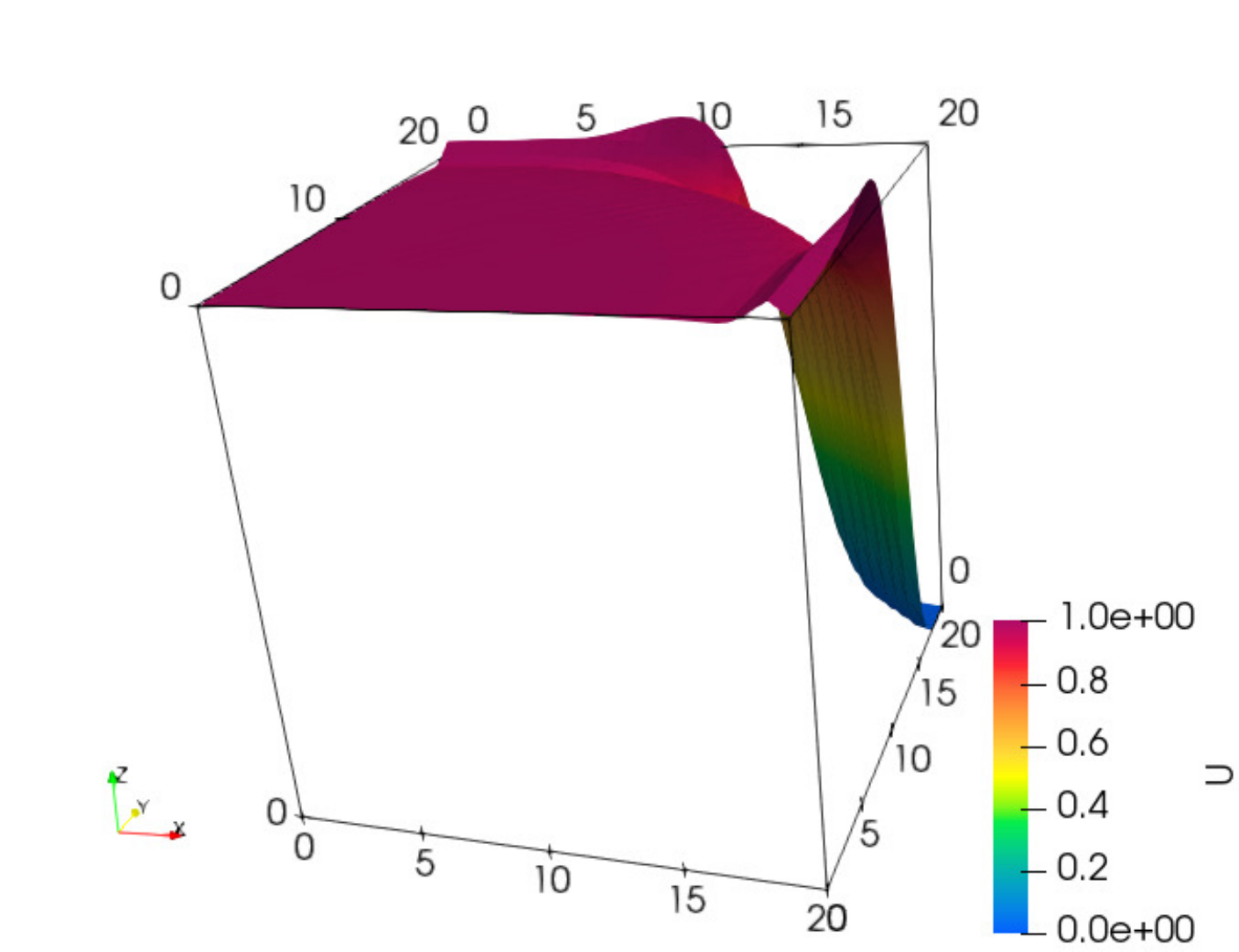}
		\caption{t = 30}
	\end{subfigure}
	\caption{
	Cancer cell invasion $u$ at different time instants $t=0, 10, 20, 30$,
obtained with the FEM-FCT scheme for $\alpha =1000$, $\mu = 1$ and $\chi=1$.
	 }
	\label{fig7}
\end{figure}

\begin{figure}[H]
	\centering
	\begin{subfigure}{0.24\textwidth}
		\includegraphics[width=\textwidth]{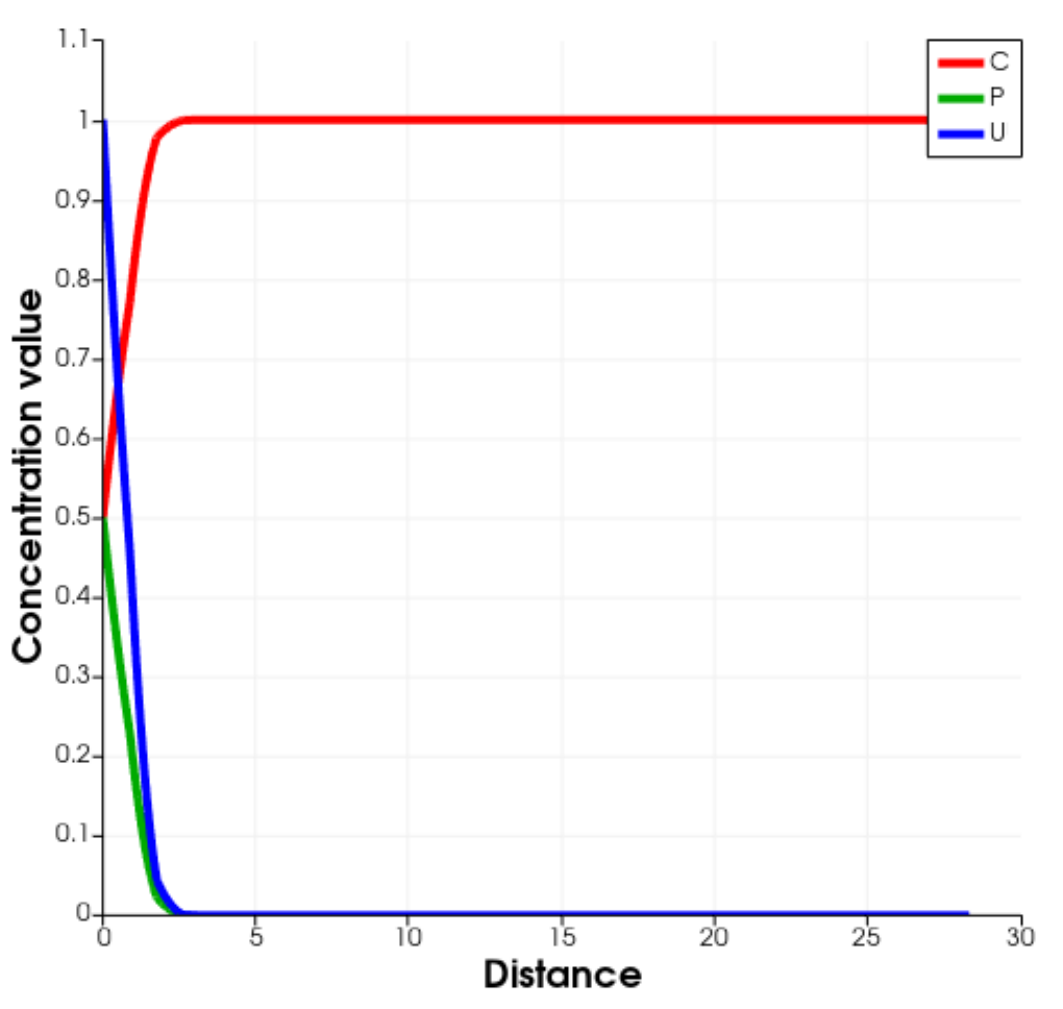}
		\caption{t = 0}
	\end{subfigure}
	\begin{subfigure}{0.24\textwidth}
		\includegraphics[width=\textwidth]{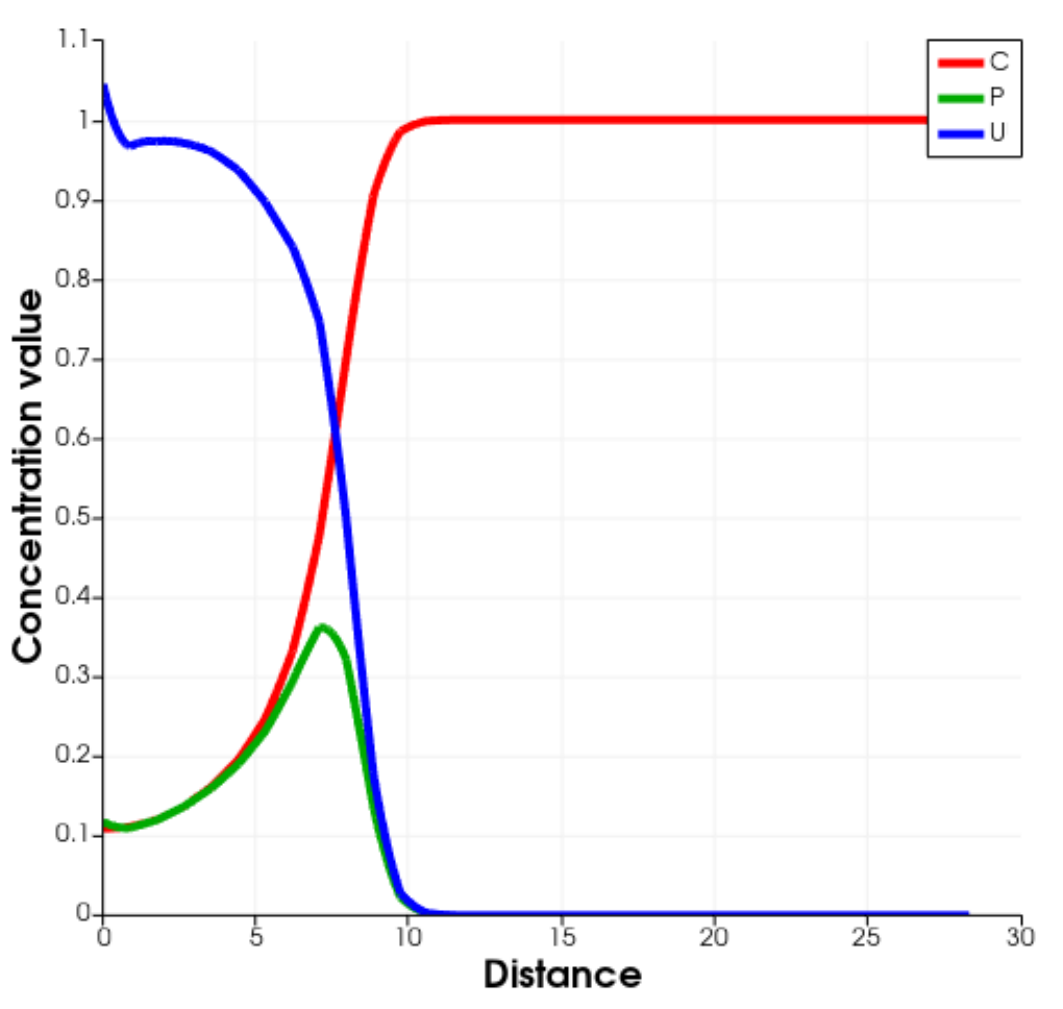}
		\caption{t = 10}
	\end{subfigure}
	\begin{subfigure}{0.24\textwidth}
		\includegraphics[width=\textwidth]{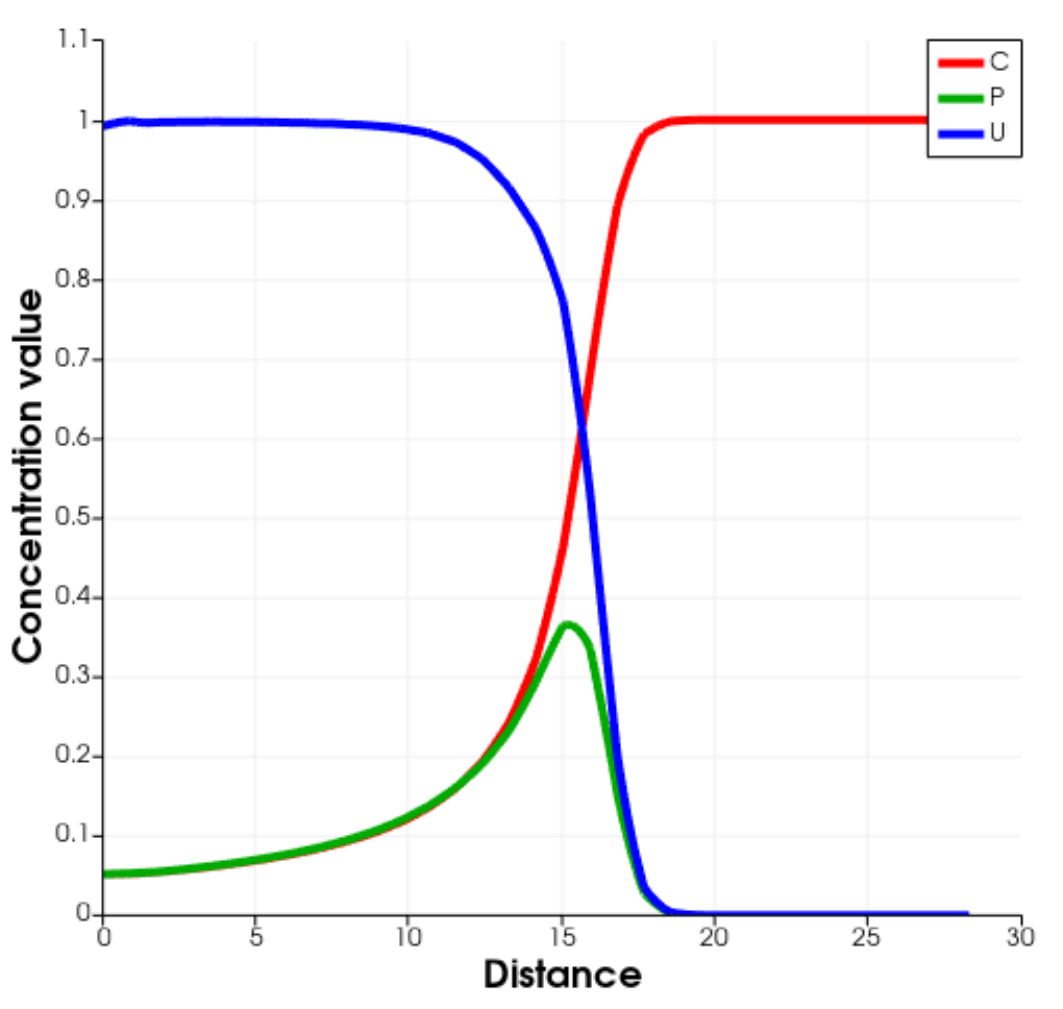}
		\caption{t = 20}
	\end{subfigure}
	\begin{subfigure}{0.24\textwidth}
		\includegraphics[width=\textwidth]{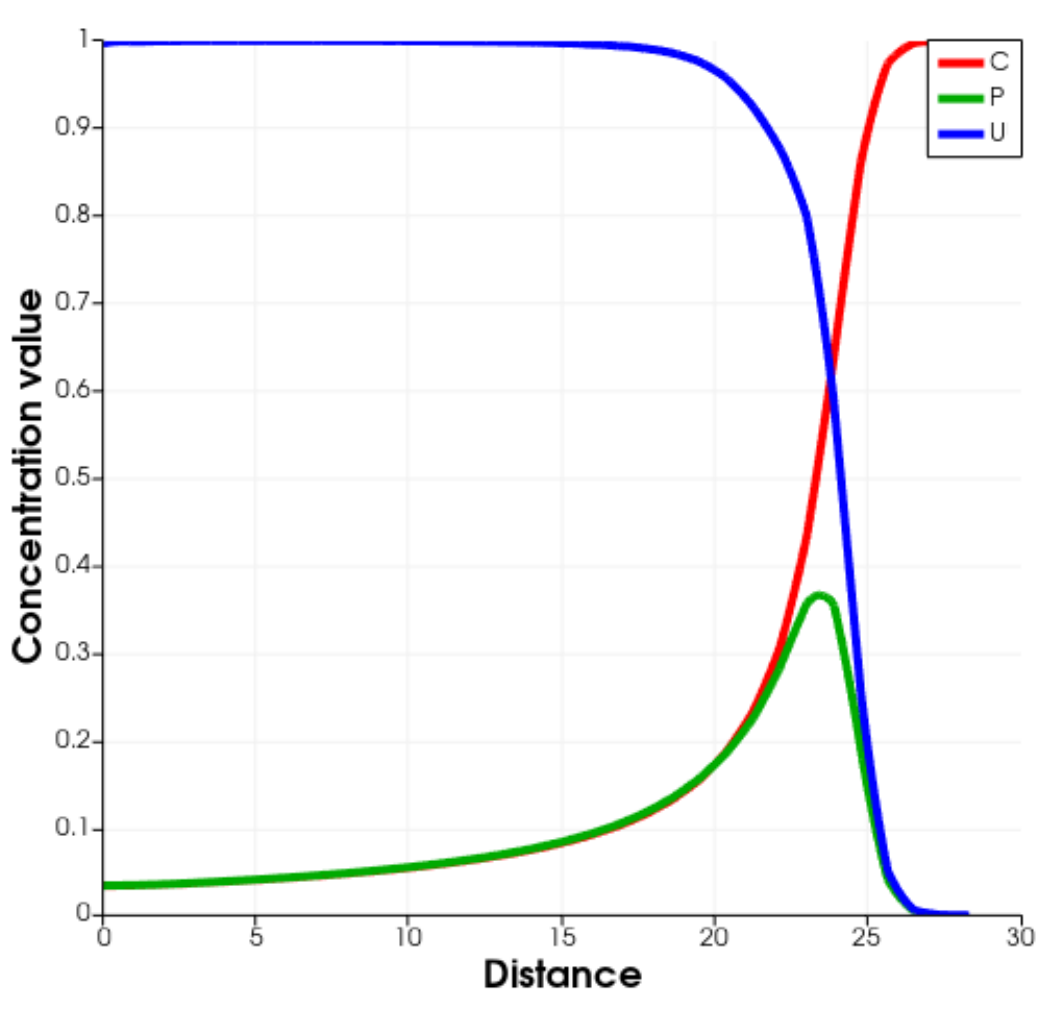}
		\caption{t = 30}
	\end{subfigure}
	\caption{
	Cancer cell invasion $u$, connective tissue $c$, and protease $p$ at
different time instants $t= 0, 10, 20, 30$, obtained with the FEM-FCT scheme
for $\alpha =1000$, $\mu = 1$ and $\chi=1$.
	 }
	\label{fig8}
\end{figure} 

\subsection{The FEM-FCT scheme in absence of diffusion for $\chi=1$, $\mu=1$}
In this section, we consider the case without the diffusion term, i.e.,
utilizing \eqref{eq1} now, and again set $\chi=\mu=1$. This case was studied in
\cite{chapwanya2014positivity, khalsaraei2016positivity}, where the authors
applied a nonstandard finite difference (NSFD) scheme using Mickens rules. The
proposed methods were successful in comparison to standard finite difference
methods at obtaining positive solutions, however, some wiggles still remained
in the vicinity of the front layer. On the other hand, deriving an efficient
NSFD scheme heavily depends on the type of the system and the discretization of
different terms. Therefore, in this work, we applied the FEM-FCT methodology to
remove the oscillations in the front layer while keeping the solutions positive
at all times, see Figs.~\ref{fig9} and \ref{fig10}. Next,
we check numerically whether the approximate solutions converge. To this end, 
we computed the integrals of the solutions at the final time $t=50$ for
different numbers of global refinements, see Table~\ref{tab2}.
The results correspond to the situation where the tumor is completely malignant 
and invades the whole extracellular matrix. 
In Table~\ref{tab3}, we study the values of the solutions at the point
$(20,20)$, the differences between two consecutive iterative solutions, and the
numbers of fixed-point iterations for different time steps. In particular, we
observe that the proposed scheme is convergent with respect to the time step 
size. The convergence of the cancer cell invasion $u$ with respect to the time
step and the mesh width at two different time instants is also studied in
Fig.~\ref{fig11} by means of solution graphs along the line $y=x$.

\begin{figure}[H]
	\centering
	\begin{subfigure}{0.24\textwidth}
		\includegraphics[width=\textwidth]{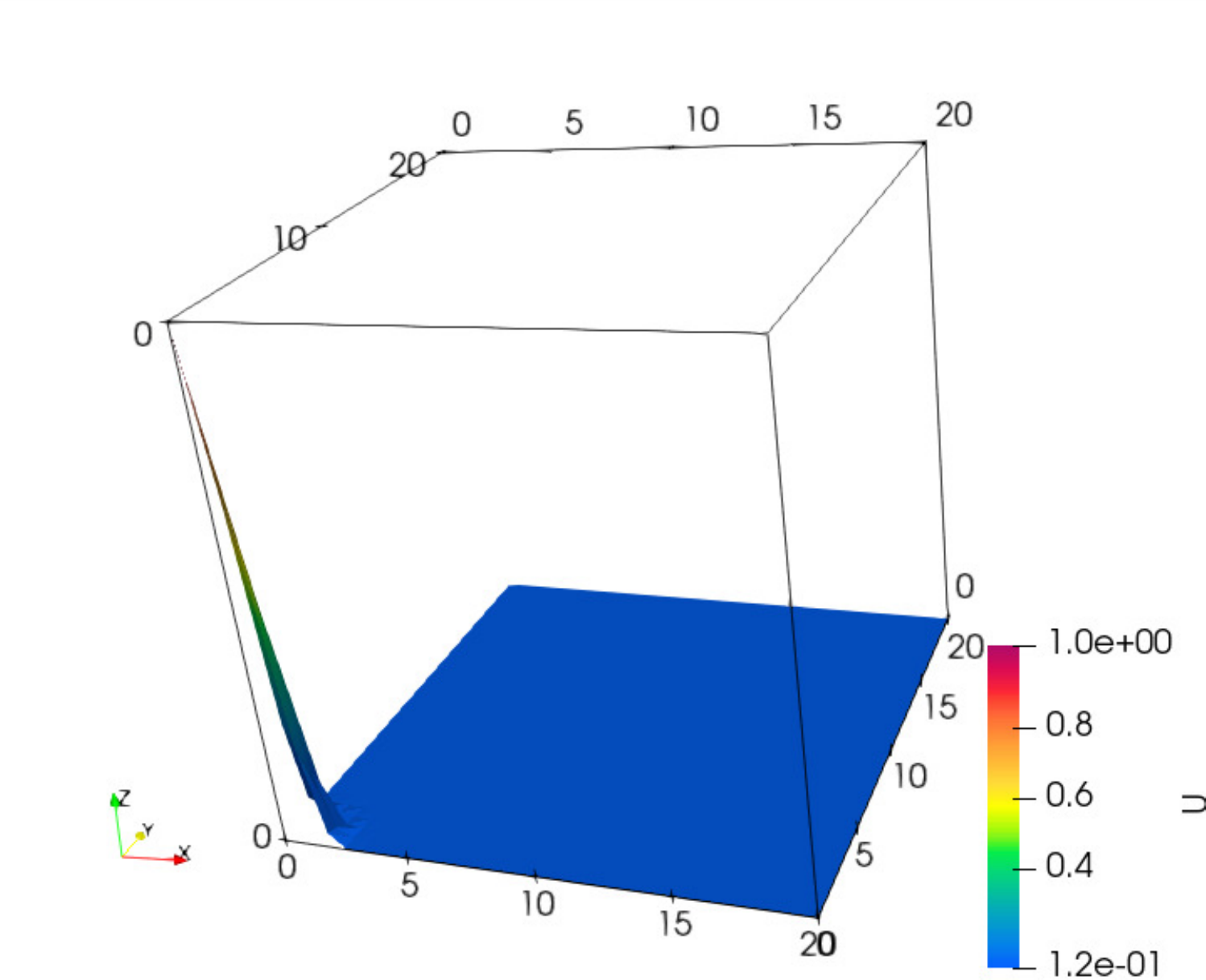}
		\caption{t = 0}
	\end{subfigure}
	\begin{subfigure}{0.24\textwidth}
		\includegraphics[width=\textwidth]{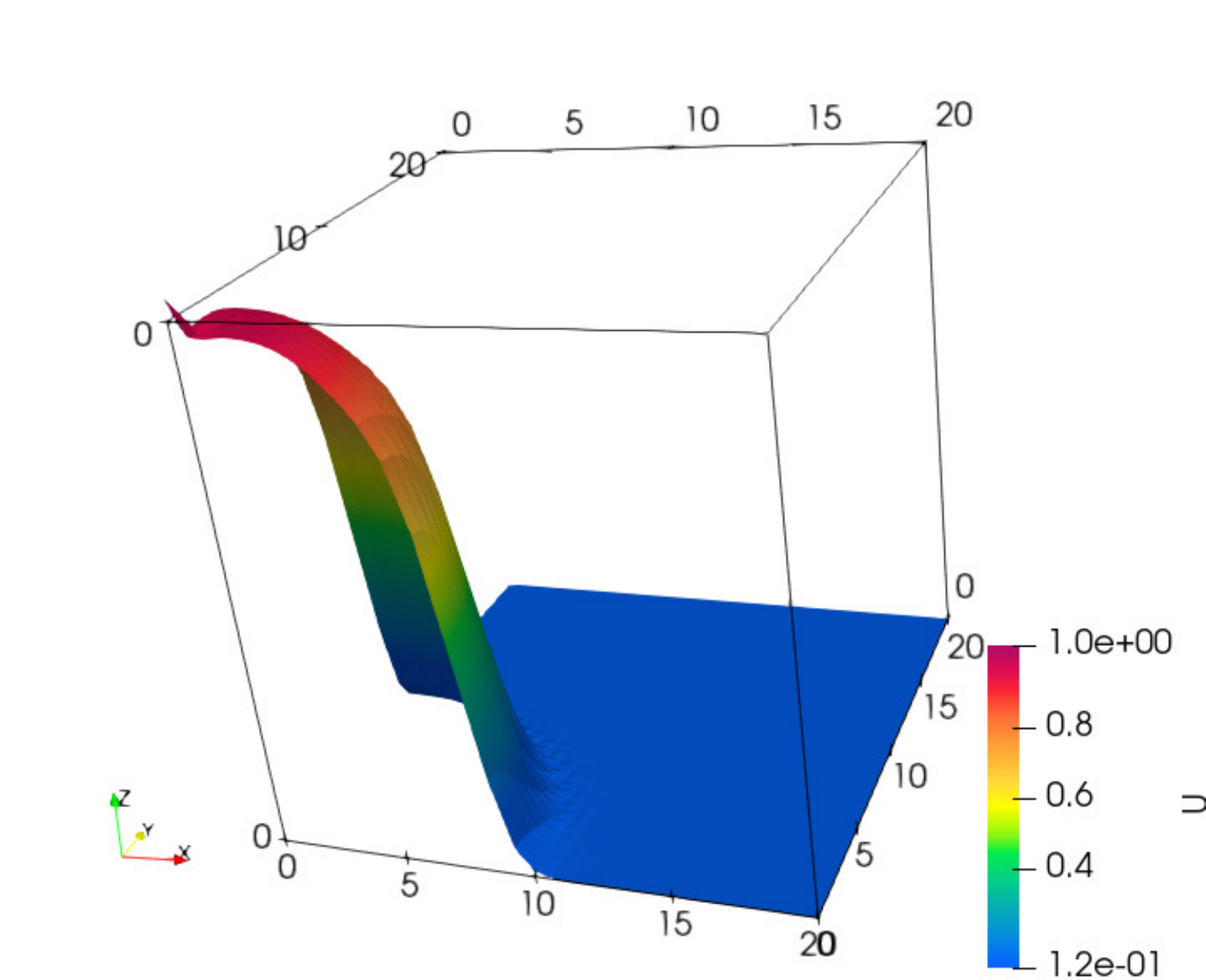}
		\caption{t = 10}
	\end{subfigure}
	\begin{subfigure}{0.24\textwidth}
		\includegraphics[width=\textwidth]{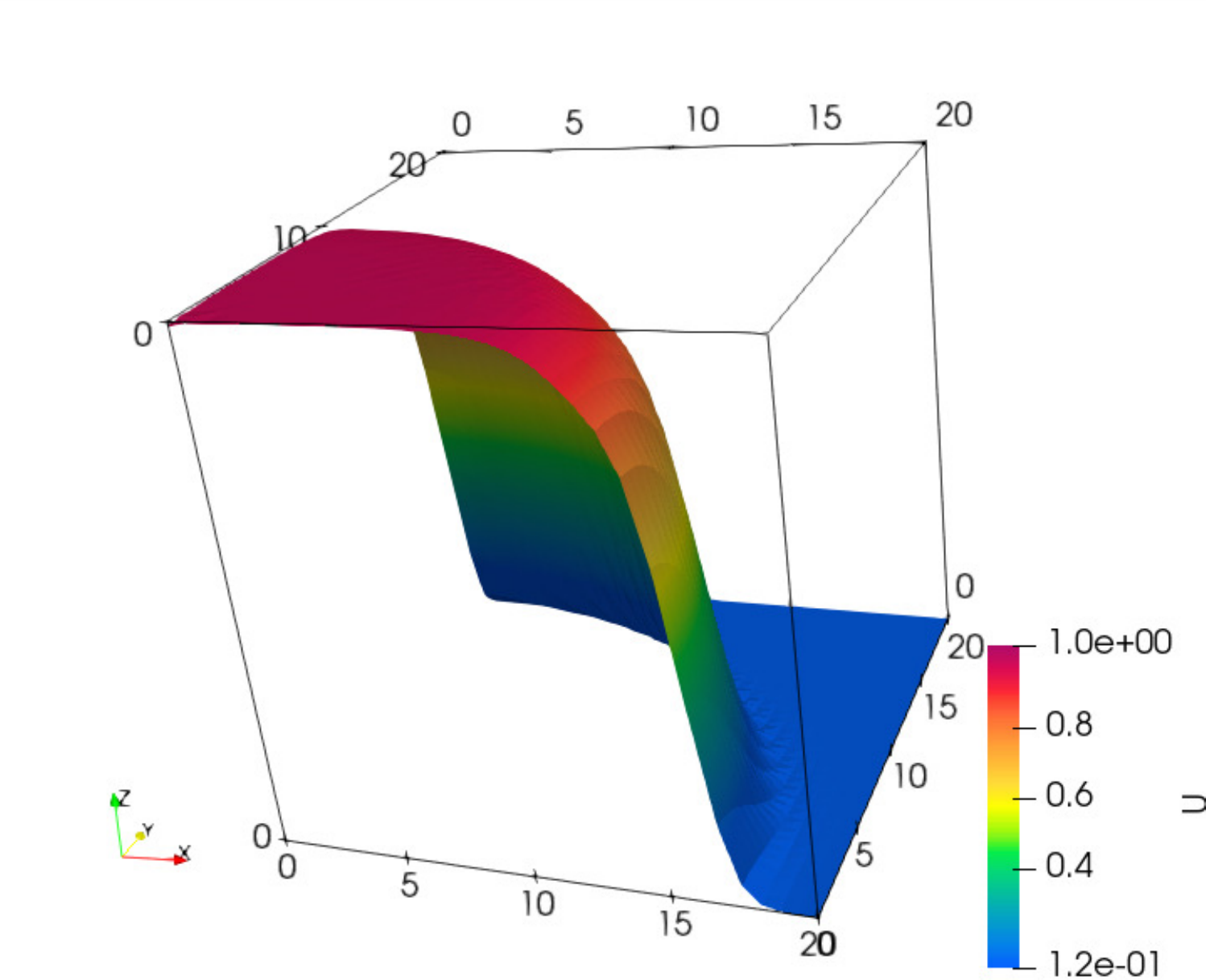}
		\caption{t = 20}
	\end{subfigure}
	\begin{subfigure}{0.24\textwidth}
		\includegraphics[width=\textwidth]{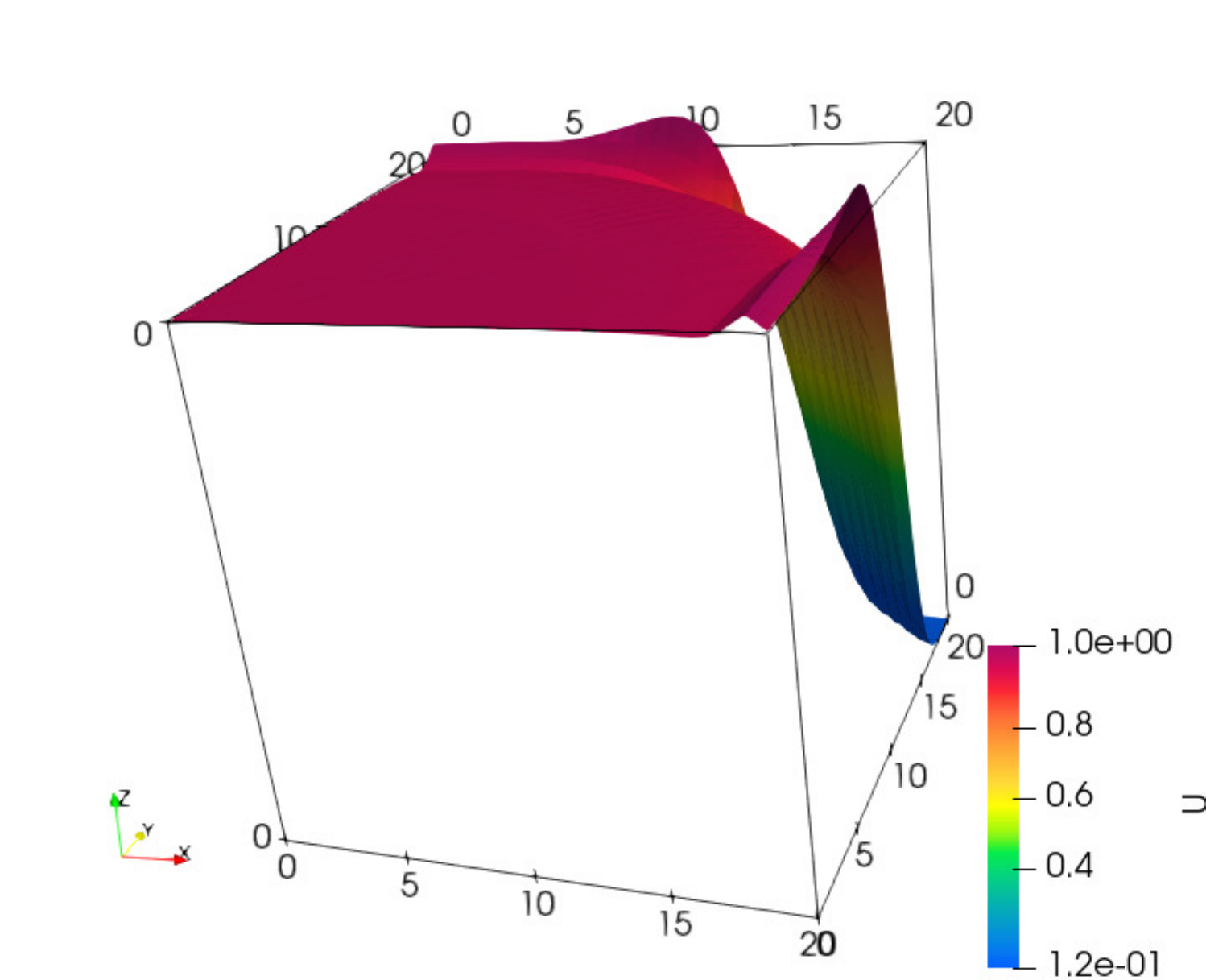}
		\caption{t = 30}
	\end{subfigure}
	\caption{
	Cancer cell invasion $u$ at different time instants $t= 0, 10, 20, 30$, obtained with the FEM-FCT scheme for $\mu = 1$ and $\chi=1$.
	 }
	\label{fig9}
\end{figure} 

\begin{figure}[H]
	\centering
	\begin{subfigure}{0.24\textwidth}
		\includegraphics[width=\textwidth]{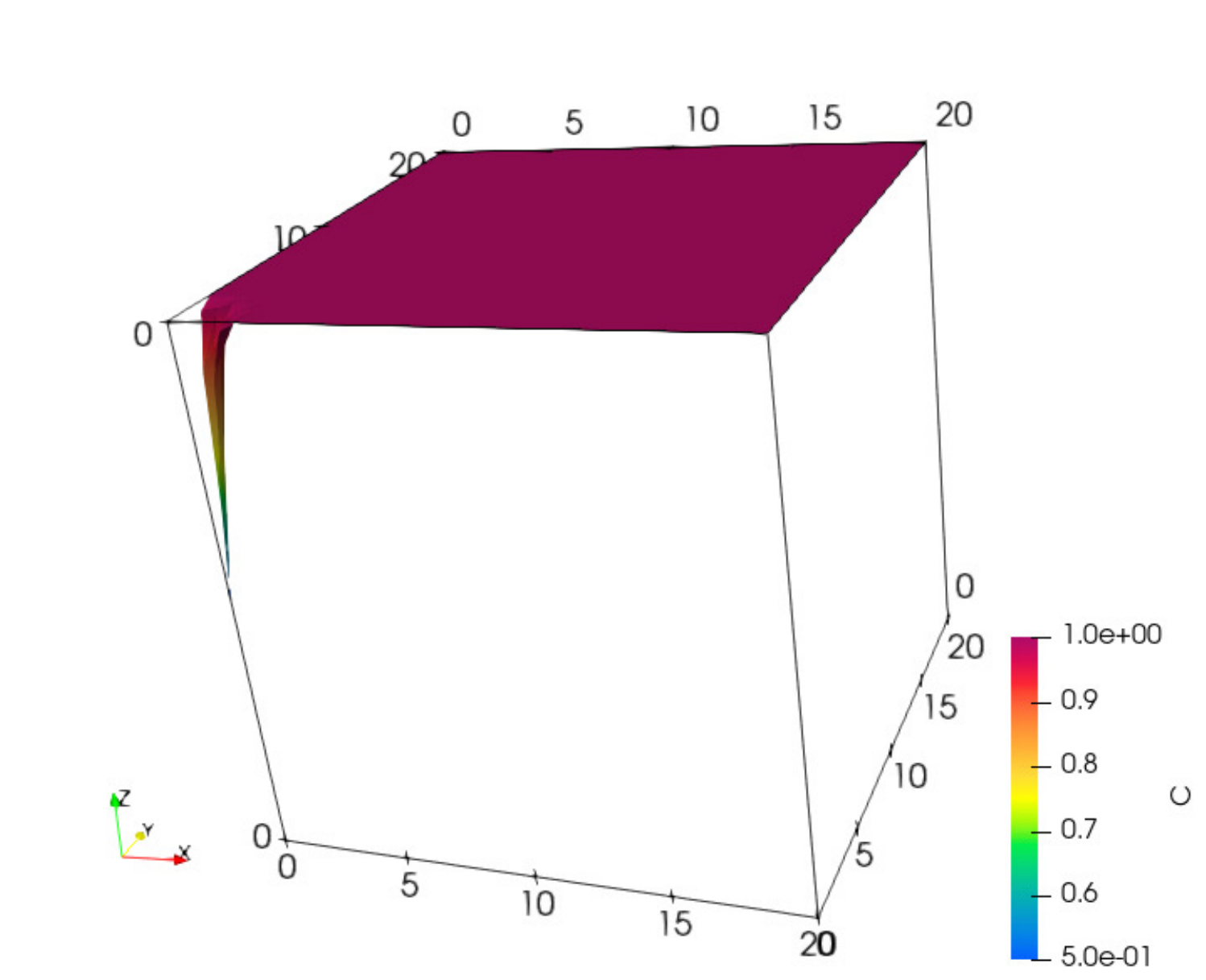}
		\caption{t = 0}
	\end{subfigure}
	\begin{subfigure}{0.24\textwidth}
		\includegraphics[width=\textwidth]{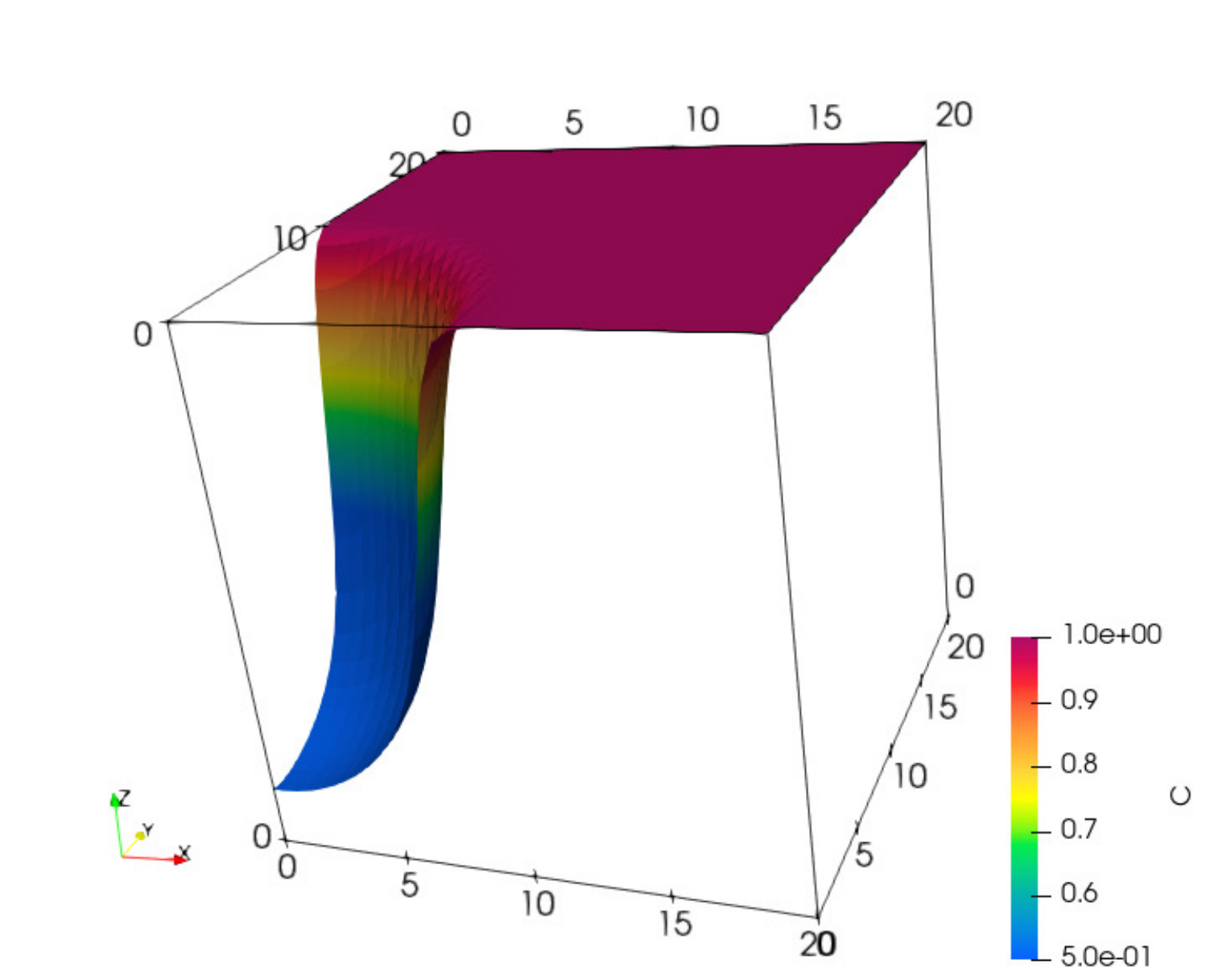}
		\caption{t = 10}
	\end{subfigure}
	\begin{subfigure}{0.24\textwidth}
		\includegraphics[width=\textwidth]{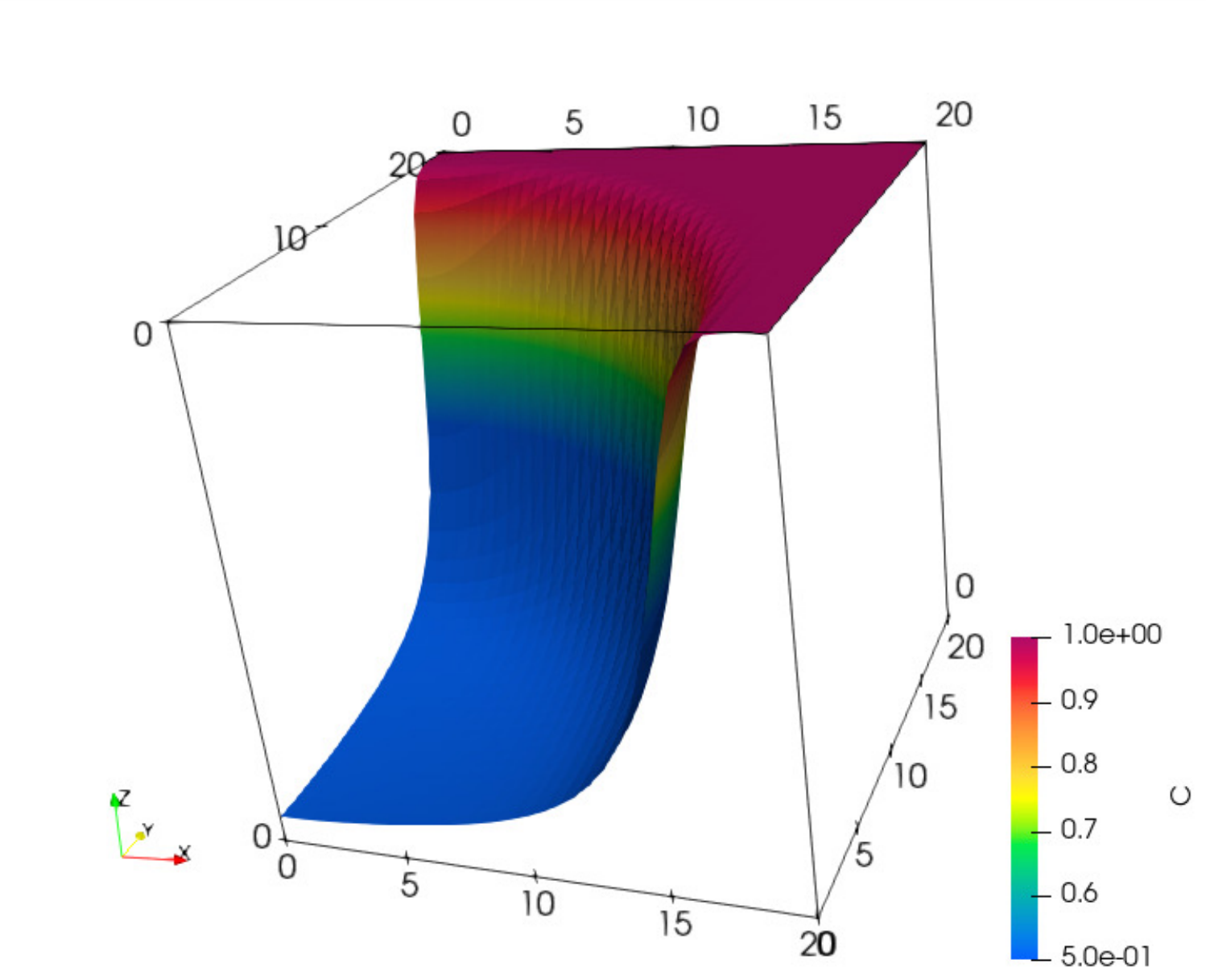}
		\caption{t = 20}
	\end{subfigure}
	\begin{subfigure}{0.24\textwidth}
		\includegraphics[width=\textwidth]{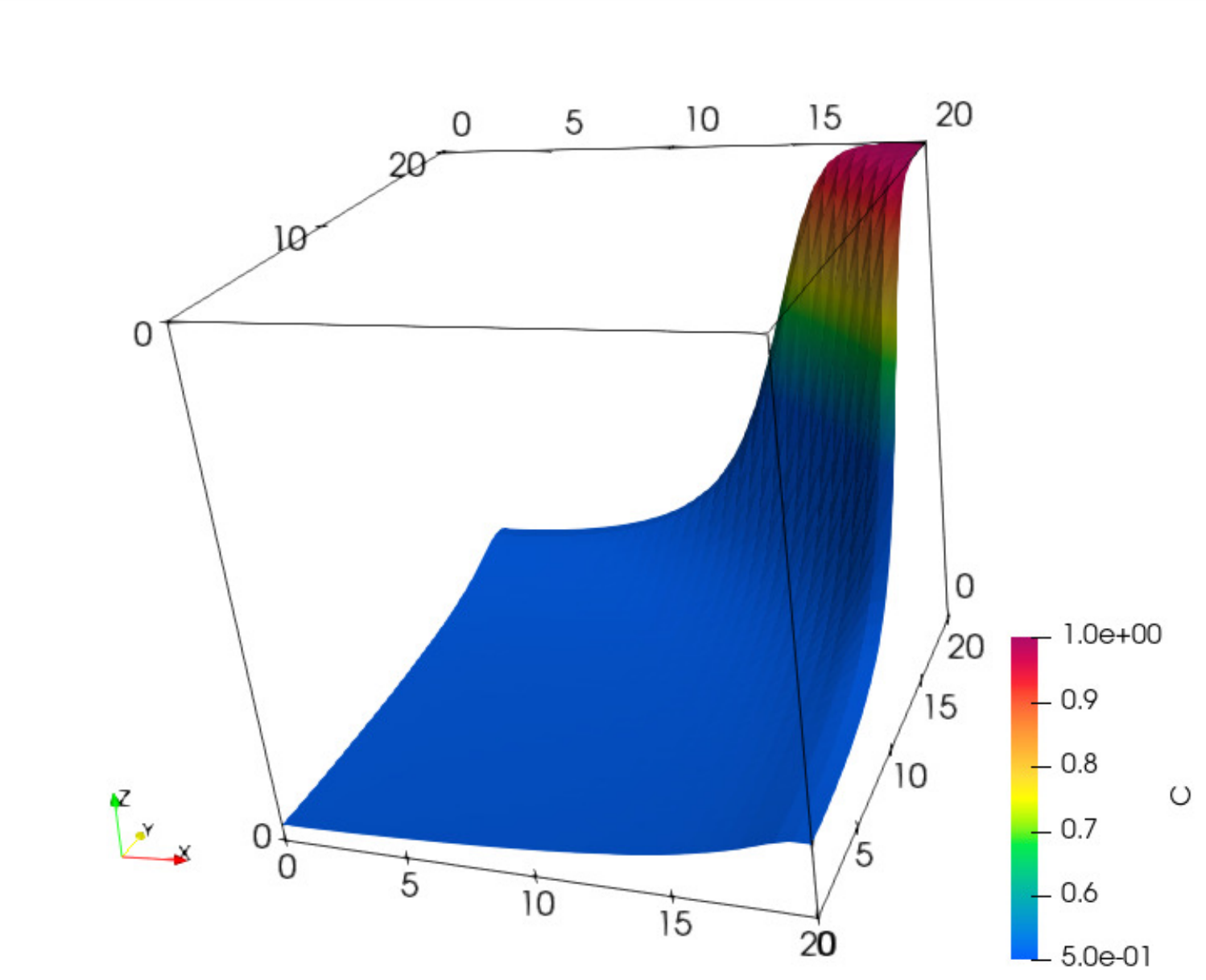}
		\caption{t = 30}
	\end{subfigure}
	\caption{
	Decay of the extracellular matrix $c$ at different time instants $t=0,
10, 20, 30$, computed using the FEM-FCT scheme for $\mu = 1$ and $\chi=1$.
	 }
	\label{fig10}
\end{figure} 

\begin{table}[h!]
\centering
\caption{Convergence of the mean values with respect to global mesh refinement 
at the last time instant $t=50$.}
 \begin{tabular}{|c c c c c c |} 
 \hline
 \# of refinements & 3 & 4 & 5 & 6 & 7 \\ 
 \hline
 \# DOF & 81 & 289 & 1089 & 4225 & 16641 \\  [0.5ex] 
 \hline\hline
$\int_{\Omega} c_h(x)  \,\mathrm{d}x$ & 0.02362193  & 0.02670467 & 0.03284535 & 0.03042843 & 0.03680451\\ 
$\int_{\Omega} p_h(x) \,\mathrm{d}x$& 0.02373726  & 0.02685417 & 0.03308441 & 0.03062835 & 0.03712137\\
$\int_{\Omega} u_h(x) \,\mathrm{d}x$& 0.99999999  & 0.99999998  & 0.99999976  & 0.99999943 & 0.99999889 \\  [0.5ex] 
 \hline
 \end{tabular}
 \label{tab2}
\end{table}

\begin{table}[h!]
\centering
\caption{Convergence of the solutions at the point $(20,20)$ 
with respect to the time step $\tau$ and convergence of the fixed-point 
iterations, both at the last time instant $t=50$. }
    \begin{tabular} { | c | c c c c c c c | }
    \hline
    $\tau$ & $c^{n+1}_{k}$ & $\Vert c^{n+1}_{k} - c^{n+1}_{k-1}\Vert$ &
$p^{n+1}_{k}$ & $\Vert p^{n+1}_{k} - p^{n+1}_{k-1}\Vert$ & $u^{n+1}_{k}$ & $\Vert
u^{n+1}_{k} - u^{n+1}_{k-1}\Vert$ &  \# Iterations   \\
    \hline
    1.0    & 0.0333129 & 4.6407721e-09 & 0.0335600 & 5.0321162e-09 & 1.0004557 & 6.9181459e-11 & 21 \\
    0.1    & 0.0388396 & 5.5772037e-09 & 0.0391805 & 5.7753161e-09 & 1.0007926 & 4.4710309e-11 & 18\\
    0.01   & 0.0387508 & 9.0822211e-09 & 0.0390900 & 9.2197598e-09 & 1.0007004 & 1.4400819e-10 & 14\\
    0.001  & 0.0387514 & 7.2775629e-09 & 0.0390907 & 7.3738174e-09 & 1.0007904 & 1.2164834e-10 & 11\\
    0.0001 & 0.0387516 & 5.8227737e-09 & 0.0390908 & 5.8989104e-09 & 1.0007904 & 9.7728734e-11 & 8\\
    \hline
    \end{tabular}
       \label{tab3}
\end{table}

\begin{figure}[H]
	\centering
	\begin{subfigure}{0.24\textwidth}
		\includegraphics[width=\textwidth]{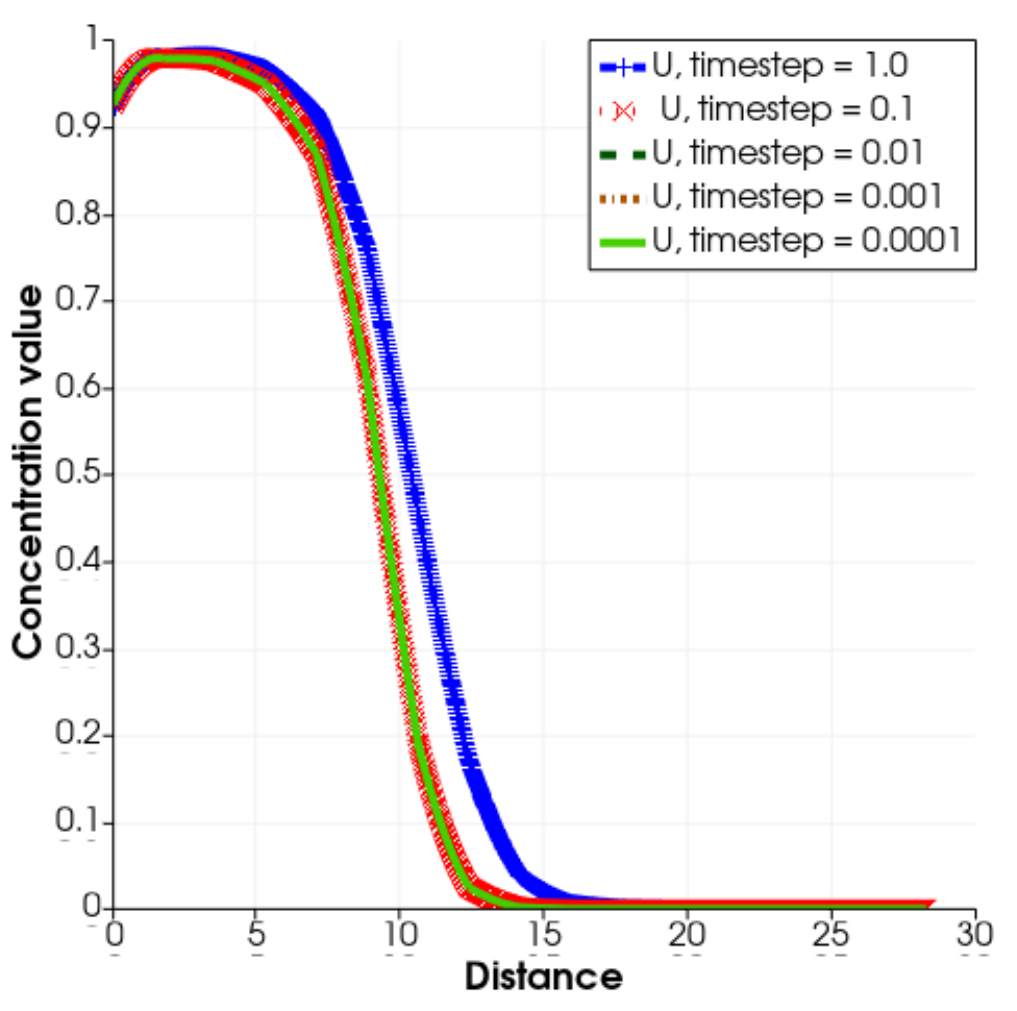}
		\caption{t = 10}
	\end{subfigure}
	\begin{subfigure}{0.24\textwidth}
		\includegraphics[width=\textwidth]{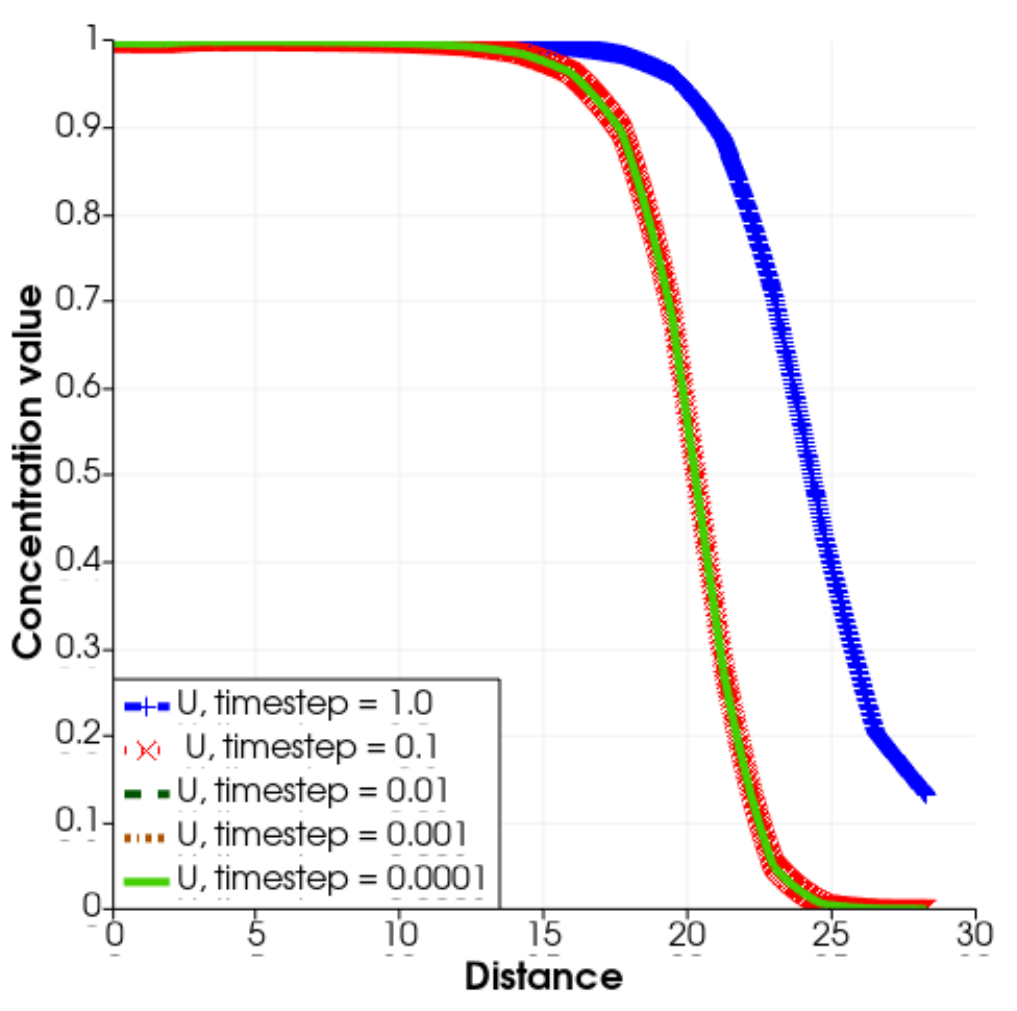}
		\caption{t = 20}
	\end{subfigure}
	\begin{subfigure}{0.24\textwidth}
		\includegraphics[width=\textwidth]{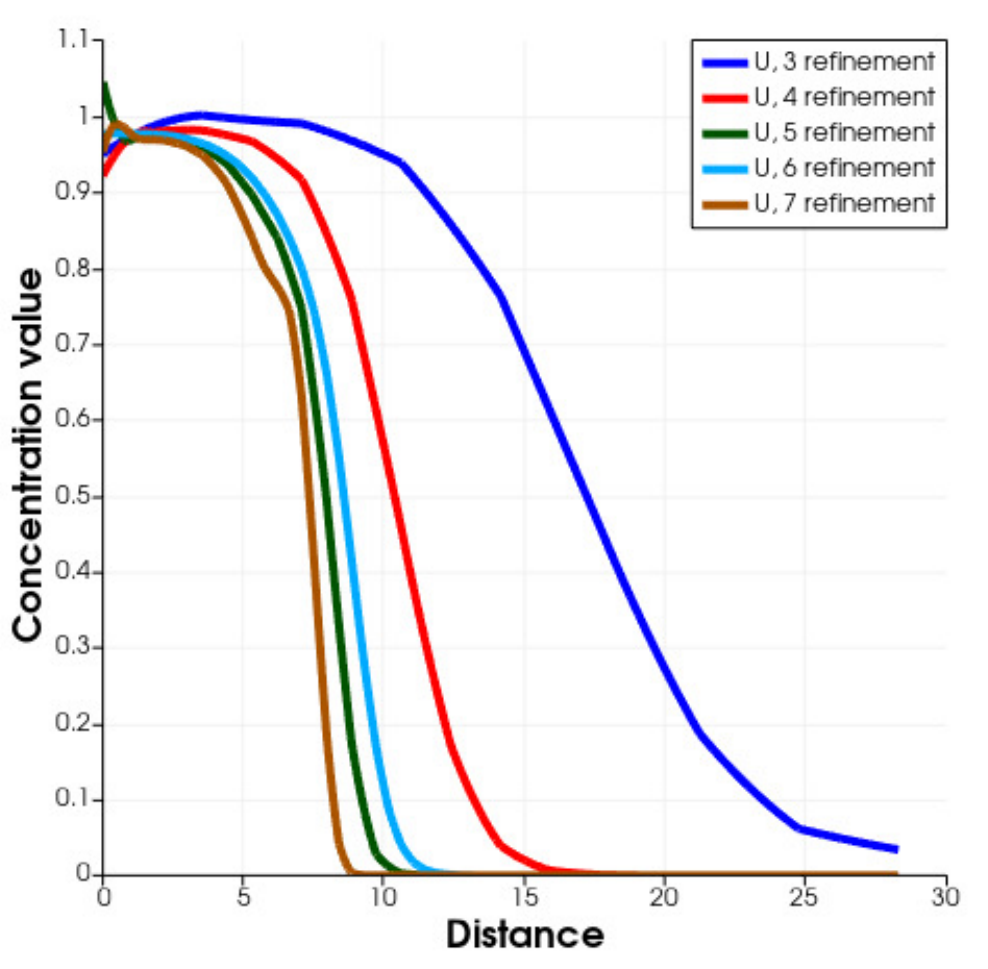}
		\caption{t = 10}
	\end{subfigure}
	\begin{subfigure}{0.24\textwidth}
		\includegraphics[width=\textwidth]{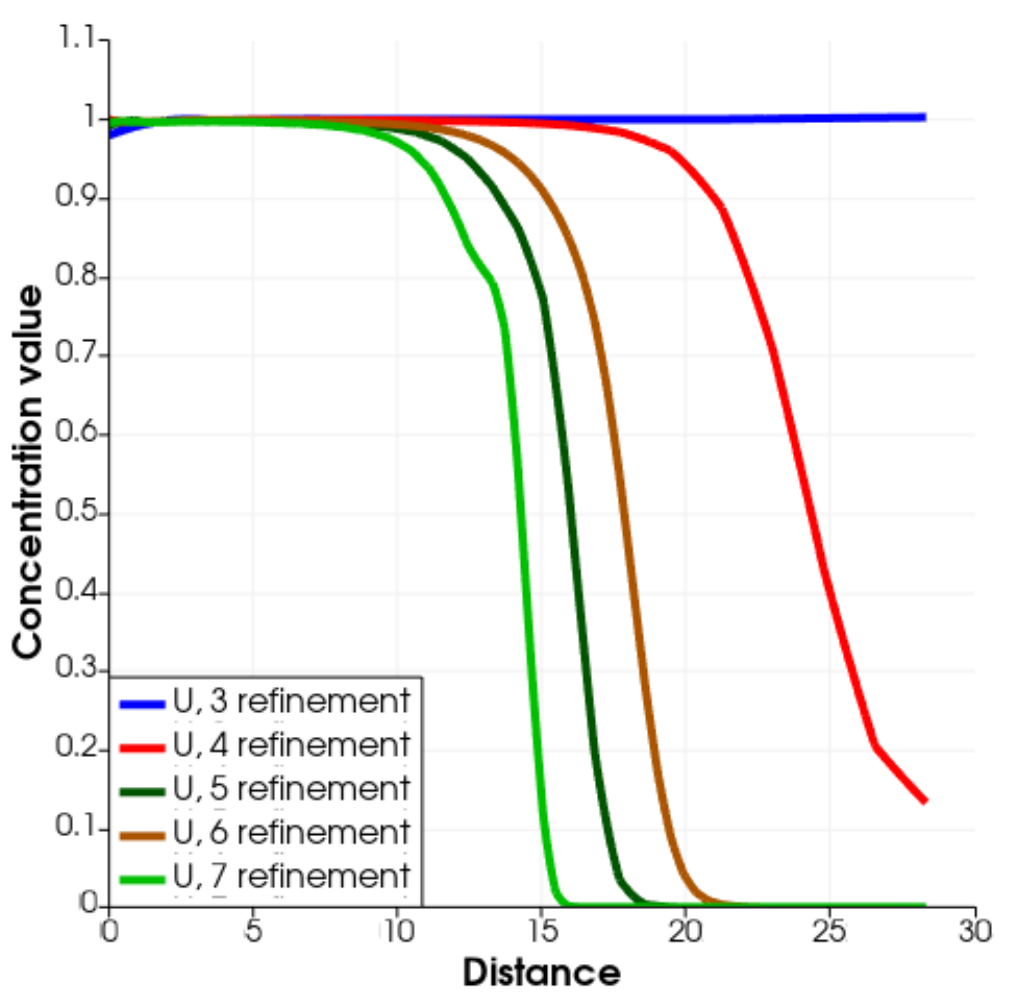}
		\caption{t = 20}
	\end{subfigure}
	\caption{
	Cancer cell invasion $u$ computed using the FEM-FCT scheme at time
instants $t=10$ and $t=20$ for different time steps (first two pictures) and
for different numbers of global refinements (last two pictures).
	 }
	\label{fig11}
\end{figure} 

 \subsection{Effect of haptotactic domination}\label{sec:hapto}
In this section, we investigate the effect of directional movement of cancer
cells inside the domain. This is a very important property in cancer
modeling which can lead to metastasis. In metastasis, the cancer cells are
moving to the other parts of the body and start proliferate, forming a new tumor
in the new part, and invade the surrounding  tissues. In this case, it is very
difficult to detect the location of cancerous cells  and this is one of the
predominant causes of most deaths due to cancer. In the following, we only study
a very simple case of haptotactic dominating mechanism of the cancer cell
motion. In addition to the absence of the diffusion effect in the system, there 
is only a small amount of the proliferation rate: we set $\mu = 0.0001$ and 
$\chi =1$ in the computations. As a result of the haptotactic migration
domination, a small cluster of cancer cells builds up at the beginning and this
initial amount is expected to move along the direction of the gradient of the
extracellular matrix. As Fig.~\ref{fig12} indicates, the numerical simulation
by the standard Galerkin FEM breaks down in a very short amount of time
after the time instant $t=15$. Next, we apply the FEM-FCT scheme and the
low-order method, see Figs.~\ref{fig14} and \ref{fig13}, respectively. We 
observe that, in both cases, the stabilization prevents the blow-up in the 
system and leads to non-negative solutions. However, some oscillations still 
remain in the interior layer. These oscillations could be suppressed by
adaptive mesh refinement, which is however out of the scope of this paper.
As expected, the low-order method provides a more diffusive solution than the
FEM-FCT scheme. It is interesting that, combining the FEM-FCT scheme with the 
backward Euler method ($\theta=1$), oscillation-free solutions are obtained,
see Fig.~\ref{fig16}. 

 \begin{figure}[H]
	\begin{subfigure}{0.24\textwidth}
		\includegraphics[width=\textwidth]{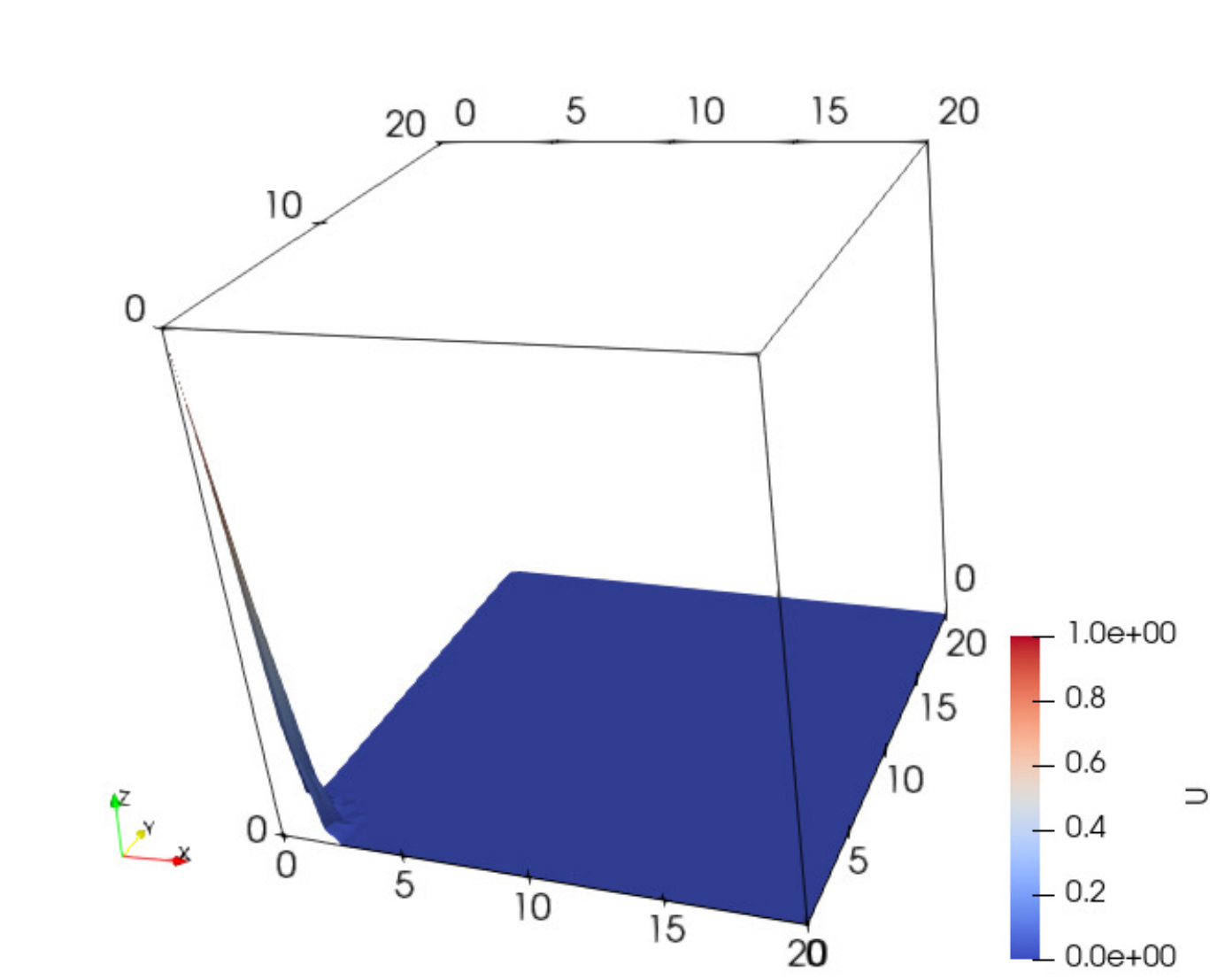}
		\caption{t = 0}
	\end{subfigure}
	\centering
	\begin{subfigure}{0.24\textwidth}
		\includegraphics[width=\textwidth]{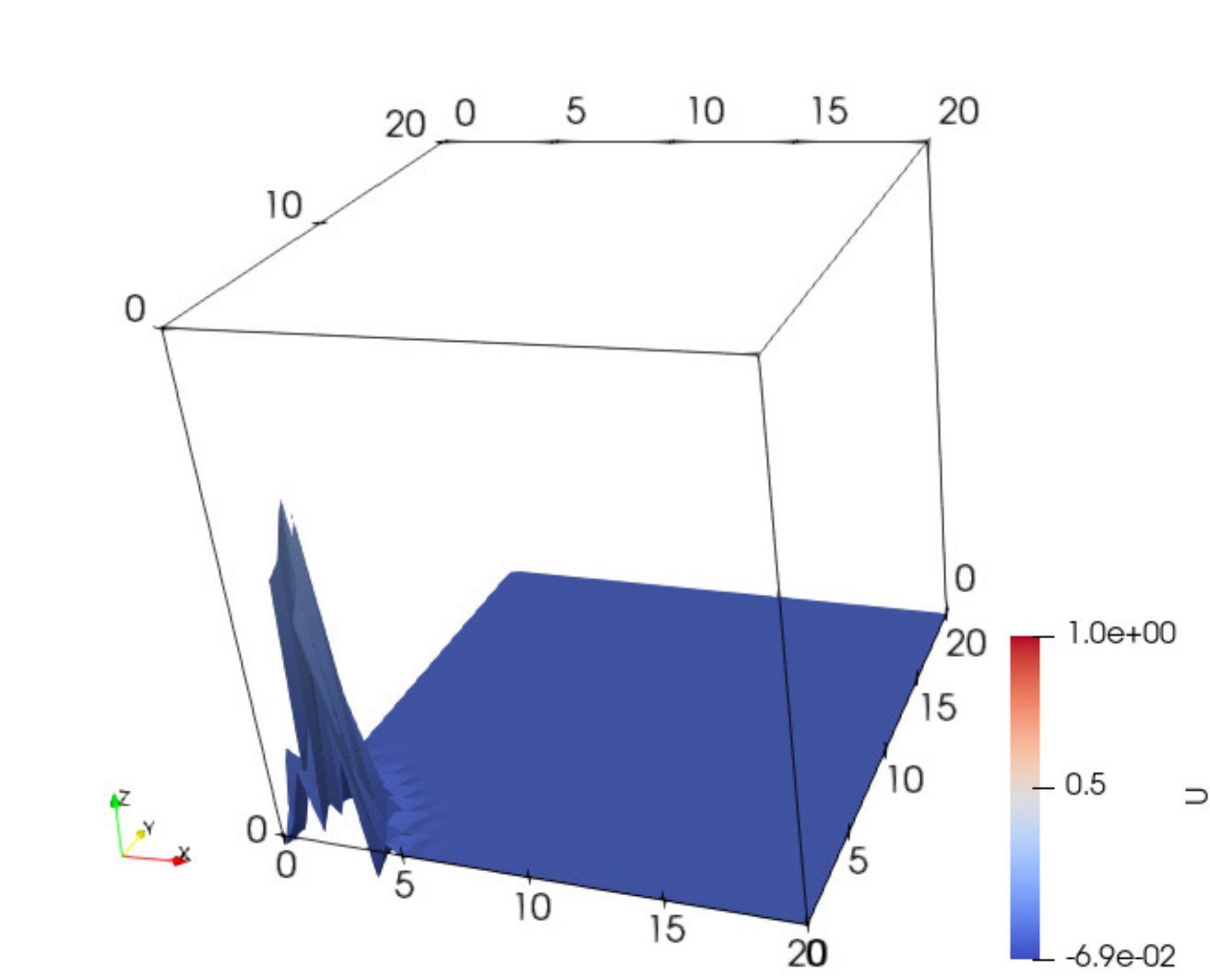}
		\caption{t = 5}
	\end{subfigure}
	\begin{subfigure}{0.24\textwidth}
		\includegraphics[width=\textwidth]{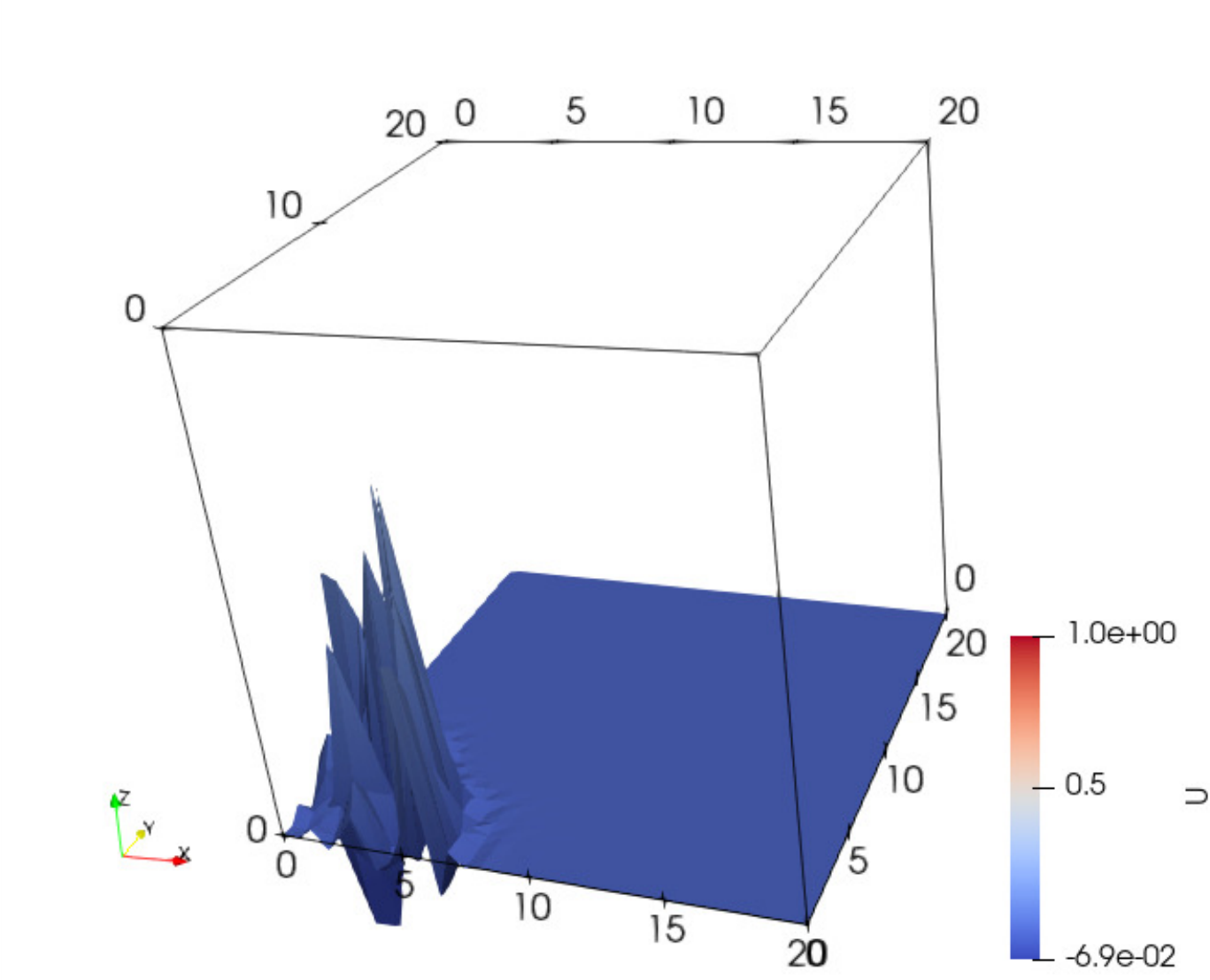}
		\caption{t = 15}
	\end{subfigure}
	\caption{
	The effect of the haptotactic rate on the cancer cell invasion $u$ at
different time instants $t= 0, 5, 15$, computed using the standard Galerkin
FEM for $\mu = 0.0001$ and $\chi=1$.
	 }
	\label{fig12}
\end{figure}  

 \begin{figure}[H]
	\centering
	\begin{subfigure}{0.24\textwidth}
		\includegraphics[width=\textwidth]{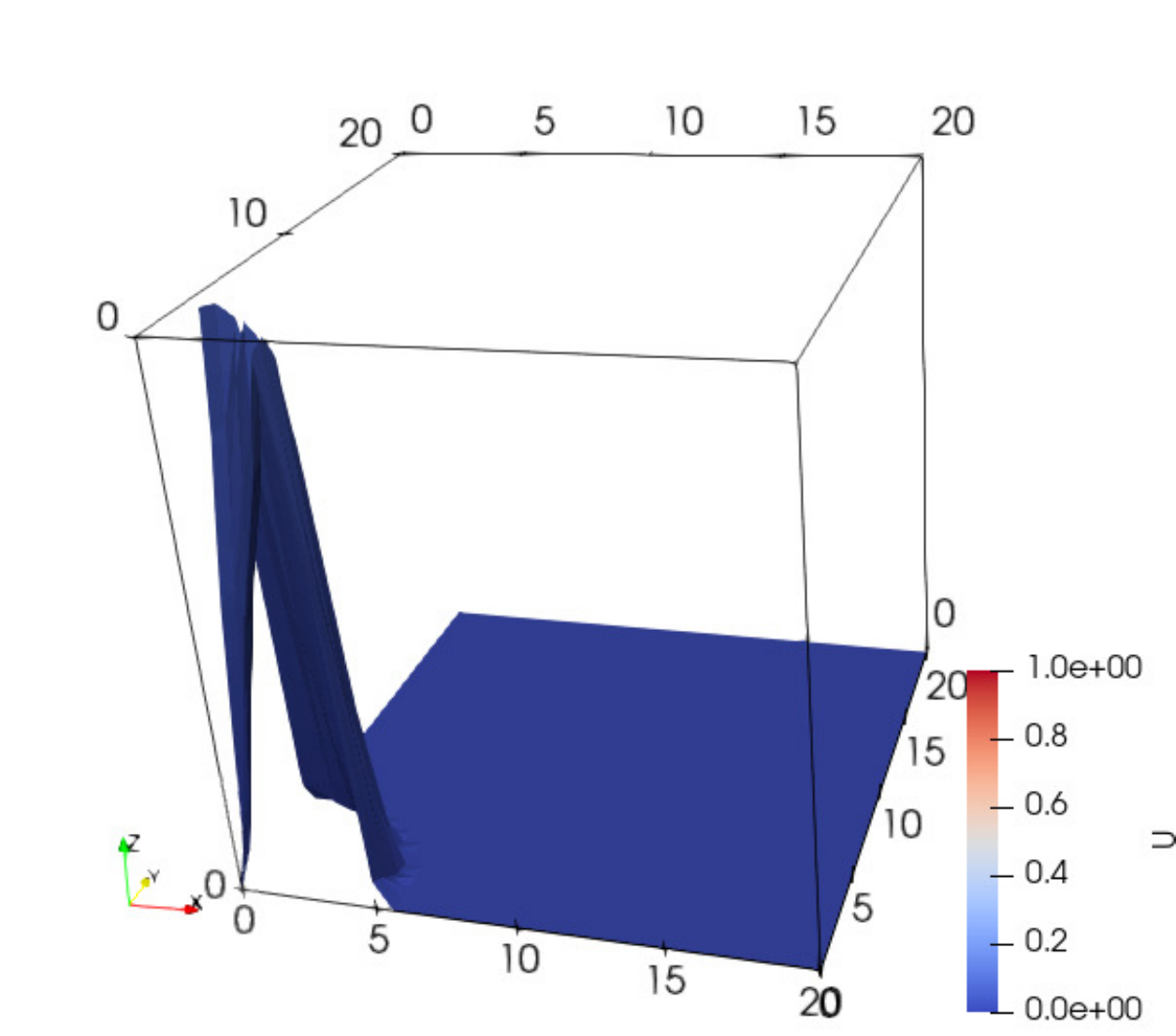}
		\caption{t = 10}
	\end{subfigure}
	\begin{subfigure}{0.24\textwidth}
		\includegraphics[width=\textwidth]{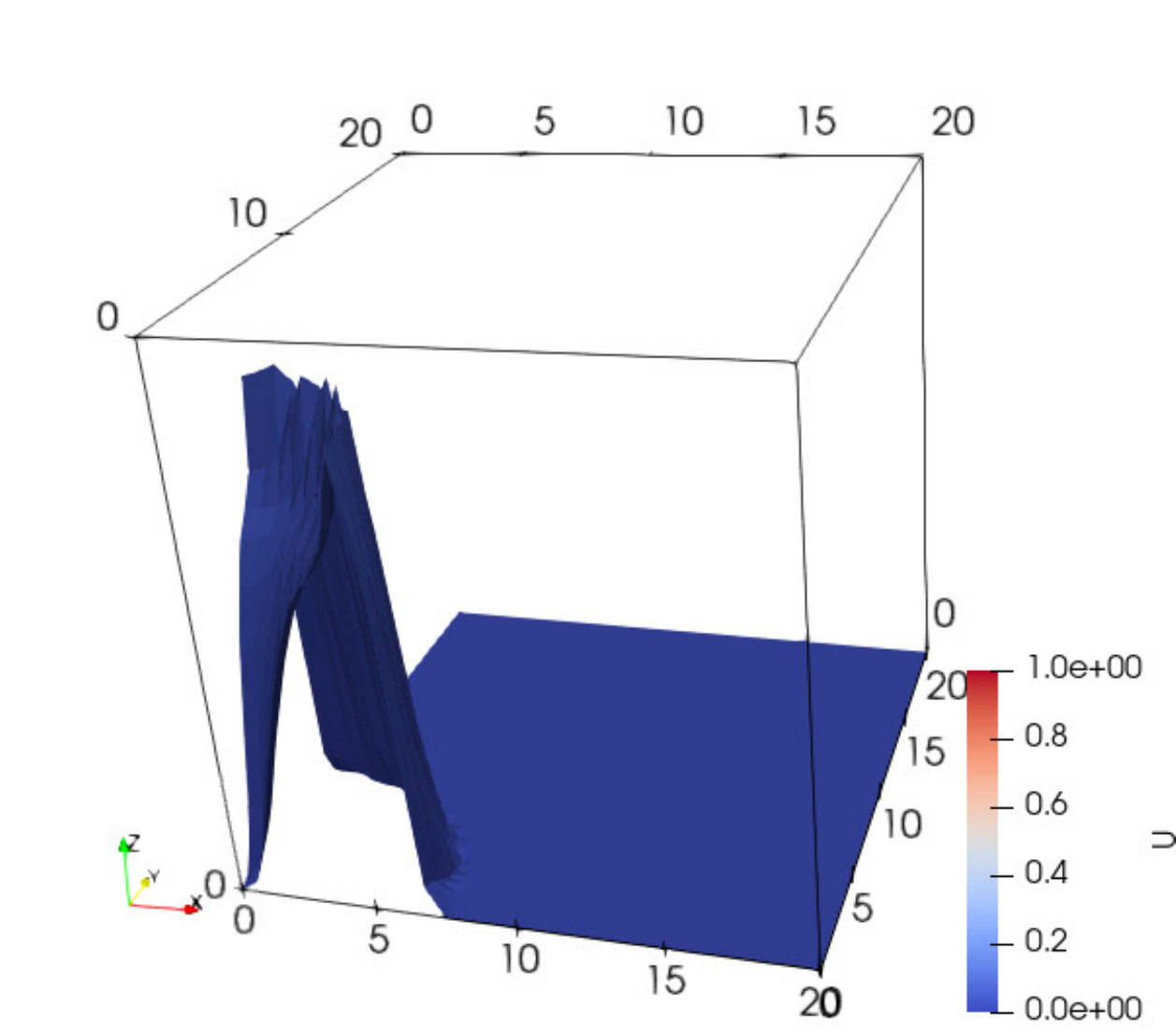}
		\caption{t = 20}
	\end{subfigure}
	\begin{subfigure}{0.24\textwidth}
		\includegraphics[width=\textwidth]{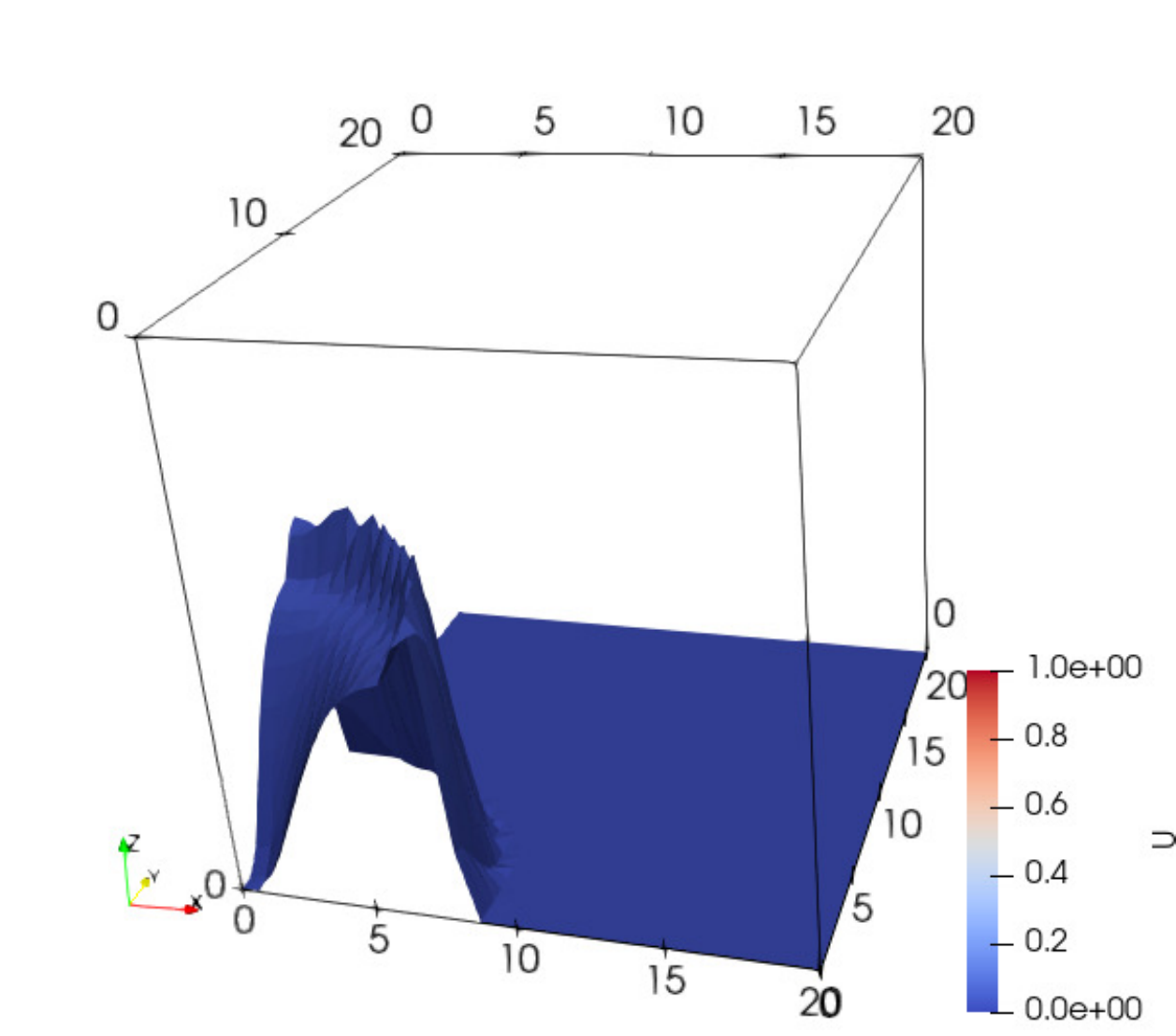}
		\caption{t = 30}
	\end{subfigure}
	\begin{subfigure}{0.24\textwidth}
		\includegraphics[width=\textwidth]{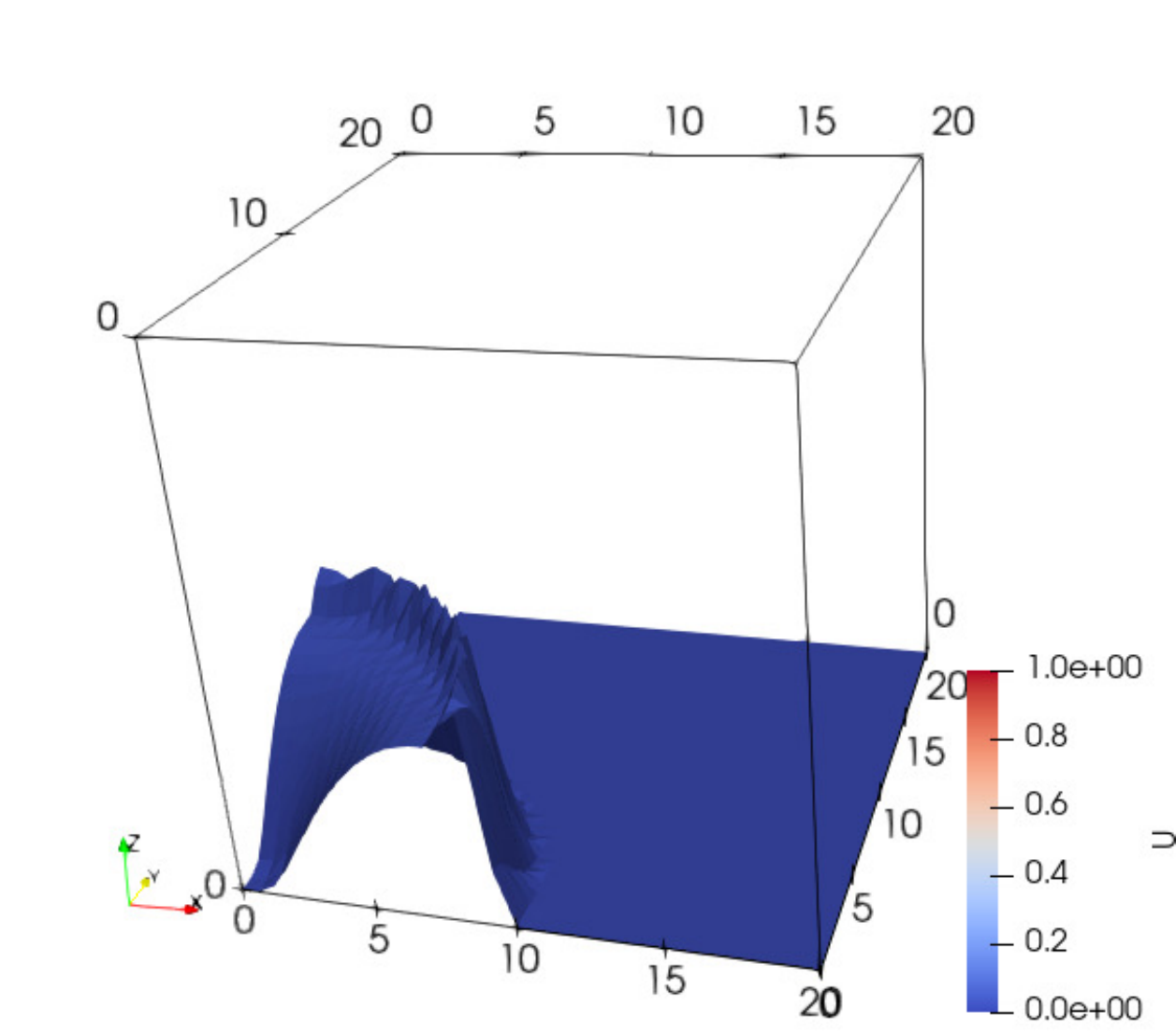}
		\caption{t = 40}
	\end{subfigure}
	\caption{
	The effect of the haptotactic rate on the cancer cell invasion $u$ at
different time instants $t=10, 20, 30, 40$, computed using the FEM-FCT scheme
for $\mu = 0.0001$ and $\chi=1$.
	 }
	\label{fig14}
\end{figure}

 \begin{figure}[H]
	\centering
	\begin{subfigure}{0.24\textwidth}
		\includegraphics[width=\textwidth]{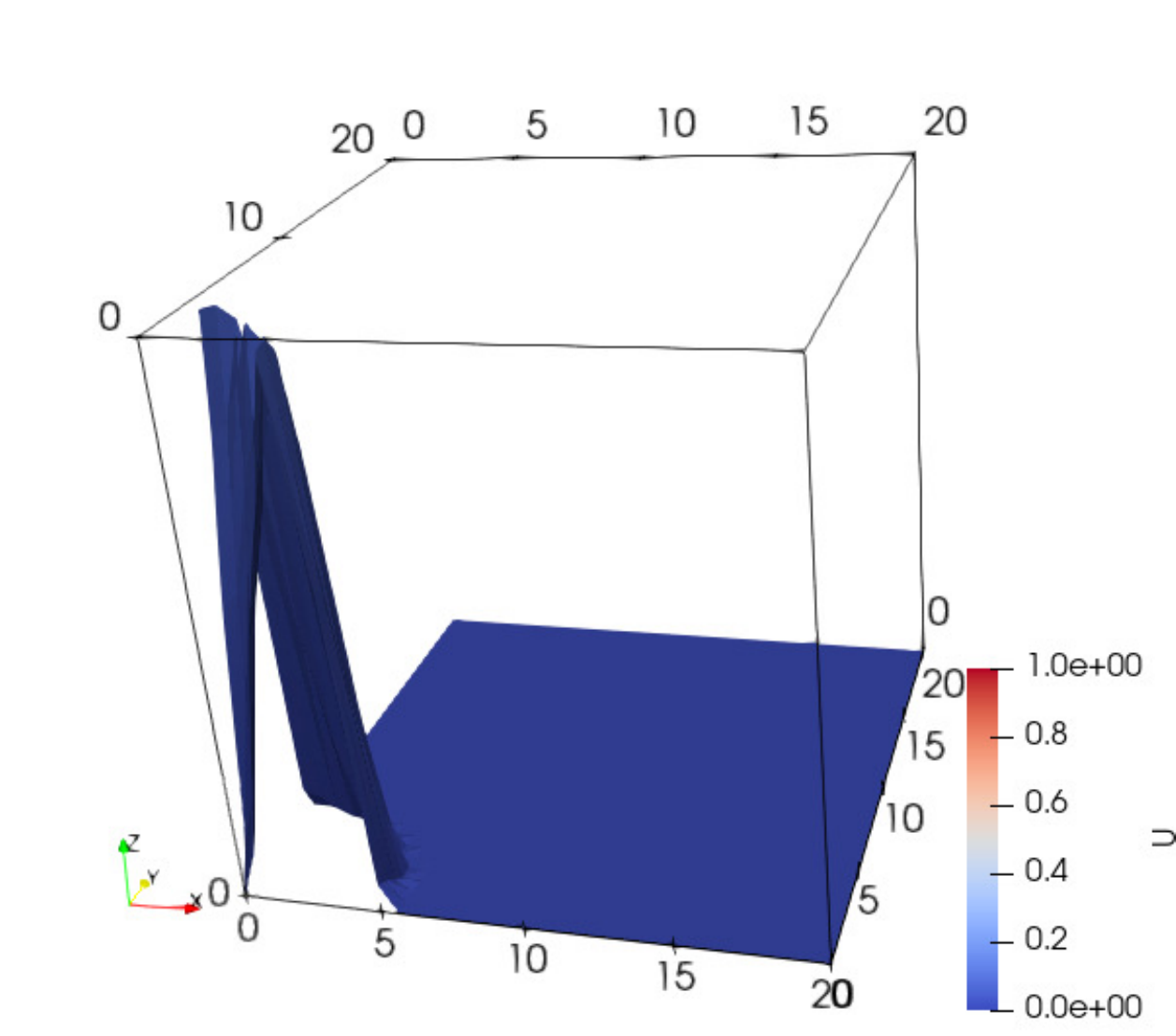}
		\caption{t = 10}
	\end{subfigure}
	\begin{subfigure}{0.24\textwidth}
		\includegraphics[width=\textwidth]{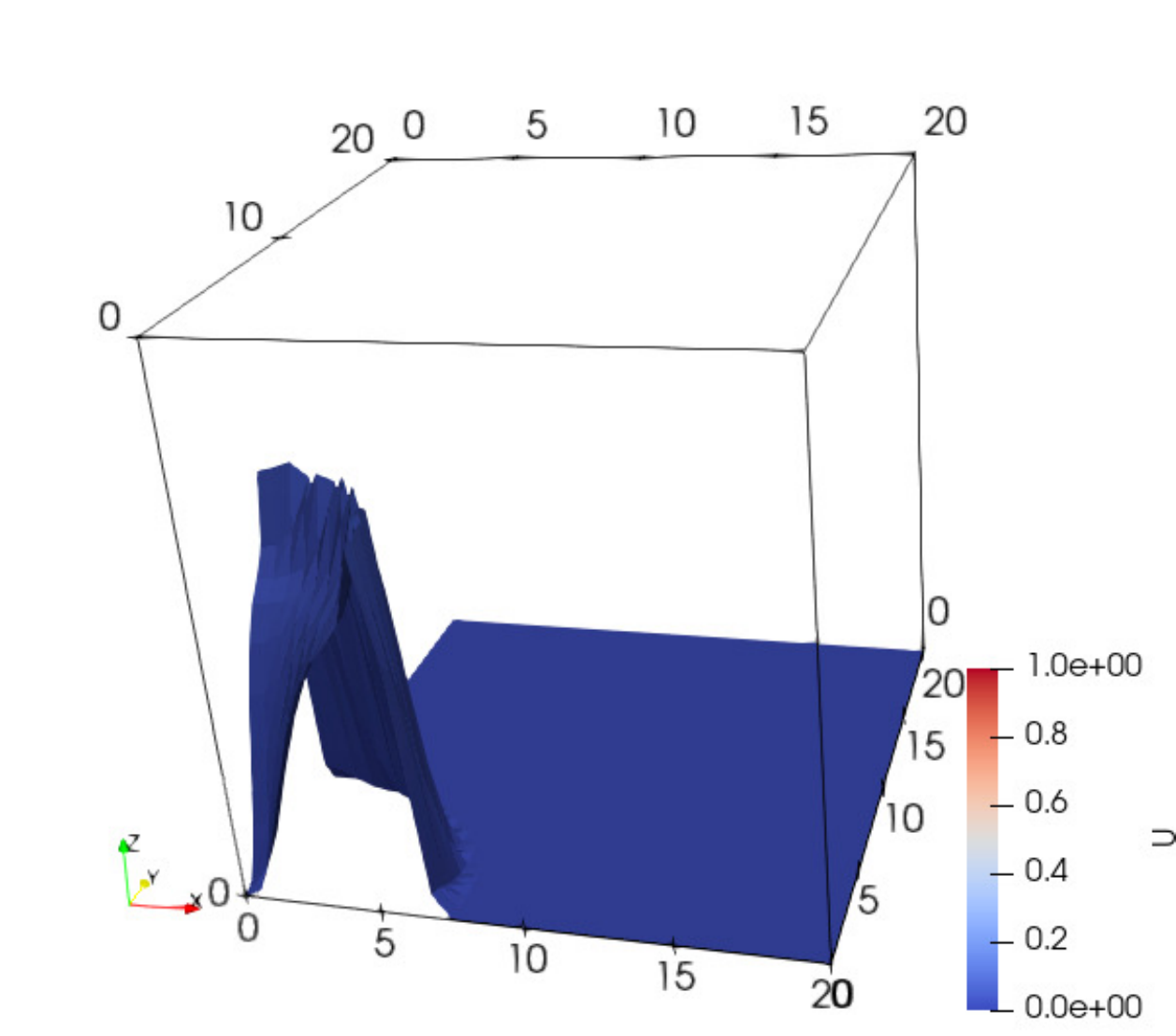}
		\caption{t = 20}
	\end{subfigure}
	\begin{subfigure}{0.24\textwidth}
		\includegraphics[width=\textwidth]{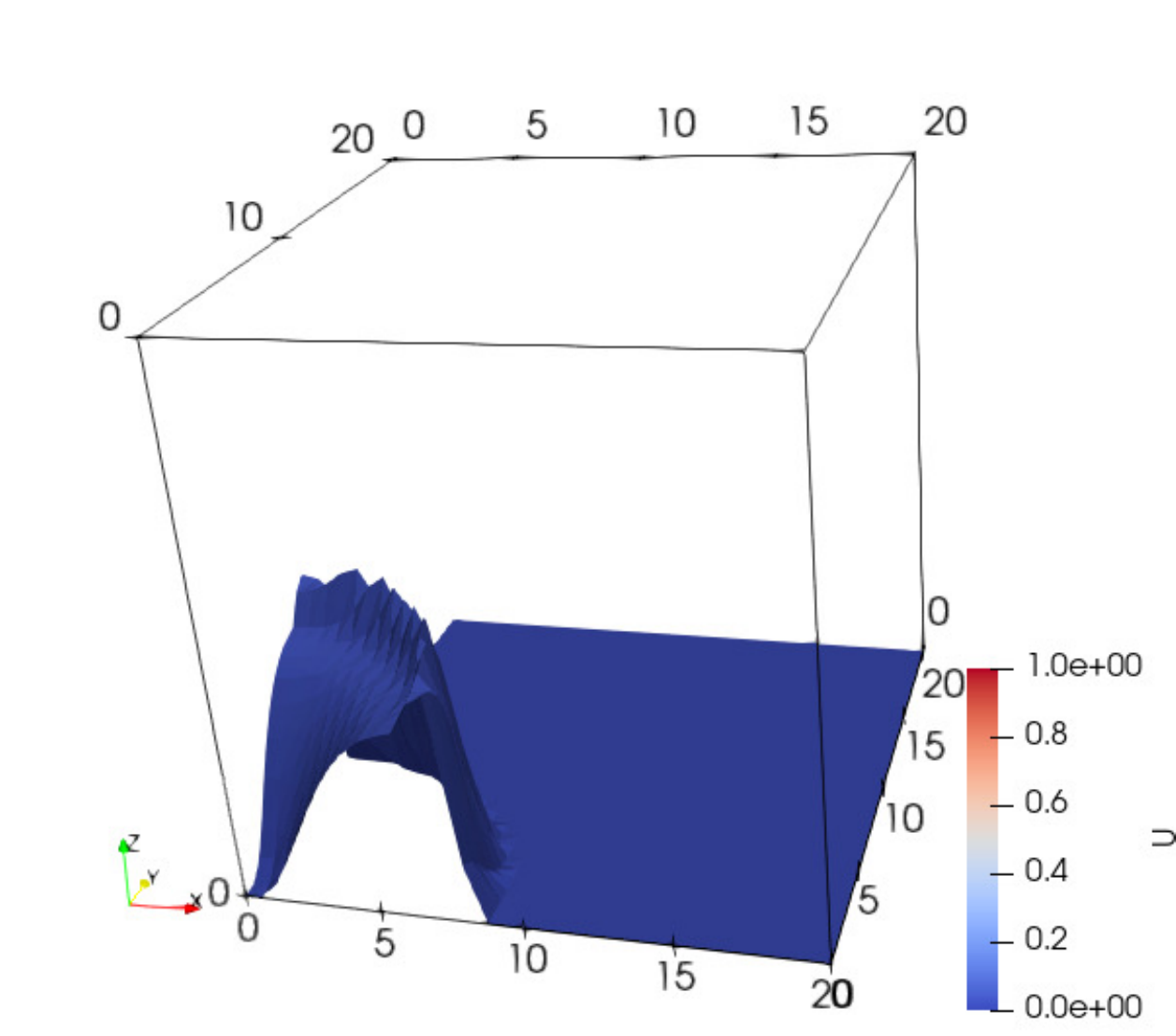}
		\caption{t = 30}
	\end{subfigure}
	\begin{subfigure}{0.24\textwidth}
		\includegraphics[width=\textwidth]{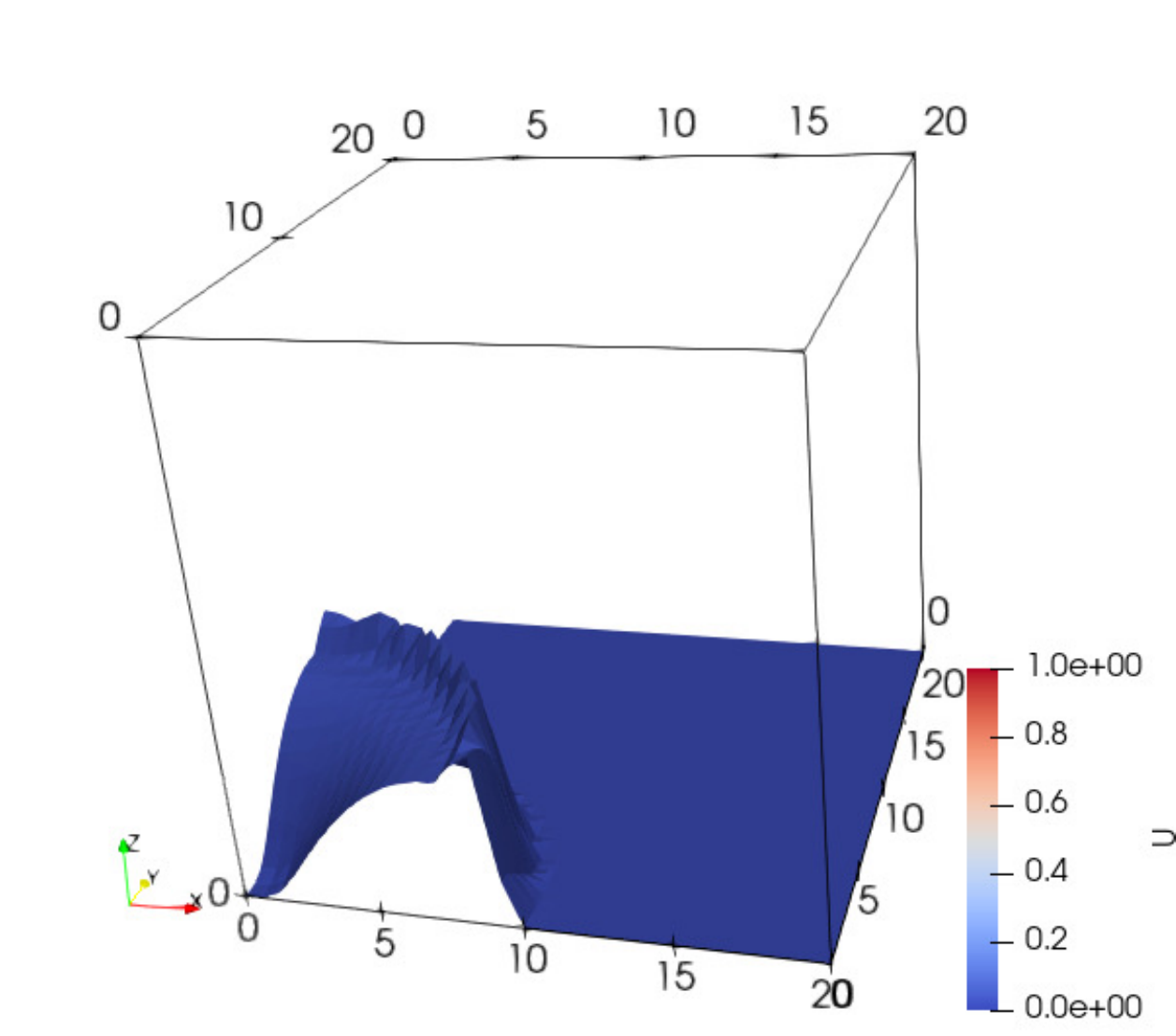}
		\caption{t = 40}
	\end{subfigure}
	\caption{
	The effect of the haptotactic rate on the cancer cell invasion $u$ at
different time instants $t=10, 20, 30, 40$, computed using the low-order
method for $\mu = 0.0001$ and $\chi=1$.
	 }
	\label{fig13}
\end{figure}

 \begin{figure}[H]
	\centering
	\begin{subfigure}{0.24\textwidth}
		\includegraphics[width=\textwidth]{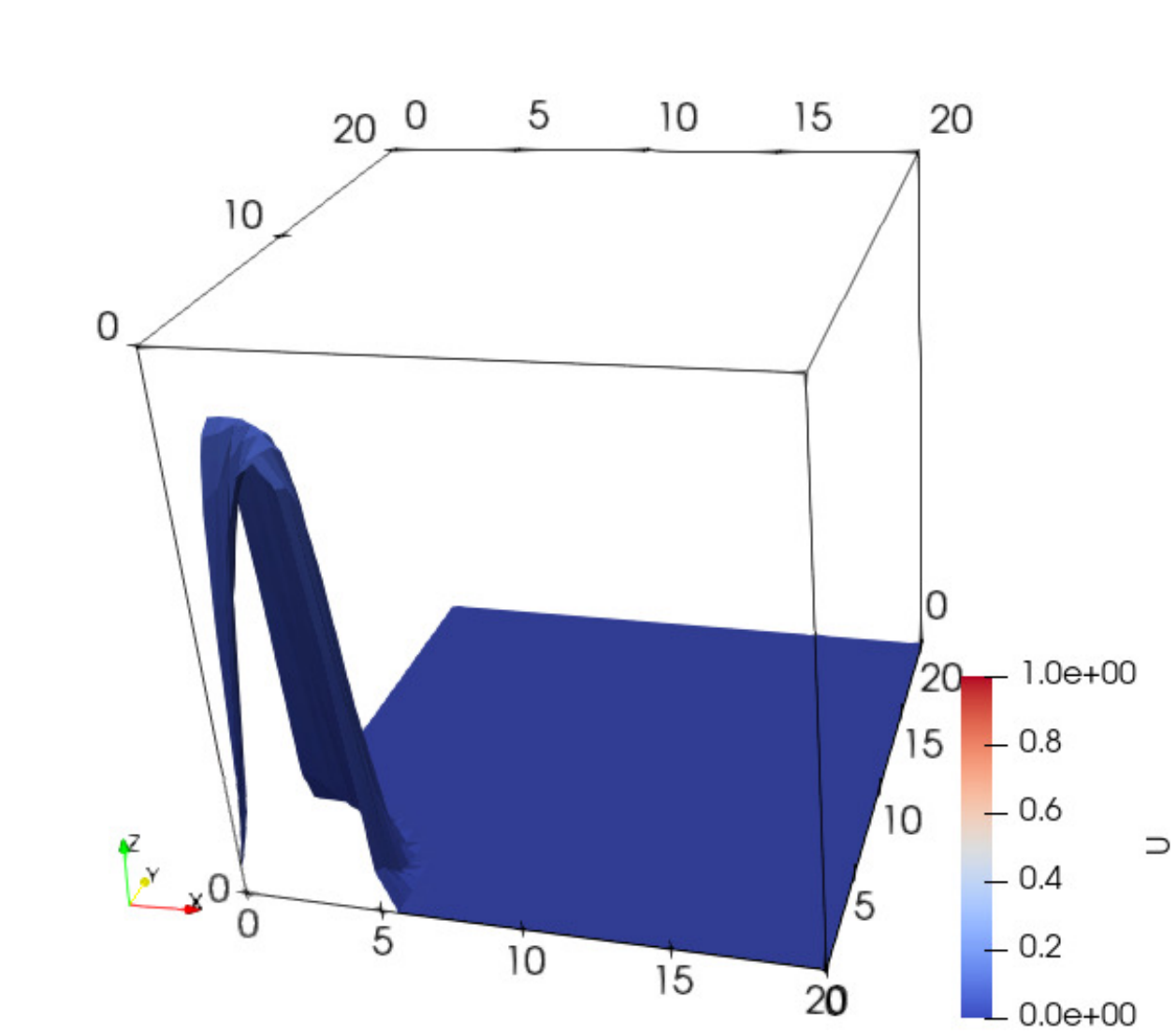}
		\caption{t = 10}
	\end{subfigure}
	\begin{subfigure}{0.24\textwidth}
		\includegraphics[width=\textwidth]{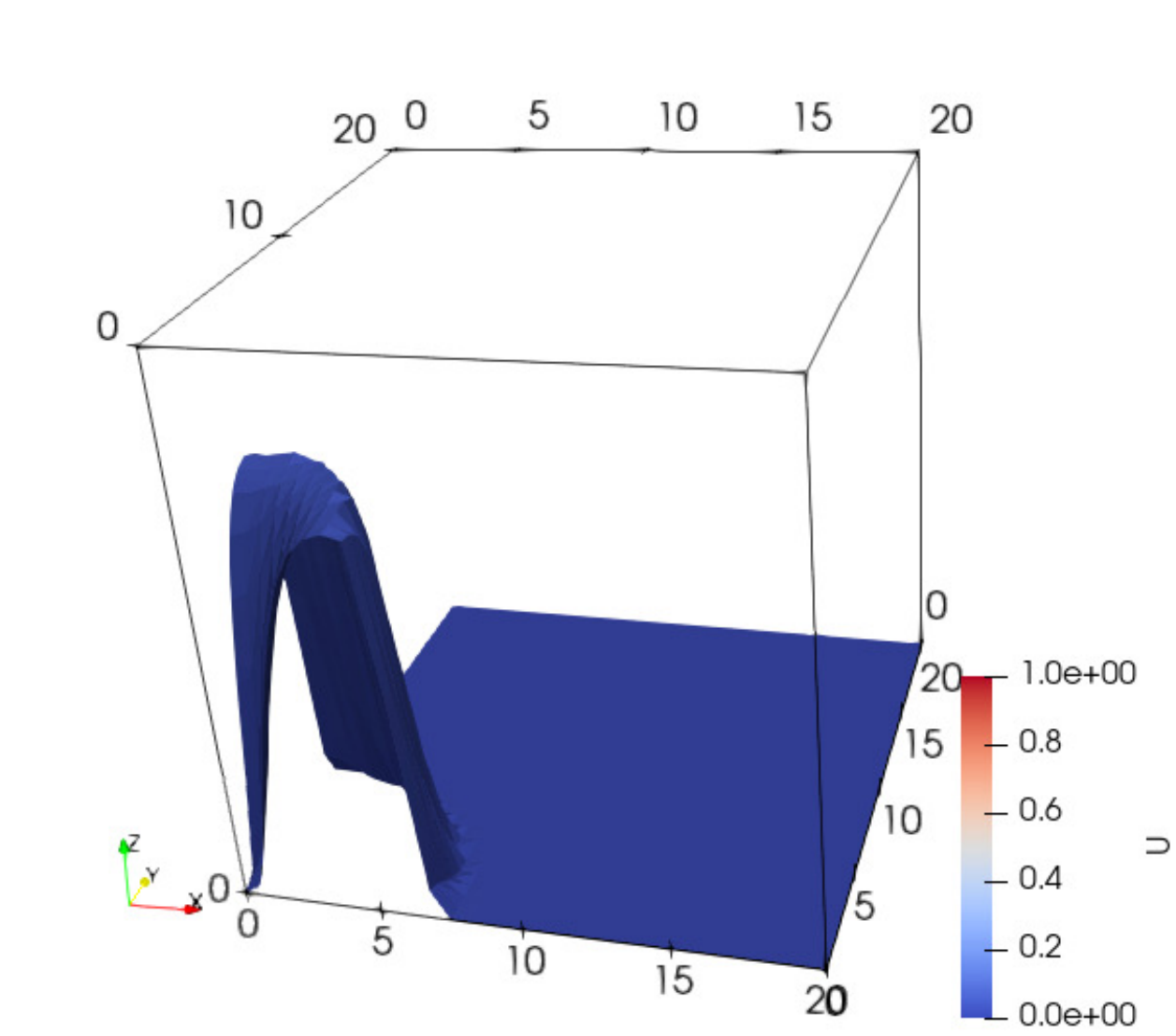}
		\caption{t = 20}
	\end{subfigure}
	\begin{subfigure}{0.24\textwidth}
		\includegraphics[width=\textwidth]{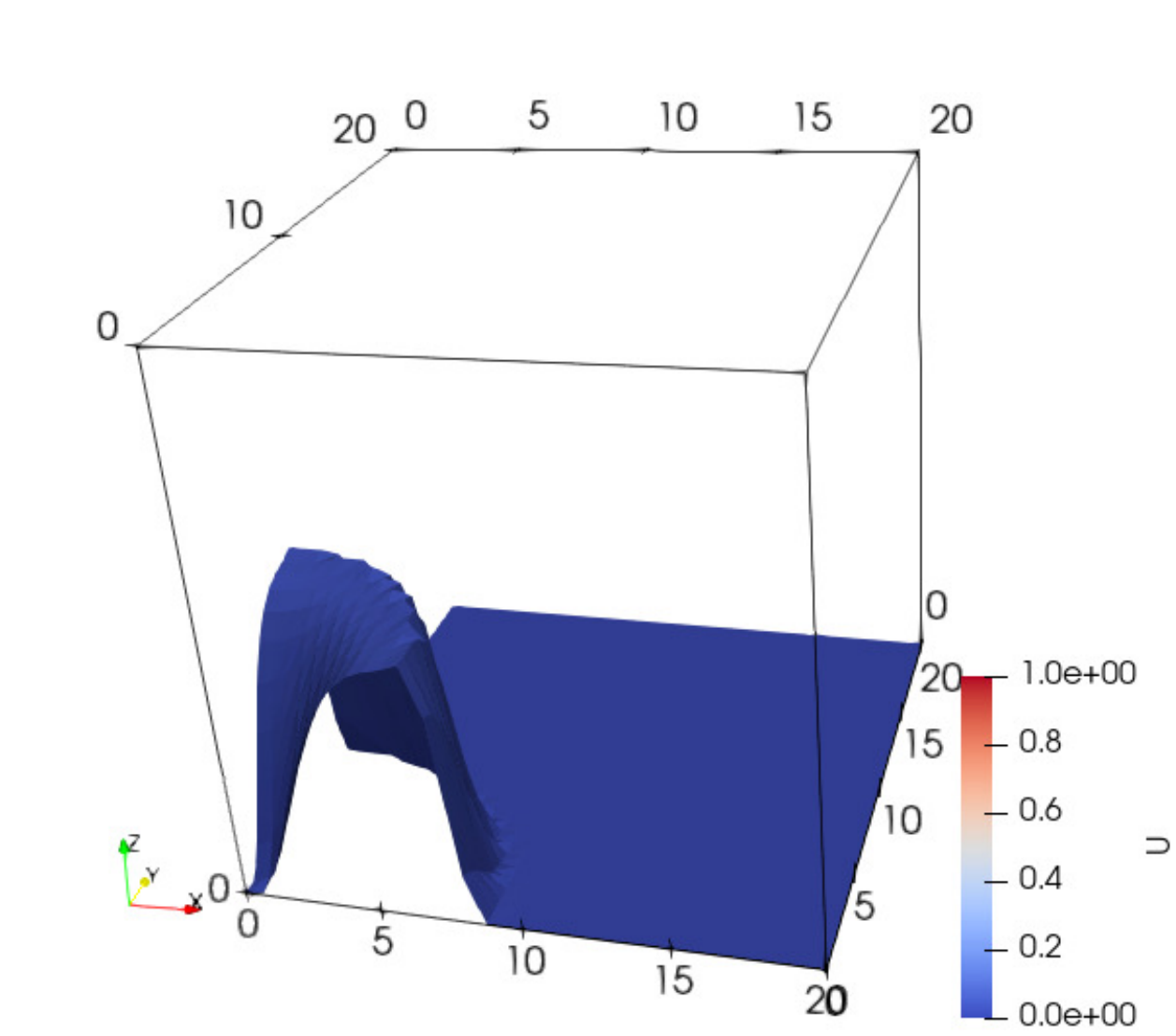}
		\caption{t = 30}
	\end{subfigure}
	\begin{subfigure}{0.24\textwidth}
		\includegraphics[width=\textwidth]{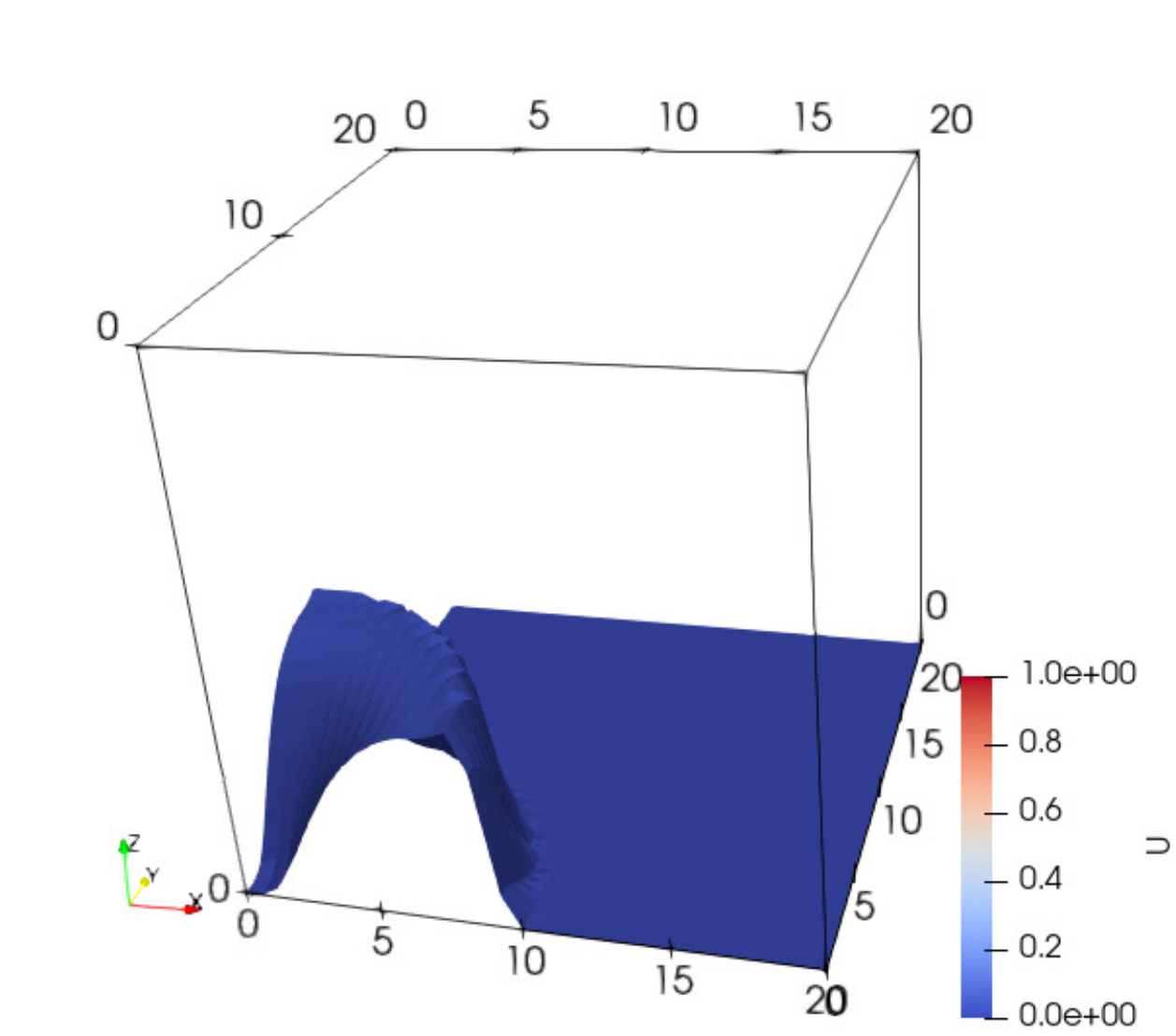}
		\caption{t = 40}
	\end{subfigure}
	\caption{
	The effect of the haptotactic rate on the cancer cell invasion $u$ at
different time instants $t=10, 20, 30, 40$, computed using the FEM-FCT scheme
with $\theta = 1$ for $\mu = 0.0001$ and $\chi=1$.
	 }
	\label{fig16}
\end{figure} 

\section{Conclusions}
\label{sec:conclusions}

In this paper, we proposed a fully discrete nonlinear high-resolution 
positivity preserving FEM-FCT scheme for chemotaxis equations without 
self-diffusion term describing a model of cancer invasion. We proved the 
solvability and positivity preservation of both the nonlinear discrete problem 
and the linear problems appearing in fixed-point iterations. A series of 
numerical experiments are shown to verify the robustness of the proposed 
method. Derivation of error estimates is left to future work.

\section*{Acknowledgments}
This work was initiated during a research stay of the first author at the
Institute of Applied Mathematics at the Leibniz University Hanover from
November 2021 to April 2022 for which hospitality is still gratefully
acknowledged. The work of Shahin Heydari was further supported through the
grant No.~396921 of the Charles University Grant Agency. The work of Petr
Knobloch was supported through the grant No.~22-01591S of the Czech Science
Foundation.

\bibliographystyle{plain}
\bibliography{cancer.bib}

\begin{thebibliography}{10}

\bibitem{aida2006lower}
Masashi Aida, Tohru Tsujikawa, Messoud Efendiev, Atsushi Yagi, and Masayasu
  Mimura.
\newblock Lower estimate of the attractor dimension for a chemotaxis growth
  system.
\newblock {\em J. London Math. Soc. (2)}, 74(2):453--474, 2006.

\bibitem{aida2004target}
Masashi Aida and Atsushi Yagi.
\newblock Target pattern solutions for chemotaxis-growth system.
\newblock {\em Sci. Math. Jpn.}, 59(3):577--590, 2004.

\bibitem{anderson2000mathematical}
Alexander~R.A. Anderson, Mark~A.J. Chaplain, E.~Luke Newman, Robert~J.C.
  Steele, and Alastair~M. Thompson.
\newblock Mathematical modelling of tumour invasion and metastasis.
\newblock {\em Computational and mathematical methods in medicine},
  2(2):129--154, 2000.

\bibitem{deal2020}
Daniel Arndt, Wolfgang Bangerth, Denis Davydov, Timo Heister, Luca Heltai,
  Martin Kronbichler, Matthias Maier, Jean-Paul Pelteret, Bruno Turcksin, and
  David Wells.
\newblock The {DEAL.II} finite element library: {D}esign, features, and
  insights.
\newblock {\em Comput. Math. Appl.}, 81:407--422, 2021.

\bibitem{dealII94}
Daniel Arndt, Wolfgang Bangerth, Marco Feder, Marc Fehling, Rene
  Gassm{\"o}ller, Timo Heister, Luca Heltai, Martin Kronbichler, Matthias
  Maier, Peter Munch, Jean-Paul Pelteret, Simon Sticko, Bruno Turcksin, and
  David Wells.
\newblock The \texttt{deal.II} library, {V}ersion 9.4.
\newblock {\em J. Numer. Math.}, 30(3):231--246, 2022.

\bibitem{BJK23}
Gabriel~R. Barrenechea, Volker John, and Petr Knobloch.
\newblock Finite element methods respecting the discrete maximum principle for
  convection-diffusion equations.
\newblock {\em SIAM Rev., accepted for publication}, 2023.

\bibitem{book1975flux}
D.L. Book, J.P. Boris, and K.~Hain.
\newblock Flux-corrected transport {II}: Generalizations of the method.
\newblock {\em J. Comput. Phys.}, 18(3):248--283, 1975.

\bibitem{BB73}
Jay~P. Boris and David~L. Book.
\newblock Flux-corrected transport. {I. SHASTA}, a fluid transport algorithm
  that works.
\newblock {\em J. Comput. Phys.}, 11(1):38--69, 1973.

\bibitem{boris1976flux}
J.P. Boris and D.L. Book.
\newblock Flux-corrected transport. {III.} minimal-error {FCT} algorithms.
\newblock {\em J. Comput. Phys.}, 20(4):397--431, 1976.

\bibitem{calvez2012blow}
Vincent Calvez, Lucilla Corrias, and Mohamed~Abderrahman Ebde.
\newblock Blow-up, concentration phenomenon and global existence for the
  {K}eller--{S}egel model in high dimension.
\newblock {\em Comm. Partial Differential Equations}, 37(4):561--584, 2012.

\bibitem{chaplain2005mathematical}
M.~A.~J. Chaplain and G.~Lolas.
\newblock Mathematical modelling of cancer cell invasion of tissue: the role of
  the urokinase plasminogen activation system.
\newblock {\em Math. Models Methods Appl. Sci.}, 15(11):1685--1734, 2005.

\bibitem{chaplain2006mathematical}
M.~A.~J. Chaplain and G.~Lolas.
\newblock Mathematical modelling of cancer invasion of tissue: dynamic
  heterogeneity.
\newblock {\em Netw. Heterog. Media}, 1(3):399--439, 2006.

\bibitem{chaplain1993model}
Mark~A.J. Chaplain and Andrew~M. Stuart.
\newblock A model mechanism for the chemotactic response of endothelial cells
  to tumour angiogenesis factor.
\newblock {\em Mathematical Medicine and Biology: A Journal of the IMA},
  10(3):149--168, 1993.

\bibitem{chapwanya2014positivity}
Michael Chapwanya, Jean M.-S. Lubuma, and Ronald~E. Mickens.
\newblock Positivity-preserving nonstandard finite difference schemes for
  cross-diffusion equations in biosciences.
\newblock {\em Comput. Math. Appl.}, 68(9):1071--1082, 2014.

\bibitem{Ciarlet}
P.~G. Ciarlet.
\newblock {\em The finite element method for elliptic problems}.
\newblock North-Holland, Amsterdam, 1978.

\bibitem{corrias2004global}
L.~Corrias, B.~Perthame, and H.~Zaag.
\newblock Global solutions of some chemotaxis and angiogenesis systems in high
  space dimensions.
\newblock {\em Milan J. Math.}, 72:1--28, 2004.

\bibitem{Dav04}
Timothy~A. Davis.
\newblock Algorithm 832: {UMFPACK} {V}4.3---an unsymmetric-pattern multifrontal
  method.
\newblock {\em ACM Trans. Math. Software}, 30(2):196--199, 2004.

\bibitem{epshteyn2009new}
Yekaterina Epshteyn and Alexander Kurganov.
\newblock New interior penalty discontinuous {G}alerkin methods for the
  {K}eller--{S}egel chemotaxis model.
\newblock {\em SIAM J. Numer. Anal.}, 47(1):386--408, 2008/09.

\bibitem{FeNeuNaWi19}
Dianlei Feng, Insa Neuweiler, Udo Nackenhorst, and Thomas Wick.
\newblock A time-space flux-corrected transport finite element formulation for
  solving multi-dimensional advection-diffusion-reaction equations.
\newblock {\em J. Comput. Phys.}, 396:31--53, 2019.

\bibitem{filbet2006finite}
Francis Filbet.
\newblock A finite volume scheme for the {P}atlak--{K}eller--{S}egel chemotaxis
  model.
\newblock {\em Numer. Math.}, 104(4):457--488, 2006.

\bibitem{fuest2022global}
Mario Fuest, Shahin Heydari, Petr Knobloch, Johannes Lankeit, and Thomas Wick.
\newblock Global existence of classical solutions and numerical simulations of
  a cancer invasion model.
\newblock {\em ESAIM Math.~Model.~Numer.~Anal.}, 57(4):1893--1919, 2023.

\bibitem{horstmann2011uniqueness}
Dirk Horstmann and Marcello Lucia.
\newblock Uniqueness and symmetry of equilibria in a chemotaxis model.
\newblock {\em J. Reine Angew. Math.}, 654:83--124, 2011.

\bibitem{horstmann2005boundedness}
Dirk Horstmann and Michael Winkler.
\newblock Boundedness vs. blow-up in a chemotaxis system.
\newblock {\em J. Differential Equations}, 215(1):52--107, 2005.

\bibitem{huang2021fully}
Xueling Huang, Xinlong Feng, Xufeng Xiao, and Kun Wang.
\newblock Fully decoupled, linear and positivity-preserving scheme for the
  chemotaxis--{S}tokes equations.
\newblock {\em Comput. Methods Appl. Mech. Engrg.}, 383:Paper No. 113909, 19,
  2021.

\bibitem{huang2020efficient}
Xueling Huang, Xufeng Xiao, Jianping Zhao, and Xinlong Feng.
\newblock An efficient operator-splitting {FEM-FCT} algorithm for {3D}
  chemotaxis models.
\newblock {\em Engineering with Computers}, 36(4):1393--1404, 2020.

\bibitem{john2021existence}
Volker John and Petr Knobloch.
\newblock Existence of solutions of a finite element flux-corrected-transport
  scheme.
\newblock {\em Appl. Math. Lett.}, 115:Paper No. 106932, 6, 2021.

\bibitem{john2021solvability}
Volker John, Petr Knobloch, and Paul Korsmeier.
\newblock On the solvability of the nonlinear problems in an algebraically
  stabilized finite element method for evolutionary transport-dominated
  equations.
\newblock {\em Math. Comp.}, 90(328):595--611, 2021.

\bibitem{keller1970initiation}
Evelyn~F. Keller and Lee~A. Segel.
\newblock Initiation of slime mold aggregation viewed as an instability.
\newblock {\em Journal of theoretical biology}, 26(3):399--415, 1970.

\bibitem{keller1971model}
Evelyn~F. Keller and Lee~A. Segel.
\newblock Model for chemotaxis.
\newblock {\em Journal of theoretical biology}, 30(2):225--234, 1971.

\bibitem{khalsaraei2016positivity}
M.~Mehdizadeh Khalsaraei, Sh. Heydari, and L.~Davari Algoo.
\newblock Positivity preserving nonstandard finite difference schemes applied
  to cancer growth model.
\newblock {\em J. Cancer Treat. Res.}, 4(4):27--33, 2016.

\bibitem{kolev2022unconditional}
Mikhail~K. Kolev, Miglena~N. Koleva, and Lubin~G. Vulkov.
\newblock An unconditional positivity-preserving difference scheme for models
  of cancer migration and invasion.
\newblock {\em Mathematics}, 10(1):131, 2022.

\bibitem{KT02}
D.~Kuzmin and S.~Turek.
\newblock Flux correction tools for finite elements.
\newblock {\em J. Comput. Phys.}, 175(2):525--558, 2002.

\bibitem{Kuz09}
Dmitri Kuzmin.
\newblock Explicit and implicit {FEM}-{FCT} algorithms with flux linearization.
\newblock {\em J. Comput. Phys.}, 228(7):2517--2534, 2009.

\bibitem{Kuz12a}
Dmitri Kuzmin.
\newblock Algebraic flux correction {I}. {S}calar conservation laws.
\newblock In Dmitri Kuzmin, Rainald L\"ohner, and Stefan Turek, editors, {\em
  Flux-corrected transport. Principles, algorithms, and applications}, pages
  145--192. Springer, Dordrecht, second edition, 2012.

\bibitem{li2017local}
Xingjie~Helen Li, Chi-Wang Shu, and Yang Yang.
\newblock Local discontinuous {G}alerkin method for the {K}eller--{S}egel
  chemotaxis model.
\newblock {\em J. Sci. Comput.}, 73(2-3):943--967, 2017.

\bibitem{lohner1987finite}
Rainald L{\"o}hner, Ken Morgan, Jaime Peraire, and Mehdi Vahdati.
\newblock Finite element flux-corrected transport ({FEM--FCT}) for the {E}uler
  and {N}avier--{S}tokes equations.
\newblock {\em Int.~J.~Numer.~Methods Fluids}, 7(10):1093--1109, 1987.

\bibitem{marchant_norbury_perumpanani}
B.~P. Marchant, J.~Norbury, and A.~J. Perumpanani.
\newblock Travelling shock waves arising in a model of malignant invasion.
\newblock {\em SIAM J. Appl. Math.}, 60(2):463--476, 2000.

\bibitem{marchant_norbury_sherratt}
B.~P. Marchant, J.~Norbury, and J.~A. Sherratt.
\newblock Travelling wave solutions to a haptotaxis-dominated model of
  malignant invasion.
\newblock {\em Nonlinearity}, 14(6):1653--1671, 2001.

\bibitem{mimura1996aggregating}
Masayasu Mimura and Tohru Tsujikawa.
\newblock Aggregating pattern dynamics in a chemotaxis model including growth.
\newblock {\em Physica A: Statistical Mechanics and its Applications},
  230(3-4):499--543, 1996.

\bibitem{nanjundiah1973chemotaxis}
Vidyanand Nanjundiah.
\newblock Chemotaxis, signal relaying and aggregation morphology.
\newblock {\em Journal of Theoretical Biology}, 42(1):63--105, 1973.

\bibitem{perumpanani1999two}
Abbey~J. Perumpanani, Jonathan~A. Sherratt, John Norbury, and Helen~M. Byrne.
\newblock A two parameter family of travelling waves with a singular barrier
  arising from the modelling of extracellular matrix mediated cellular
  invasion.
\newblock {\em Phys.~D}, 126(3-4):145--159, 1999.

\bibitem{ropp2009stability}
David~L. Ropp and John~N. Shadid.
\newblock Stability of operator splitting methods for systems with indefinite
  operators: advection-diffusion-reaction systems.
\newblock {\em J. Comput. Phys.}, 228(9):3508--3516, 2009.

\bibitem{saito2007conservative}
Norikazu Saito.
\newblock Conservative upwind finite-element method for a simplified
  {K}eller--{S}egel system modelling chemotaxis.
\newblock {\em IMA J. Numer. Anal.}, 27(2):332--365, 2007.

\bibitem{sokolov2015afc}
Andriy Sokolov, Ramzan Ali, and Stefan Turek.
\newblock An {AFC}-stabilized implicit finite element method for partial
  differential equations on evolving-in-time surfaces.
\newblock {\em J. Comput. Appl. Math.}, 289:101--115, 2015.

\bibitem{sokolov2013numerical}
Andriy Sokolov, Robert Strehl, and Stefan Turek.
\newblock Numerical simulation of chemotaxis models on stationary surfaces.
\newblock {\em Discrete Contin. Dyn. Syst. Ser. B}, 18(10):2689--2704, 2013.

\bibitem{strehl2010flux}
R.~Strehl, A.~Sokolov, D.~Kuzmin, and S.~Turek.
\newblock A flux-corrected finite element method for chemotaxis problems.
\newblock {\em Comput. Methods Appl. Math.}, 10(2):219--232, 2010.

\bibitem{strehl2013positivity}
Robert Strehl, Andriy Sokolov, Dmitri Kuzmin, Dirk Horstmann, and Stefan Turek.
\newblock A positivity-preserving finite element method for chemotaxis problems
  in 3{D}.
\newblock {\em J. Comput. Appl. Math.}, 239:290--303, 2013.

\bibitem{sulman2019positivity}
M.~Sulman and T.~Nguyen.
\newblock A positivity preserving moving mesh finite element method for the
  {K}eller--{S}egel chemotaxis model.
\newblock {\em J. Sci. Comput.}, 80(1):649--666, 2019.

\bibitem{tao2009combined}
Youshan Tao and Mingjun Wang.
\newblock A combined chemotaxis-haptotaxis system: the role of logistic source.
\newblock {\em SIAM J. Math. Anal.}, 41(4):1533--1558, 2009.

\bibitem{Temam77}
Roger Temam.
\newblock {\em {N}avier-{S}tokes equations. {T}heory and numerical analysis}.
\newblock North-Holland, Amsterdam, 1977.

\bibitem{tyson1999minimal}
R.~Tyson, S.R. Lubkin, and James~D. Murray.
\newblock A minimal mechanism for bacterial pattern formation.
\newblock {\em Proceedings of the Royal Society of London. Series B: Biological
  Sciences}, 266(1416):299--304, 1999.

\bibitem{tyson2000fractional}
Rebecca Tyson, L.~G. Stern, and Randall~J. LeVeque.
\newblock Fractional step methods applied to a chemotaxis model.
\newblock {\em J. Math. Biol.}, 41(5):455--475, 2000.

\bibitem{Var00}
Richard~S. Varga.
\newblock {\em Matrix iterative analysis}.
\newblock Springer-Verlag, Berlin, 2000.

\bibitem{wu2005signaling}
Dianqing Wu.
\newblock Signaling mechanisms for regulation of chemotaxis.
\newblock {\em Cell research}, 15(1):52--56, 2005.

\bibitem{Zal79}
Steven~T. Zalesak.
\newblock Fully multidimensional flux-corrected transport algorithms for
  fluids.
\newblock {\em J. Comput. Phys.}, 31(3):335--362, 1979.

\bibitem{zhang2016characteristic}
Jiansong Zhang, Jiang Zhu, and Rongpei Zhang.
\newblock Characteristic splitting mixed finite element analysis of
  {K}eller--{S}egel chemotaxis models.
\newblock {\em Appl. Math. Comput.}, 278:33--44, 2016.

\bibitem{zhao2020petrov}
Shubo Zhao, Xufeng Xiao, Jianping Zhao, and Xinlong Feng.
\newblock A {P}etrov--{G}alerkin finite element method for simulating
  chemotaxis models on stationary surfaces.
\newblock {\em Comput. Math. Appl.}, 79(11):3189--3205, 2020.

\end{thebibliography}
\end{document}